\documentclass[12pt,a4paper]{article}
\usepackage{etex}
\usepackage[utf8]{inputenc}
\usepackage[TS1,T1]{fontenc}
\usepackage{xspace}
\usepackage[a4paper,tmargin=3truecm,bmargin=3truecm,rmargin=2.2truecm,lmargin=2.2truecm]{geometry}

\usepackage[english]{babel}
\usepackage{hyperref}
\usepackage{lmodern}

\usepackage{enumitem}
\newlist{enum-hypothesis}{enumerate}{1}
\setlist[enum-hypothesis]{label=(\arabic*),itemsep=0pt, parsep=0pt}

\usepackage{amsfonts,amssymb,amsmath,mathtools,slashed}
\usepackage[thmmarks,amsmath]{ntheorem}

\usepackage[all]{xy}

\newtheorem{theorem}{Theorem}[section]
\newtheorem{proposition}[theorem]{Proposition}
\newtheorem{lemma}[theorem]{Lemma}
\newtheorem{corollary}[theorem]{Corollary}

\newtheorem{definition}[theorem]{Definition}

\theoremsymbol{$\square$}
\theoremstyle{plain}
\theorembodyfont{\upshape}
\newtheorem{remark}[theorem]{Remark}

\theoremsymbol{}
\theoremstyle{break}
\theorembodyfont{\slshape}
\newtheorem{hypothesis}[theorem]{Hypothesis}
\newtheorem{case}[theorem]{Case}

\theoremheaderfont{\scshape}
\theorembodyfont{\upshape} 
\theoremstyle{nonumberplain}
\theoremseparator{} 
\theoremsymbol{\rule{1.3ex}{1.3ex}} 
\newtheorem{proof}{Proof}

\usepackage[square,numbers,sort,compress,merge]{natbib}
\hypersetup{
plainpages=false,
colorlinks=true,
linkcolor=black, 
anchorcolor=black, 
citecolor=black, 
urlcolor=black, 
menucolor=black, 
filecolor=black, 
bookmarksopen=true,
bookmarksnumbered=true}

\numberwithin{equation}{section}

\hypersetup{
pdfauthor={Thierry Masson, Bruno Iochum},
pdftitle={???},
pdfsubject={},
pdfcreator={pdflatex},
pdfproducer={pdflatex},
pdfkeywords={}
}

\newcommand\gR{{\mathbb R}}
\newcommand\gC{{\mathbb C}}
\newcommand\gN{{\mathbb N}}
\newcommand\gS{{\mathbb S}}
\newcommand\gZ{{\mathbb Z}}

\newcommand\hgR{\widehat{\gR}}
\newcommand\hD{\widehat{D}}
\newcommand\hH{\widehat{H}}
\newcommand\hJ{\widehat{J}}

\newcommand\hrho{\widehat{\rho}}
\newcommand\hpi{\widehat{\pi}}
\newcommand\hxi{\widehat{\xi}}
\newcommand\hY{\widehat{Y}}

\newcommand\hU{\widehat{U}}
\newcommand\DD{\mathcal{D}}
\newcommand\Dg{\slashed{D}_g}

\newcommand\hvarphi{{\widehat{\varphi}}}
\newcommand\hsigma{{\widehat{\sigma}}}
\newcommand\htheta{{\widehat{\theta}}}
\newcommand\hb{{\hat{b}}}

\newcommand\hdelta{\widehat{\delta}}
\newcommand\hf{{\hat{f}}}
\newcommand\hg{{\hat{g}}}

\newcommand\tpi{{\tilde{\pi}}}

\newcommand\tphi{{\tilde{\phi}}}
\newcommand\talpha{{\widetilde{\alpha}}}
\newcommand\tLambda{{\widetilde{\Lambda}_\varphi}}
\newcommand\ttLambda{{\widetilde{\Lambda}_\tvarphi}}
\newcommand\tvarphi{{\widetilde{\varphi}}}

\newcommand{\cone}{\eta}
\newcommand{\ct}{\omega}

\newcommand\caB{{\mathcal B}}
\newcommand\caH{\mathcal{H}}
\newcommand\caF{\mathcal{F}}
\newcommand\caJ{\mathcal{J}}
\newcommand\caK{\mathcal{K}}

\newcommand\caS{\mathcal{S}}
\newcommand\caT{\mathcal{T}}

\newcommand\caM{\mathcal{M}}
\newcommand\caN{\mathcal{N}}
\newcommand\caV{\mathcal{V}}

\newcommand\caY{\mathcal{Y}}

\newcommand\dd{\text{\textup{d}}} 
\newcommand{\vc}{\vcentcolon =} 
\newcommand{\op}{\text{op}}

\newcommand{\bbbone}{{\text{\usefont{U}{dsss}{m}{n}\char49}}}
\newcommand{\red}{\text{red}}

\newcommand{\GNS}{{\text{GNS}}}
\newcommand{\vol}{{\text{vol}}}

\DeclareMathOperator{\Op}{Op}
\DeclareMathOperator{\Dom}{Dom}
\DeclareMathOperator{\Supp}{Supp}
\DeclareMathOperator{\Span}{Span}
\DeclareMathOperator{\Id}{Id}
\DeclareMathOperator{\Tr}{Tr}	 
\DeclareMathOperator{\Ran}{Ran}	 

\newcommand\Aut{\text{{Aut}}}

\newcommand\algA{\mathcal{B}}
\newcommand\algB{\mathcal{B}}

\newcommand\algM{\mathcal{M}}

\newcommand{\bA}{\mathbf{A}} 
\newcommand{\bB}{\mathbf{B}} 

\newcommand{\onehalf}{\frac{1}{2}}

\newcommand{\HS}{{\text{H.S.}}}

\DeclarePairedDelimiter\abs{\lvert}{\rvert}
\DeclarePairedDelimiter\norm{\lVert}{\rVert}

\newcounter{mnotecount}[section]
\renewcommand{\themnotecount}{\thesection.\arabic{mnotecount}}
\newcommand{\mnote}[1]
{\protect{\stepcounter{mnotecount}}$^{\mbox{\footnotesize
$
\bullet$\themnotecount}}$ \marginpar{
\raggedright\tiny\em
$\!\!\!\!\!\!\,\bullet$\themnotecount: #1} }

\begin{document}

\title{Crossed product extensions of spectral triples}
\author{B. Iochum, T. Masson}
\date{}
\maketitle
\begin{center}
Centre de Physique Théorique \\
{\small Aix-Marseille Université \& Université de Toulon \&  CNRS  (UMR 7332)}\\
{\small  Luminy, Case 907, 13288 Marseille Cedex 9, France}
\end{center}

\vspace{3cm}

\begin{abstract}
Given a spectral triple $(A,H,D)$ and a $C^*$-dynamical system $(\bA, G, \alpha)$ where $A$ is dense in $\bA$ and $G$ is a locally compact group, we extend the triple to a triplet $(\algB,\caH,\DD)$ on the crossed product $G \ltimes_{\alpha, \red} \bA$ which can be promoted to a modular-type twisted spectral triple within a general procedure exemplified by two cases: the $C^*$-algebra of the affine group and the conformal group acting on a complete Riemannian spin manifold.
\end{abstract}

\newpage

\tableofcontents

\section{Introduction}

Let us first describe what is the goal of this work: given a $C^*$-dynamical system $(\bA, G, \alpha)$ where $G$ is a locally compact group and a spectral triple $(A,H,D)$ where $A$ is a dense $^*$-subalgebra of $\bA$, we extend this triple to a new one $(\algB,\caH,\DD)$ on the crossed product $\bB=G \ltimes_{\alpha, \red} \bA$ within a general procedure of which we give a significant variant later. This construction is based on a covariant representation of the dynamical system and an hypothesis on the response of $D$ to this covariance by a conformal-like deformation where the constraints on the conformal factor are essentially driven by one-cocycles in the multiplier algebra of $\bA$, see Hypothesis \ref{hyp-dirac}. The choice of a dense $^*$-algebra $\algB$ of $\bB$ is studied and different possibilities are offered, like the requirement for $\algB$ to be a Fréchet algebra in order to preserve $K$-theory. An explicit operator $\DD$ is given on the Hilbert space $\caH= L^2(G) \otimes H\otimes \gC^2$ where the last $\gC^2$ is added to create room for an action of Pauli matrices. What we get is in fact a modular-type spectral triple where the axiom of compact resolvent is modified, see \eqref{modularcompactresol}.

In a variant version, we assume the existence of a KMS-weight $\varphi$ on $\bA$ such that $H$ is just the associated GNS Hilbert space. As a compatibility between the dynamical system and the spectral triple, we assume a relative invariance of $\varphi$, namely that $\varphi \circ \alpha=\varphi(q \,\cdot\, q^*)$ (see Hypothesis \ref{hypo-weight}). Here $q$ only needs to be a $\alpha$-one-cocycle valued in the multiplier algebra of $\bA$. Via the unitary representation of $G$ which automatically implements the action $\alpha$ on the von Neumann algebra $\pi_\varphi(\bA)''$, the hypothesis on $D$ (see \eqref{eq-comp-D-alpha} below) allows the desired extension which is compatible with modular theory of both the dynamical system and the original triple.

We also consider the extension of a real operator $J$ on the original triple $(A,H,D)$, see Proposition \ref{prop-hJ-operator}.

All previous hypotheses stem from two types of examples: the first ones are the $C^*$-algebras of semidirect group products, like the affine group involved in the $\kappa$-Minkowski  deformation, whose symmetry group is a Hopf algebraic deformation of the universal enveloping algebra of the Poincaré Lie algebra (see for instance \cite{Ioc12} for historical references). The second ones are the conformal deformations of a manifold. 

 It is worth mentioning that in the extension process, the operator $\DD$ is the sum of an extension of the original operator $D$ and an operator (see \eqref{eq-DH}) which allows an increasing of the original spectral dimension. This object is natural in the $C^*(G)$ framework quoted above and known as the Duflo--Moore operator \cite{DuflMoor76a,Ioc12}. The natural equivalent context in noncommutative geometry are the modular spectral triples, see for instance \cite{Kaa11,CNNR,Mata13,Mata14b,Ks,CPR,CRT,RS14}, but seen in the non-unital context.

The extension problem has already been considered in the literature: we were mainly inspired here by the Connes and Moscovici approaches \cite{ConnesMoscovici2006, Moscovici2010}. However, we differ by our choice of Hilbert space which is justified by the possibility to increase the spectral dimension as already quoted. In \cite{GabrielGrensing}, a construction of spectral triples for a certain class of crossed product-like algebras is proposed. This construction is proved to give a concrete unbounded representative of the Kasparov product of the original spectral triple and the Pimsner--Toeplitz extension associated to the crossed product by a Hilbert module. 
In \cite{RennieRobertsonSims}, a related context of a Kasparov module representing the extension class defining the Cuntz–Pimsner algebra is developed.

Our construction works for any locally compact group $G$, so covers dynamical systems, while in \cite{BMR,Pat,HSWZ}, $G$ equals $\gZ$ or is discrete. However, there, the proposed operator $\DD$ acts on the same Hilbert space as here (also via Pauli matrices) and build on the original $D$ and a differential operator on the circle (via Fourier transform) which commute (see \cite[Equation (6)]{BMR}), while our chosen $\DD$ is strongly dependent of the group and built with non-commutative pieces (namely $\hD$ and $\caT_{\cone,\,\ct}$ which do not commute in general). 
The equicontinuity of the action plays an important role in these works and we outline the fact that for the conformal deformation of a manifold, this ``inessential'' property of the group (see Section \ref{conformal}) is translated into triviality of the first cohomology group of the one-cocycles used for the construction of $\DD$, see \eqref{eq-comp-D-alpha}. 

A recent related work is \cite{Farzad-Gabriel}, where the conformal invariance of the Euler characteristic for a $C^*$-dynamical system is interpreted as a Chern--Gauss--Bonnet theorem.

This paper is organized as follows: in Section \ref{sec-motivation-main-results} we expose in details our motivations and give the main result with its proof. Section \ref{About T and Theta} is devoted to the selfadjointness of the constructed $\DD$ and modular operator $\Theta$. Then the role of the cocycles valued in the multiplier algebra is outlined and different choices for the smooth algebra $\algB$ in the crossed product are offered, especially in the Fréchet context. In Section \ref{relatively invariant weight}, we investigate the case where the covariant dynamical system is given by a weight $\varphi$ on $\bA$. This weight has an extension to $\pi_\varphi(\bA)''$ which is used with its dual weight to recover for instance the real structure. The case of interest of an original operator $D$ associated to a derivation of $A$ is also considered. Section \ref{sec-C*-algebra-semidirect-product} is devoted to the special case where $\bA$ is the $C^*$-algebra of a group, with a peculiar attention to Plancherel weight and its dual. The examples of affine and conformal groups are finally studied in Section \ref{sec-examples}.

\section{The extended modular-type twisted spectral triple}

\subsection{The motivations and main result}
\label{sec-motivation-main-results}

We now enter into the details of the construction giving the motivations through few examples.

By a  $C^*$-dynamical system $(\bA, G, \alpha)$ we mean that  $\bA$ is a $C^*$-algebra,  $G$ is a second-countable locally compact Hausdorff group, with Haar measure $\mu_G$ and modular function $\Delta_G$, and $\alpha : G \to \Aut(\bA)$ is a morphism from $G$ to the $^*$-automorphisms of $\bA$ such that, for any $a\in \bA$, $r\in G \mapsto \alpha_r(a)\in \bA$ is a continuous map.

By a spectral triple $(A, H, D)$ we mean that $A$ is a (non necessarily unital) $^*$-algebra with a faithful nondegenerate representation $\rho$ on a separable Hilbert space $H$ and $D$ is a selfadjoint operator on $H$ such that $[D,\rho(a)]$ are densely defined and bounded and 
\begin{align}
\label{compactresol}
\rho(a) (\bbbone+D^2)^{-1/2}\text{ is a compact operator for any }a\in A.
\end{align}

Let us now consider a spectral triple $(A, H, D)$, hereafter called the ``original'' one, where $A$ is a dense $*$-subalgebra of the $C^*$-algebra $\bA$ of the given dynamical system $(\bA, G, \alpha)$. The purpose of the following construction is to define an extension of the original  triple into a (modular type) spectral triple defined for the $C^*$-crossed product  $\bB\vc G \ltimes_{\alpha, \red} \bA$. This construction actually leads to a \emph{twisted} spectral triple $(\algB, \caH,\DD,\beta)$ where $\beta$ is an automorphism of a dense subalgebra $\algB$ of $\bB$.  Recall that in such a triple, $\beta$ is an automorphism of $\algB$ and $\pi$ is a representation of $\algB$ on $\caH$ such that the unitarity condition $\beta(a^*)=[\beta^{-1}(a)]^*$ holds true for $a\in \algB$, and the twisted commutator $[\DD,\pi(a)]_\beta\vc\DD\pi(a)-\pi[\beta(a)]\DD$ extends to a bounded operator for any $a\in \algB$, see \cite{ConnesMoscovici2006}.

A twisted spectral triple $(\algB, \caH,\DD,\beta)$ will be said of {\it modular-type} when the twist is  implemented by an unbounded positive operator $\Theta$ on $\caH$: on a dense domain, we have
\begin{equation*}
\pi[\beta(a)]=\Theta \,\pi(a)\, \Theta^{-1},\quad \text{for any }a\in \algB,
\end{equation*}
and the resolvent condition \eqref{compactresol} is replaced by a weaker one
\begin{align}
\label{modularcompactresol}
\Theta\, \pi(a) (\bbbone+\DD^2)^{-(1+\epsilon)/2}\text{ is a compact operator for any }a\in \algB \text{ and }\epsilon > 0.
\end{align}
The notion of modular spectral triple was considered in \cite{CRT, CNNR, RS14}. Essentially, this means that a trace $\tau(\cdot)$ on the algebra $\bB$ is swapped with the weight $\tau(\Theta\,\cdot)$ where $\Theta$ is a positif operator and the authors considered the general case of semifinite von Neumann algebras with an arbitrary trace. That operator can be the generator of an automorphism like in KMS theory. Moreover the automorphism can be used (or not) to twist the commutators. A tentative of a general definition of twisted modular spectral triple within the semifinite context is formalized in \cite{Kaa11} but only in the unital context. Since we want to deal with the non-unital case which is still waiting for a satisfactory definition, we only use the term ``modular type'' in the lazy above sense where we are in a peculiar situation with the existence of a twist implementing the modular object $\Theta$ and $\tau$ is nothing else than the usual trace. Moreover, to avoid the construction of the fixed point algebra under the modular group defined by $\Theta$ and the constraint that $\pi(a)(\bbbone+\DD^2)^{-1/2}$ are $\Tr(\Theta\,\cdot)$-compact operators (as required in Ibid.), we will use \eqref{modularcompactresol}; see however Remark~\ref{comparaison}. Of course, to assume that $\epsilon=0$ in \eqref{modularcompactresol} would seem more natural, but in one of our driving examples (see Case~\ref{Case-Affine} for the affine group), we are unable to check the compactness of $\Theta\, \pi(a) (\bbbone+\DD^2)^{-1/2}$, so we only assume \eqref{modularcompactresol}.

For the crossed product, the space $C_c(G, \bA)$, of compactly supported continuous functions $G \to \bA$, is a $*$-algebra for the associative product and the involution
\begin{align}
\label{eq-product-alpha}
(f \star_\alpha g)(r) &\vc \int_G \dd\mu_G(r') \, f(r') \,\alpha_{r'}[g(r'^{-1} r)],
\\
\label{eq-involution-alpha}
f^*(r) &\vc \Delta_G(r)^{-1} \, \alpha_{r}[ f(r^{-1})^* ], \quad\text{for $f ,g\in C_c(G, \bA)$},
\end{align}
where on the right hand side the involution is the one in $\bA$. We want to keep track of the representation of the original triple, but the vector space $C_c(G, A)$ is not a good candidate for a dense $*$-subalgebra in $\bB=G \ltimes_{\alpha, \red} \bA$, since it  is not an algebra a priori: the integral \eqref{eq-product-alpha} is defined by duality (it commutes with elements in the dual) and for any continuous function $h: G\to A$, the result of $\int_G \dd\mu_G (r) \,h(r)$ is in $\mu_G(\Supp h) \times \overline{\text{co}}(h(G))$  where $\overline{\text{co}}(K)$ is the closed convex hull of a compact space $K$), which is not necessary in $A$ (see \cite[3.27 Theorem]{Rudin} or \cite[Section 1.5]{Will07a}). If we insist that $C_c(G, A)$ becomes an algebra, we have to assume some extra conditions: the automorphisms $\alpha_{r}$ preserves $A$, the algebra $A$ is endowed with a completely metrizable locally convex  topology \cite[5.35 Theorem]{Ali}, and $\alpha$ is compatible with this topology in a certain sense, see Section~\ref{Remarque-Frechet} for details. In the special case of a transformation group $(G,X)$ where $X$ is a topological space and $A=C_c(X)$, then, a good candidate is $\algA=C_c(G\times X) \subset C_c(G,C_c(X))$ (remark that the inclusion can be strict, see \cite[Remark 2.32]{Will07a}).

In the following, we will only assume the existence of a ``good'' algebra $\algA$ inside $C_c(G, A)$: in the several examples studied in this paper, natural candidates for such algebras will be given, see Sections~\ref{sec-example-affine-group} and \ref{conformal}.

We need also to introduce some compatibility conditions between the given dynamical system and the spectral triple, with in mind the following precise examples which motivate our Hypothesis~\ref{hyp-dirac} below.

\begin{case} 
\label{Case-Affine} 
$\bA=C^*(N)$ is the $C^*$-algebra of a locally compact group $N$ on which $G$ acts, so that the extended spectral triple concerns the $C^*$-algebra $G \ltimes C^*(N) \simeq C^*(G \ltimes N)$. 
\end{case}
For instance, the affine group $\gR \ltimes \gR$, considered in  \cite{Mata14a,Ioc12}, enters in this case: $\gR$ acts on $\bA=C^*(\gR)$ and our construction is such that it produces the twisted spectral triple proposed in \cite{Mata14a}, see Section~\ref{sec-example-affine-group}. 

Since this case is also a transformation group, we besides consider
\begin{case}
\label{Case-Conformal}
Let $(M,g)$ be a smooth complete Riemannian spin manifold where $g$ is a given metric. Let $[g]$ be the class of metrics conformally equivalent to $g$. $M$ is acted upon by the Lie group $G=SCO(M, [g])$ of diffeomorphisms that preserve the orientation and the conformal and spin structures (or one of its closed subgroups). \\
The data are: $A=C^\infty(M)$ with pointwise product, $H=L^2(M,\slashed{S})$ where $\slashed{S}$ is the spinor bundle, $D$ is the Dirac operator $\Dg$ and the representation $\rho$ is just the pointwise multiplication on sections. The associated dynamical system is $(\bA= C_0(M), G,\alpha)$, where $\alpha_\phi(f)= f\circ\phi$ for $\phi\in G$, $f\in \bA$.
\end{case}
A related twisted spectral triple was investigated in \cite{Moscovici2010} for this situation, but here, we do not take only the discrete  crossed product of $G$ and $M$ as in \cite{Moscovici2010} or \cite[Section 2.3]{ConnesMoscovici2006}: we also change the Hilbert space representation, see Section \ref{conformal} for details.
\bigskip

Motivated by Case~\ref{Case-Conformal}, we will assume that $D$ is pointwise unitarily equivalent to a ``conformal'' deformation $c^*\,D\,c$, where $c$ is a map from $G$ to $\caB(H)$ (bounded operators on $H$) endowed with a cocycle property on $c\, c^*$.

Let us now introduce some notations: $M(\bA)$ is the multiplier algebra of $\bA$ endowed with the strict topology and $Z(M(\bA))$ is its center. We will consider the following subgroups $X$ of the group $M(\bA)^\times$ of invertible elements in $M(\bA)$:  $Z(A)^\times$ is the group of multipliers preserving $A$ and commuting with $A$ (this group is abelian, see Lemma \ref{group of cocycles}), and $UM(A)$  is the group of unitary elements of $M(\bA)$ which preserve $A$. The representation $\rho$ of $\bA$ extends to a representation, still denoted by $\rho$, of $M(\bA)$ \cite[II.7.3.9]{Blac06a}. There also exists an extended action of $G$ on $M(\bA)$, still denoted by $\alpha$, which is strictly continuous \cite[II.10.3.2]{Blac06a}.

Our construction relies on the following data that will be commented soon after.

\begin{hypothesis}
\label{hyp-dirac}
\begin{enum-hypothesis}
\item\label{hyp-spectral} $(A,H,D)$ is a spectral triple where $A$ is a dense $*$-subalgebra in a $C^*$-algebra $\bA$.

\item\label{hyp-dyn syst} $(\bA, G, \alpha)$ is a $C^*$-dynamical system, such that $\alpha_r(A) = A$ for any $r \in G$.

\item \label{hyp-algA} There exists an algebra $\algA \subset C_c(G, A)$ which is a dense $*$-subalgebra of $C_c(G, \bA)$ for the product \eqref{eq-product-alpha} and the involution \eqref{eq-involution-alpha}.

\item \label{hyp-U} There exists a faithful non-degenerate covariant representation $(\rho, U, H)$ of the $C^*$-dynamical system $(\bA, G, \alpha)$, namely  $r\in G \mapsto U_r \in \caB(H)$ is a strongly continuous unitary representation of $G$, and
\begin{align}
\label{eq-UaU*}
U_r\,\rho(a)\,U_r^*=\rho[\alpha_r(a)],\quad r\in G,\, a\in \bA.
\end{align}
We denote by $\hrho$ the integrated representation of $(\rho,U,H)$ of the crossed product algebra 
\begin{align*}
&\bB\vc G \ltimes_{\alpha, \red} \bA \text{ acting on }\\
&\hH \vc L^2(G,d\mu_g)\otimes H.
\end{align*}

\item \label{hyp-comp-D-alpha} There exists a map $z : G \mapsto \caB(H)$ such that\\
$U_r\,D\,U^*_r$ and ${z(r)^*}^{-1}\, D\, z(r)^{-1}$ have a dense common core for any $r \in G$ and on this core
\begin{align}
U_r \,D \,U_r^*= {z(r)^*}^{-1}\, D\, z(r)^{-1}, \label{eq-comp-D-alpha}
\end{align}
and there is a (positive) continuous $\alpha$-one-cocycle valued in $Z(A)^\times$, $r \mapsto p(r)$, such that 
\begin{align}
& \rho[p(r)]=z(r)\,z(r)^*, \nonumber\\
\label{eq-hyp-p(r)}
& r\in G \mapsto p(r)^{\pm 1}\,f(r) \in A \text{ is in $\algA$ for any $f \in \algA$.}
\end{align}

\item \label{preservdom}
Control of domains: the space
\begin{align}
\label{hY}
&\hY \vc \{ \hxi \in \hH \mid\hxi(r) \in \Dom(D) \,\dd\mu\text{-}a.e. \text{ and the map: }\nonumber\\
&\hspace{+4cm}r\in G \mapsto \norm{\big(D \pm i \ct\,\rho[p(r)]\big)\, \hxi(r)}_H \text{ is in }L^2(G) \}
\end{align}
is dense in $\hH$ for some non-zero coefficient $\ct\in \gR$. \\
Moreover, there exists $Y_{\algA} \subset \hY$ which is also a dense subspace of $\hH$ such that
\begin{align*}
\hrho(f)\,Y_\algA  \subset \hY \text{ for any } f\in \algA. 
\end{align*}

\item \label{hyp-comm-bounded}
For any $f \in \algA$, there exists a constant $M_{f,z}$ such that
\begin{equation}
\label{eq-hyp-comm-bounded}
\norm{[D, \rho(f(r))z(r) ]\, z(r)^{-1}}_{_{\caB(H)}} \leq M_{f,z}\, \chi_{S_f}(r), \quad \text{for any $r \in G$}
\end{equation}
where $\chi_{S_f}$ is the characteristic function of the compact support $S_f$ of $f$ in $G$.
\end{enum-hypothesis}
\end{hypothesis}

{\bf Comments:}

- The point \ref{hyp-spectral} and \ref{hyp-dyn syst} are the initial objects.

- Point \ref{hyp-algA} is necessary since, as already explained in the Introduction, $C_c(G,A)$ is not necessarily an algebra.

- Point \ref{hyp-U} required only that the action $\alpha$ is unitarily implemented.

- Equation \eqref{eq-comp-D-alpha} emphasizes the reaction of the orignal $D$ to previous unitary implementation. It does not completely define the map $z:G \to \caB(H)$: let $u : G\to U\caB(H)$ be such that $[D,u(r)]=0$, for all $r\in G$; then $z'(r)=z(r)\,u(r)$ is also a solution of \eqref{eq-comp-D-alpha} while $\rho[p'(r)]=z'(r)z'(r)^*$ remains equal to $\rho[p(r)]$.

- The domain of our future $\DD$ will be controlled by \eqref{hY} which is sufficient to obtain selfadjointness for $\DD$ (see Proposition \ref{D et Theta selfadjoints}). The coefficient $\ct$ is a relative weight between the operator $D$ and the cocycle $p$.

- The constraint \eqref{eq-hyp-comm-bounded} is a technical assumption to handle the boundedness of the twist commutator with our future $\DD$. Then the inclusion $\pi(\algA)\Dom \DD \subset \Dom \DD$, which is crucial for $\DD$ as emphasized in \cite{FMR14}, is controlled by the hypothesis \ref{preservdom} as seen in Proposition \ref{prop-twistcom}. 

- The construction of $Y_\algA$ is strongly dependent of the choice of $\algA$, see Section \ref{Remarque-Frechet}. 

\bigskip

The main points of our construction are the followings: 

-- The Hilbert space is $\hH$, so that $\hxi\in \hH$ means $\hxi(r)\in H$ for almost all $r\in G$ and $\int_G \dd\mu_G(r)\, \norm{\hxi (r)}^2 < \infty$. The space $C_c(G, H)$ is dense in $\hH$. 
Let $\rho_\alpha$ be the induced representation of $\bA$ on $\hH$ defined by $[\rho_\alpha(a)\,\hxi\,](r) \vc \rho[\alpha_{r^{-1}}(a)]\,\hxi(r)$ for any $\hxi \in \hH$ and $a \in \bA$, and let $\lambda_G$ be the left regular representation of $G$ on $\hH$ defined by $(\lambda_G(s) \,\hxi\,)(r) \vc \hxi(s^{-1} r)$ for any $s \in G$. Then $( \rho_\alpha, \lambda_G, \hH)$ is a covariant representation of $(\bA, G, \alpha)$, and we denote by $\hrho \vc \lambda_G \ltimes \rho_\alpha$ its integrated representation, given explicitly by
\begin{equation}
\label{eq-defrepresentationhrho}
[\hrho(f)\,\hxi \,](r) \vc \int_G \dd\mu_G(r') \, \rho (\alpha_{r^{-1}} [f(r')]) \,\hxi(r'^{-1} r), \quad \text{ for $\hxi \in \hH$.} 
\end{equation}
This representation restricts to a $*$-representation of $\algA$. The crossed product $\bB=G \ltimes_{\alpha, \red} \bA$ is the norm closure of $\hrho[C_c(G, \bA)]$ in $\caB(\hH)$, and by density of $\algA$, it is also the norm closure of $\hrho(\algA)$.\\ 
We will use the unitary operator on $\hH$:
\begin{align*}
\hU:\,r\in G \mapsto U_r \in \mathcal{U}(H).
\end{align*}

-- Let us first define
\begin{align}
\label{Yp}
&\hY_p \vc \{ \hxi \in \hH  \text{ such that the map }r \in G \mapsto \norm{\rho[p(r)]\, \hxi(r)}_H \text{ is in } L^2(G) \},
\end{align}
and then, $\htheta : \hH \to \hH$ as the (unbounded) operator with $\Dom(\htheta\,) \vc \hU^*\,\hY_p$ (whose density is proved in Lemma~\ref{lem-densedomaintheta}) by
\begin{equation}
\label{thetahat}
(\htheta \,\hxi\,)(r) \vc \rho (\alpha_{r^{-1}}[p(r)])\, \hxi(r) = U_r^* \rho [p(r)] U_r \,\hxi(r), \quad \text{for $\hxi \in \Dom(\htheta\,)$ and $r \in G$.}
\end{equation}
By the central property of $p$ and the cocycle relation $$\alpha_{r^{-1}}[p(r)] \,p(r^{-1})  = p(r^{-1}) \,\alpha_{r^{-1}}[p(r)] = p(r^{-1}r) = 1,$$
one shows that the (unbounded) operator $\htheta^{-1}$ defined by
\begin{equation*}
(\htheta^{-1} \,\hxi\,)(r) \vc \rho[p(r^{-1})]\, \hxi(r) \text{ on } \Dom(\htheta^{-1}) 
\vc \{ \hxi \in \hH \mid [r \mapsto \norm{\rho[p(r^{-1})]\hxi(r)}_H] \in L^2(G)\}
\end{equation*}
is the inverse of $\htheta$.

\begin{lemma}
\label{lem-densedomaintheta}
The subspace $C_c(G,H)$ is dense in  $\hH$, contained in $\Dom(\htheta) \cap \Dom(\htheta^{-1})$ and is stable by $\htheta$ and $\htheta^{-1}$. In particular $\Dom(\htheta)$ and $\Dom(\htheta^{-1})$ are dense in $\hH$.
\end{lemma}

\begin{proof}
Recall that $M(A)$ is a $C^*$-algebra endowed with the strict topology and that the extension of $\rho$ to $M(\bA)$ is a $*$-homomorphism \cite[II.7.3.9]{Blac06a}). Then for any continuous map $c : G \to M(\bA)$, the composition  $r\in G \mapsto c(r) \in M(\bA) \mapsto \rho[c(r)] \in \caB(H)$ is a continuous map. Since $r \mapsto r^{-1}$ is continuous on $G$, $r \mapsto \rho[c(r^{-1})]$ is also continous. Thus $\rho[p(r^{\pm 1})]\,\hxi(r)$ are continuous in $r$ when $\hxi \in C_c(G, H)$. Since $r \mapsto U_r$ is strongly continuous, $r \mapsto U_{r} \,\hxi(r)$ is continuous for $\hxi \in C_c(G,H)$. Using compositions of these continous maps, we get that $C_c(G,H) \in \Dom(\htheta) \cap \Dom(\htheta^{-1})$ and that $C_c(G,H)$ is stable by $\htheta$ and $\htheta^{-1}$.
\end{proof}

We will show in Proposition~\ref{prop-beta} that the map $\beta$ defined by 
\begin{equation}
\label{beta}
\beta(f)(r) \vc p(r) \,f(r),\quad f\in C_c(G, \bA),\,r\in G
\end{equation}
is an automorphism of $C_c(G, \bA)$ satisfying
\begin{equation}
\label{eq-beta-p}
\hrho\, [\beta(f)]\,\hxi = \htheta \,\hrho(f)\, \htheta^{-1}\, \hxi \quad \text{for any $\hxi \in C_c(G, H)$ and $f \in C_c(G, \bA)$.}
\end{equation}
We prove in Proposition \ref{prop-beta} that $\beta$ is also an automorphism of $\algA$.

-- Let $\hD$ be the unbounded operator well defined on $\hU^*\,C_c(G,\Dom(D))$ by
\begin{equation}
\label{Dchap}
(\hD \,\hxi\,)(r) \vc U_r^*\, D \,U_r \,\hxi(r).
\end{equation}
Then the twisted commutator $[\hD, \hrho(f)]_\beta \vc \hD\, \hrho(f) - \hrho(\beta[f]) \hD$ is bounded for any $f \in \algA$ as shown in Lemma~\ref{prop-hD}.

-- The last important object is the unbounded operator defined on $\hU^*\,\hY_p$ by
\begin{align}
\label{eq-DH}
\caT_{\cone,\,\ct} \vc \cone \bbbone + \ct \,\htheta
\end{align}
where $\cone, \ct \in \gR$ with $\ct\neq 0$ are arbitrary real parameters. This operator has also bounded twisted commutator with $\hrho(f)$ as shown in Lemma \ref{prop-caT}. 

\bigskip
Let us now define the ingredients of our modular-type $\beta$-twisted spectral triple $(\algA, \caH,\DD)$:

-- $\caH \vc \hH \otimes \gC^2 = L^2(G, \dd\mu_G) \otimes H \otimes \gC^2$, 

-- $\pi\vc \hrho \otimes \Id_{\gC^2}$ on $\caH$,

-- the operator $\DD$ on $\caH$ is given by
\begin{equation}
\label{eq-DiracG}
\DD \vc \hD \otimes \gamma^1 + \caT_{\cone,\,\ct} \otimes \gamma^2 ,
\end{equation}
where $\gamma^1,\gamma^2$ are the usual selfadjoint Pauli matrices on $\gC^2$, 

- We promote $\htheta$ to an operator on $\caH$ via
\begin{align}
\label{defTheta}
\Theta \vc \htheta \otimes \bbbone.
\end{align}
We show in Proposition \ref{D et Theta selfadjoints} that $\DD$ and $\Theta$ are selfadjoint operators on specified domains.

Our main result in this section is the following

\begin{theorem}
\label{thm-twisted-triple}
Given a spectral triple $(A, H, D)$ and a dynamical system $(\bA,G,\alpha)$ satisfying Hypothesis~\ref{hyp-dirac}, then $(\algA, \caH,\DD,\beta)$ is a modular-type $\beta$-twisted spectral triple (modulo some convergence of an integral on the group to get \eqref{modularcompactresol}, see Proposition \ref{prop-compactness} or its corollary).
\end{theorem}

\subsection{\texorpdfstring{About $\DD$ and $\Theta$}{About D and Theta}} 
\label{About T and Theta}

For any $r \in G$ and $\cone,\ct \in \gR,\,c\neq 0$,  let us define successively the following operators:
\begin{alignat*}{2}
& T_p(r) \vc \cone + \ct\, \rho[p(r)] &\quad&\text{acting on }H,\\
& \DD_r \vc D\otimes \gamma^1+T_p(r)\otimes \gamma^2 &&\text{acting on } H\otimes \gC^2,\\
& V_r \vc U_r \otimes \bbbone_2  &&\text{acting on } H\otimes \gC^2.
\end{alignat*}
Then 
\begin{align*}
(\DD \,\xi)(r)=V_r^*\,\DD_r\,V_r\,\xi(r),\quad \text{for } \xi \in \Dom(\DD).
\end{align*}
We propose now the set which will be helpful for the domain of $\DD$:
\begin{align*}
\caY \vc \{\psi\in \caH \,\vert\,\psi(r) \in \Dom(D)\otimes \gC^2 \,\dd\mu\text{-}a.e. \text{ and } [r \in G \mapsto \norm{\DD_r \,\psi}_{H\otimes \gC^2}] \in L^2(G)\}.
\end{align*}
Defining $\caV: r\in G\mapsto V_r$ which is a unitary on $\caH$, we get

\begin{proposition}
\label{D et Theta selfadjoints}
The operator $\DD$ defined on $\Dom(\DD) \vc \caV^*\caY$ is selfadjoint. \\
The operator $\Theta$ defined on $\Dom(\Theta) \vc \caV^*(\hY_p \otimes \gC^2)$ is selfadjoint and positive.
\end{proposition}

\begin{proof}
We first show that $\DD$ is well defined on $\caV^*\caY$:  
$\norm{\DD \,\caV^*\psi}^2_\caH=\int_G\dd \mu(r)\,\norm{\DD_r\, \psi(r)}^2_{H\otimes \gC^2}$ 
is finite for $\psi \in \caY$ by hypothesis since $\psi(r)=V^*_r(\hxi(r)\otimes v)$ with $v\in \gC^2$ and $\hxi(r) \in \Dom(D)$ for any $r\in G$, so $\DD_r \psi(r)=DU^*_r\,\hxi(r)\otimes \gamma^1\,v+T_p(r)U_r^*\,\hxi(r)\otimes \gamma^2\, v$ is well defined by Hypothesis \ref{hyp-dirac}-\ref{hyp-comp-D-alpha}.
Since $T_p(r)$ is a bounded selfadjoint operator on $H$, $T_p(r)\otimes \gamma^2$ is a bounded selfadjoint perturbation of the selfadjoint unbounded operator $D\otimes \gamma^1$ on $\Dom(D)\otimes \gC^2$, and $\DD_r$ is a selfadjoint on $\Dom(\DD_r)\vc \Dom(D) \otimes \gC^2$ \cite[Chap. V, Sect 4, Theorem 4.3]{Kato}. Moreover, one checks directly that $\DD$ is symmetric on $\caV^*\caY$, which is dense by Hypothesis~\ref{hyp-dirac}-\ref{preservdom}.

To obtain selfadjointness of $\DD$, we show that $\Ran(\DD\pm i)=\caH$: \\
Since $\DD_r$ is selfadjoint for any $r \in G$, $\Ran(\DD_r+i) = H\otimes \gC^2$. When $\varphi \in  \caH$, define $\varphi_0(r) \vc (\DD_r+i)^{-1} V_r \,\varphi(r)$ for any $r \in G$. Then 
\begin{align*}
\norm{\varphi_0(r)}_{H\otimes \gC^2} \leq \norm{(\DD_r+i)^{-1}}_{\caB(H\otimes \gC^2)} \norm{\varphi(r)}_{H\otimes \gC^2} \leq \norm{\varphi(r)}_{H\otimes \gC^2},
\end{align*}
so that $\varphi_0 \in \caH$, and 
$\norm{\DD_r \,\varphi_0(r)}_{H\otimes \gC^2} \leq \norm{\DD_r (\DD_r+i)^{-1}}_{\caB(H\otimes \gC^2)} \norm{\varphi(r)}_{H\otimes \gC^2} \leq \norm{\varphi(r)}_{H\otimes \gC^2}$. Thus $\varphi_0(r) \in \Dom(\DD_r)$ and $r \mapsto \norm{\DD_r \varphi_0(r)}_{H\otimes \gC^2}$ is in $L^2(G)$. Then $\psi \vc \caV^* \varphi_0$ belongs to $\caV^*\caY$ and $(\DD+i)\psi = \varphi$ by construction, yielding $(\DD+i)\caV^*\caY = \caH$. Similarly, $(\DD-i)\caV^*\caY = \caH$, showing $\DD$ is selfadjoint on $\caV^*\caY$.

Let us now consider $\Theta$. It is sufficient to prove that $\htheta$ is selfadjoint on its domain $\Dom(\htheta) \vc \hU^*\,\hY_p$ and positive. From Lemma~\ref{lem-densedomaintheta}, $\Dom(\htheta)$ is dense in $\hH$. For $r \in G$, the operator $\xi \mapsto \theta_r \,\xi \vc \rho (\alpha_{r^{-1}}[p(r)])\, \xi$ is positive and bounded on $H$, so that $\Ran(\theta_r \pm i) = H$. Similar arguments already used for $\DD$ shows that $\htheta$ is selfadjoint and positive on $\hU^*\,\hY_p$. 
\end{proof}

While the hypothesis \eqref{eq-hyp-comm-bounded} is sufficient to get a bounded twisted commutator, we cannot expect it is also sufficient to get \eqref{modularcompactresol} and we need more information like for instance the convergence of the integral in Proposition \ref{prop-compactness} for $c=1<s$.

\begin{lemma}
\label{Xborne}
Despite the fact that $\Theta$ can be an unbounded operator,
\begin{align}
\label{borne}
\Theta^c(\bbbone+\DD^2)^{-c/2} \text{ is bounded for any }c \geq 0.
\end{align}
In particular, $\Theta^c(\bbbone+\DD^2)^{-1/2}$ is bounded when $0\leq c \leq1$.
\end{lemma}

\begin{proof}
If $A=\Theta^c(\bbbone+\DD^2)^{-c/2}$, then $A$ is bounded since $\vert A^*\vert$ is bounded: 
\begin{equation*}
\vert A^* \vert^2=\Theta^c(\bbbone+\DD^2)^{-c}\Theta^c \leq \Theta^c(\ct^2 \,\Theta^2)^{-c} \Theta^c = \ct^{-2c}\,\bbbone.
\end{equation*}
Thus $\Theta^c(\bbbone+\DD^2)^{-1/2}=\Theta^c\,(\bbbone+\DD^2)^{-c/2}\,(\bbbone+\DD^2)^{-(1-c)/2}$ is bounded when $0\leq c \leq1$.
\end{proof}

\begin{remark}\label{rmk-udu-d-bounded}
Our construction differs from \cite{BMR, HSWZ, Pat}: for $r\in G$, the difference $U_r\,D\,U_r^* - D$ does not extend a priori to a bounded operator. If for instance $G$ is discrete and $z(r)=\vartheta(r)\,\bbbone$, with $\vartheta:G\to \gC^\times$ a group homomorphism (see Section \ref{sec-example-discrete-affine-group} for an example), this difference is equal to $(\abs{\vartheta(r)}^{-2}-1)\,D$ which is unbounded for each $r$ not in the kernel of $\vartheta$. So we are not dealing with an unbounded equivariant Kasparov module defining a class in $KK_1^G(\bA,\gC)$. 
\end{remark}

\subsection{About the spectral dimension}
\label{bifurcation C-M}

The change of Hilbert space $H \to \hH$ is quite important in this construction. We have supposed in Hypothesis~\ref{hyp-dirac} that $(\rho, U, H)$ is a covariant representation of $(\bA, G, \alpha)$, so that it is possible to consider the integrated representation $U \ltimes \rho$ of $G \ltimes_{\alpha, \red} \bA$ on $H$ defined by
\begin{equation*}
(U \ltimes \rho)(f) \,\xi \vc \int_G \dd\mu_G(r)\, \rho[f(r)]\, U_r \,\xi \quad \text{for any $f \in L^1(G, A)$ and $\xi \in H$.}
\end{equation*}
Then the original operator $D$ acts on $H$ and we have 
\begin{align*}
D (U \ltimes \rho)(f)\, \xi &= \int_G \dd\mu_G(r)\, D \rho[f(r)] \,U_r\,\xi ,\\
(U \ltimes \rho)\beta(f) D \,\xi &= \int_G \dd\mu_G(r)\, \rho[p(r) f(r)] U_r D\, \xi =  \int_G \dd\mu(r)\,  \rho[f(r)\, p(r)] U_r D U_r^* \, U_r\, \xi \\
		      &= \int_G \dd\mu_G(r)\,  \rho[f(r)] z(r) z(r)^*\,  {z(r)^*}^{-1} D z(r)^{-1} \, U_r \,\xi\\
		      & = \int_G \dd\mu_G(r)\,  \rho[f(r)] z(r)  D z(r)^{-1} \, U_r \,\xi
\end{align*}
so that
\begin{equation*}
\big[D, (U \ltimes \rho)(f) \big]_\beta \,\xi = \int_G \dd\mu_G(r)\,  \big[D, \rho[f(r)] z(r) \big] z(r)^{-1} \, U_r \,\xi.
\end{equation*}
Using the inequality \eqref{eq-hyp-comm-bounded}, the operator $[D, (U \ltimes \rho)(f)]_\beta$ is bounded as shown in the proof of Lemma \ref{prop-hD}. Thus $(\algA, H, D, \beta)$ defines a $\beta$-twisted spectral triple once we know that \eqref{compactresol} holds true, namely that $(U \ltimes \rho)(f)(\bbbone+D^2)^{-1/2}$ is compact:
\begin{align*}
(U \ltimes \rho)(f)(\bbbone+D^2)^{-1/2} & =\int_G \dd\mu_G(r)\, \rho[f(r)]\, U_r(\bbbone+D^2)^{-1/2} \\
&= \int_G \dd\mu_G(r)\, U^*_r\,\rho[\beta(f)(r)]\,(\bbbone+D^2)^{-1/2}.
\end{align*}
Since \eqref{compactresol} is satisfied for $(A,H,D)$, the integrand is in $\caK(H)$ (compact operators on $H$) for any $f \in L^1(G,A)$ and $r\in G$. Since this integral coincides with the Bochner integral because $\caK(H)$ is separable \cite{Lang93a}, it is compact.\\
Of course, if the crossed product is the discrete one as used in \cite{ConnesMoscovici2006}, then the integral is replaced by a finite series and \eqref{compactresol} is satisfied!

However, one will see on the affine group example, see Remark \ref{affine bifurcation C-M}, that the spectral dimension of this last  twisted spectral triple remains unchanged, while we precisely want to measure the possible influence of the action of $G$ on the original spectral dimension of $(A,H,D)$. Recall that for a finitely summable triple $(\mathcal{A},\caH,\DD)$, the spectral dimension is defined by (see \cite{CGRS})
\begin{align}
\label{dimension}
p\vc \inf \{s>0  \ \mid \ \forall a\in \mathcal{A}^+,\,\Tr  \pi(a)(\bbbone+\DD^2)^{-s/2} <\infty \}.
\end{align}
This is mainly why we prefer to change the Hilbert space representation from $H$ to $\hH$. This in fact creates room enough to implement the twist $\beta$ via an operator $\htheta$ on $\hH$ which is crucially involved in the definition of $\DD$, see \eqref{eq-DiracG}.

Conversely, if the original triple is finitely summable, our modular-type twisted spectral triple possibly loses this property. This is for instance the case of the affine group where $\pi(a)\,(\bbbone+\DD)^{-s/2}$ is never trace-class for any $s>0$ \cite[Proposition 27]{Mata14a}. Thus, we are driven to follow a different path and we investigate the compactness of the operator $\Theta\, \pi(a)\,(\bbbone+\DD)^{-s/2}$ as a function of $s\in \gR$, see Proposition \ref{prop-compactness}. 
Moreover, for the affine group considered in Section \ref{sec-example-affine-group}, we prove in Remark \ref{Dixmiertraceable} that $\Theta\, \pi(a)\,(\bbbone+\DD)^{-s/2} \,\pi(b)$ is in the Dixmier-class for $s=2$ for any $a,b \in \algA$.
 This is why we asked in a modular-type spectral triple for the replacement of \eqref{compactresol} by \eqref{modularcompactresol}.
 
\bigskip
We first want to know when $\Theta\, \pi(f) (1 + \DD^2)^{-s/2}$ is a compact operator and begin with a simple lemma on the compactness of operators which are compact-operator-valued and whose kernel satisfies a Hilbert--Schmidt type condition.

For a separable Hilbert space $H$, let $\caK(H)$ be the separable $C^*$-algebra of compact operators on $H$ and $L^2(G \times G, \dd\mu_G \times \dd\mu_G) \otimes \caK(H)$ be the completion of the algebraic tensor product $L^2(G \times G, \dd\mu_G \times \dd\mu_G) \odot \caK(H)$ for the Hilbert--Schmidt norm
\begin{equation*}
\norm*{K}_\HS^2 \vc \int_{G \times G} \dd\mu_G(r) \,\dd\mu_G(r') \, \norm*{K(r,r')}_{\caB(H)}^2
\end{equation*}
for any $K \in L^2(G \times G, \dd\mu_G \times \dd\mu_G) \odot \caK(H)$.

\begin{lemma}
\label{lem-HilbertSchmidt}
When $K \in L^2(G \times G, \dd\mu_G \times \dd\mu_G) \otimes \caK(H)$ (i.e. $\norm*{K}_\HS<\infty$), the associated operator $X_K$ defined on $\hH \vc L^2(G, \dd\mu_G) \otimes H$ by $(X_K \,\hxi\,)(r) \vc \int_G \dd\mu_G(r') \, K(r,r') \,\hxi(r')$ is compact.
\end{lemma}

\begin{proof}
This is an adaptation of the usual proof that any Hilbert--Schmidt operator is compact. The separability of $\caK(H)$ ensures that the integrals we consider are strong integrals, in the sense of Bochner \cite{Lang93a}. For any $\hxi \in \hH$, one has with Hölder inequality
\begin{align*}
\norm{X_K \,\hxi\,}_{\hH}^2 
&= \int_G \dd\mu_G(r) \, \norm{(X_K \,\hxi\,)(r)}_{H}^2  \\
& \leq \int_G \dd\mu_G(r) \left( \int_G \dd\mu_G(r') \, \norm*{K(r,r')}_{\caB(H)} \norm{\hxi(r')}_{H} \right)^2
= \norm*{K}_\HS^2 \, \norm{\hxi\,}_{\hH}^2
\end{align*}
so $\norm*{X_K}_{\caB(\hH)} \leq \norm*{K}_\HS < \infty$.

Let $\{ \phi_n \}_{n \in \gN}$ be an orthonormal basis of the separable Hilbert space $L^2(G, \dd\mu_G)$. Then the $\phi_m \otimes \phi_n$'s, for $m,n \in \gN$, produce an orthonormal basis for $L^2(G \times G, \dd\mu_G \times \dd\mu_G)$, and by definition, the kernel $K \in L^2(G \times G, \dd\mu_G \times \dd\mu_G) \otimes \caK(H)$ is limit of elements of the form
\begin{equation*}
K_N = \sum_{m,n \leq N} \phi_m \otimes \phi_n \otimes k_{m,n} \quad \text{ with } k_{m,n}  \in \caK(H).
\end{equation*}
The operator $X_N$ on $\hH$ defined by the kernel $K_N$ belongs to $\caK(L^2(G, \dd\mu_G)) \otimes \caK(H) \subset \caK(\hH\,)$ because it is a finite sum of rank one operators along $L^2(G, \dd\mu_G)$ and compact operators along $H$. The inequality $\norm*{X - X_N}_{\caB(\hH)} \leq \norm*{K - K_N}_\HS$ implies that $X \in \caK(\hH)$.
\end{proof}

\begin{proposition}
\label{prop-compactness}
Suppose that $[\hD, \caT_{\cone,\,\ct}]$ extends to a bounded operator on $\hH$, and let $c\geq 0$ and $s\geq 1$.

(1) The operator $\Theta^c\, \pi(f) (\bbbone + \DD^2)^{-s/2}$ on $\caH$ is compact for any $f \in \algA$ when
\begin{equation}
\label{eq-compact}
\int_G \dd\mu_G(r)\, \Delta_G(r)^{-1} \norm*{\rho[p(r)]}_{\caB(H)}^{2c }\, \norm{[\bbbone + T_p(r)^2]^{-1}}_{\caB(H)}^s < \infty.
\end{equation}

(2) When the original triple $(A,H,D)$ is unital, the operator $\Theta^c\, (\bbbone + \DD^2)^{-s/2}$ is compact when \eqref{eq-compact} holds true even if $\algA$ is not unital.
\end{proposition}

Before the proof we begin with the following remark: since $$\norm{[\hD, \caT_{\cone,\,\ct}]\, \hxi\,}^2_{\hH} \leq \int_G \dd\mu_G(r) \norm{[D, T_p(r)]}_{\caB(H)}^2 \; \norm{\hxi(r)}_{H}^2,$$ the boundedness of $[\hD, \caT_{\cone,\,\ct}]$ is related to the behavior in $r$ of the family of operators $[D, T_p(r)] = \ct [D, \rho(p(r))]$ on $H$. For instance, if $\norm{[D, \rho(p(r))]}_{\caB(H)}$ is uniformly bounded in $r\in G$, then $[\hD, \caT_{\cone,\,\ct}]$ is bounded.

When $\algA$ is unital, so is  $\bB=G \ltimes_{\alpha, \red} \bA$ and this happens if and only if $G$ is discrete and $\bA$ is unital \cite[II.10.3.9]{Blac06a}. In this case, $(1)$ implies $(2)$. 

\begin{proof}
(1): From the definition \eqref{eq-DiracG} of $\DD$, one has
\begin{equation}
\label{Dcarre}
\DD^2 = (\hD^2 + \caT_{\cone,\,\ct}^2) \otimes \bbbone_2 + i[\caT_{\cone,\,\ct}, \hD\,] \otimes \gamma^3,
\end{equation}
so that $\DD^2$ is a bounded perturbation of $(\hD^2 + \caT_{\cone,\,\ct}^2) \otimes \bbbone_2$ and it is sufficient to prove compactness of the operator $\htheta^c\, \hrho(f) [\bbbone + \hD^2 + \caT_{\cone,\,\ct}^2]^{-s/2}$ on $\hH$. 

Since $i[\caT_{\cone,\,\ct}, \hD\,] \,\hxi(r)=U_r^*\, i[T_p(r), D] \,U_r\,\hxi(r)$ by \eqref{Dchap} and \eqref{eq-UaU*} and $i[T_p(r), D]$ is a symmetric operator on $H$ for any $r \in G$, $i[\caT_{\cone,\,\ct}, \hD\,]$ is symmetric so selfadjoint being bounded. Thus \eqref{Dcarre} implies that the unbounded operator $(\hD^2 + \caT_{\cone,\,\ct}^2) \otimes \bbbone_2 $ is selfadjoint and positive on $\caH$ since $\DD^2$ is selfadjoint by Proposition \ref{D et Theta selfadjoints}.
For $\hxi \in \Dom(\hD^2 + \caT_{\cone,\,\ct}^2)$, we get
\begin{equation*}
( [\bbbone + \hD^2 + \caT_{\cone,\,\ct}^2] \,\hxi \,)(r) = U_r^*\, [\bbbone + D^2 +  T_p(r)^2]\, U_r \, \hxi(r).
\end{equation*}
Since $X^{-s/2} = \Gamma(s/2)^{-1} \int_0^\infty \dd t\, t^{s/2-1} e^{-t\,X} $ for any positive operator $X$ and $s > 0$, one gets, with $X = \bbbone+\hD^2 + \caT_{\cone,\,\ct}^2$,  
\begin{equation*}
( [\bbbone + \hD^2 + \caT_{\cone,\,\ct}^2]^{-s/2}\, \hxi \,)(r) = U_r^*\,[\bbbone + D^2 + T_p(r)^2 ]^{-s/2} U_r \, \hxi(r).
\end{equation*}
Given $c >0$ and $s\geq1$, one has
\begin{align*}
\MoveEqLeft[3] (\, \theta^c \,\hrho(f) [\bbbone + \hD^2 + \caT_{\cone,\,\ct}^2\,]^{-s/2}\, \hxi \,)(r)
\\
&= 
\int_G \dd\mu_G(r') \, \rho(\alpha_{r^{-1}}[p(r)^c f(r')]) \,U_{r'^{-1}r}^{*} [ \bbbone + D^2 + T_p(r'^{-1}r)^2 ]^{-s/2} \,U_{r'^{-1}r}\, \hxi(r'^{-1}r)\\
& = \int_G \dd\mu_G(r') \, K_{c,s}(r,r')\, \hxi(r')
\end{align*}
where
\begin{equation}
\label{K_s}
K_{c,s}(r,r') \vc \Delta_G(r')^{-1}\, U_{r'}^{*} \,\rho\big( \alpha_{r'r^{-1}}[ p(r)^c f(r r'^{-1})] \big) \, [\bbbone + D^2 + T_p(r')^2 ]^{-s/2} U_{r'}\,.
\end{equation}
One has $\rho(a)(\bbbone+D^2)^{-1} \in \caK(H)$ since by hypothesis \eqref{compactresol} on the spectral triple $(A,H, D)$,  $\rho(a)(\bbbone+D^2)^{-1/2} \in \caK(H)$, so 
\begin{equation*}
\rho(a)\,[\bbbone + D^2 + T_p(r')^2]^{-1} = \rho(a)(\bbbone+D^2)^{-1}\big[\bbbone-T_p(r')^2\,[\bbbone+D^2+T_p(r')^2]^{-1}\big]
\end{equation*}
is compact, as a bounded deformation of $\rho(a)(\bbbone+D^2)^{-1}$. \\Since an operator $T$ is compact if $TT^*$ is compact, we get that $\rho(a) [\bbbone+D^2+T_p(r')^2]^{-1/2}$ is compact, and so is $\rho(a) [\bbbone+D^2+T_p(r')^2]^{-s/2}$ for $s \geq 1$. This implies that $K_{c,s}(r,r') \in \caK(H)$ for any $r,r' \in G$ and any $c>0,\,s \geq 1$.

For $r \in G$, one has
\begin{equation*}
[\bbbone + D^2 + T_p(r)^2]^{-1} = [\bbbone + T_p(r)^2]^{-\onehalf} \big[ \bbbone + (\bbbone + T_p(r)^2)^{-\onehalf}  D^2 \, (\bbbone + T_p(r)^2)^{-\onehalf} \big]^{-1} (\bbbone + T_p(r)^2)^{-\onehalf}
\end{equation*}
thus $\left[\bbbone + D^2 + T_p(r)^2\right]^{-1} \leq \left[\bbbone + T_p(r)^2 \right]^{-1}$ and
\begin{equation*}
\norm{[\bbbone + D^2 + T_p(r)^2]^{-s/2} }_{\caB(H)}^2
= \norm{[\bbbone + D^2 + T_p(r)^2]^{-1}}_{\caB(H)}^s
\leq \norm{[\bbbone + T_p(r)^2 ]^{-1}}_{\caB(H)}^s.
\end{equation*}
Using $\norm*{\rho(\alpha_r[p(r')^c])}_{\caB(H)} = \norm*{\rho[p(r')]}^c_{\caB(H)}$ for $r,r' \in G$ by \eqref{eq-UaU*}, $M_f^2:=\sup_{r\in S_f}\norm{\rho[f(r)]}^2_{\caB(H)}$, one gets the estimate
\begin{equation*}
\norm*{K_{c,s}(r,r')}_{\caB(H)}^2
\leq M_f^2 \, \Delta_G(r')^{-2} \, \bbbone_{S_f}(r r'^{-1}) \, \norm{\rho[p(r)]}_{\caB(H)}^{2c} \, \norm{[\bbbone + T_p(r')^2 ]^{-1}}_{\caB(H)}^s.
\end{equation*}
Let $N_{z, S_f} \vc \sup_{r \in S_f} \norm{\rho[p(r)]}_{\caB(H)}$. One has $r r'^{-1} \in S_f$ iff $r \in S_f r'$ (right translated of $S_f$ by $r'$), and 
\begin{align*}
\sup_{r \in S_f r'} \norm{\rho[p(r)]}_{\caB(H)} 
&= \sup_{r \in S_f} \norm*{\rho[p(rr')]}_{\caB(H)} 
 \\ &\leq \sup_{r \in S_f}  \norm*{\rho[p(r)]}_{\caB(H)} \norm*{\rho[p(r')]}_{\caB(H)} 
 \quad \text{ by \eqref{eq-relationc(rr')} and \eqref{eq-UaU*}}
\\
&\leq N_{z, S_f} \norm*{\rho[p(r')]}_{\caB(H)}.
\end{align*} 
This gives
\begin{align}
&\norm*{K_{c,s}}^2_\HS \nonumber\\
& \quad\leq M_f^2 \int_{G \times G} \dd\mu_G(r) \dd\mu_G(r') \, \Delta_G(r')^{-2} \, \bbbone_{S_f}(r r'^{-1}) \, \norm{\rho[p(r)]}_{\caB(H)}^{2c} \, \norm{[\bbbone + T_p(r')^2 ]^{-1}}_{\caB(H)}^s \label{doubleintegral}
\\
&\quad \leq M_f^2  N_{z, S_f}^{2c}\, \mu_G(S_f) \int_G \dd\mu_G(r')\, \Delta_G(r')^{-1} \norm{\rho[p(r')]}_{\caB(H)}^{2c} \, \norm{[\bbbone + T_p(r')^2 ]^{-1}}_{\caB(H)}^s. \nonumber
\end{align}
Thus, if \eqref{eq-compact} holds true, the last integral converges and using Lemma~\ref{lem-HilbertSchmidt}, one concludes that $\theta^c\, \hrho(f) [\bbbone + \hD^2 +\caT_{\cone,\,\ct}^2]^{-s/2}$ is compact.

(2): When $A$ is unital, $(\bbbone+ D^2)^{-1/2}$ is compact, so one can take $a=\bbbone$ and forget about $\hrho(f)$ in the above arguments. So the kernel $K_{c,s}(r,r')$ in \eqref{K_s} is now diagonal and there is no need of the support of $f$ to take care of one of the two integrals in \eqref{doubleintegral}.
\end{proof}
Remark that the compactness of $\Theta^c\, \pi(f) (\bbbone + \DD^2)^{-1/2}$ simply means that $\Theta^c\, \pi(f)$ is relatively compact with respect to $\DD$ for any $f\in \algA$: the operator $A=\Theta^c\pi(f)(\DD-i\bbbone)^{-1}$ is compact since $AA^*=[\Theta^c\pi(f)(\bbbone+\DD^2)^{-1/2}]\,[\Theta^c\pi(f)(\bbbone+\DD^2)^{-1/2}]^*$ is compact. In particular $\pi(f^*)\Theta^c\pi(f)$ is a selfadjoint operator which is $\DD$-bounded, so there exists $a,b\geq 0$ depending on $f$, such that 
\begin{equation*}
\norm{\pi(f^*)\Theta^c\pi(f) \xi} \leq a \norm{\DD\xi} +b \norm{\xi}\text{ for any }\xi \in \Dom(\DD),
\end{equation*}
and the infimum of all admissible $a$'s is zero.
We will see an example in \eqref{rescompacte}.

Note that in the affine group case, the integral of \eqref{eq-compact} with $c=0$ diverges, see \eqref{div}; similarly, the integral \eqref{eq-compact} of Proposition \ref{prop-compactness} diverges for $s=1$, see \eqref{div2} for any $c\geq 1$.

\begin{corollary}
\label{compresolvent}
Assume that \eqref{eq-compact} holds true for any $c$ and $s$ such that $0<c <1<s$. Assume also that $\beta_c$ is an automorphism of $\algA$ (see Proposition \ref{prop-beta}) and $\Theta^c(\bbbone+\DD^2)^{-c/2}$ converges in norm to $\bbbone$ when $c \to 0$. \\
Then, the operator $\pi(f)(\bbbone+\DD^2)^{-1/2}$ is compact for any $f \in \algA$.
\end{corollary}
\begin{proof}
By hypothesis, the operator
\begin{equation*}
\Theta^c\, \pi(f) (\bbbone + \DD^2)^{-c/2}(\bbbone+\DD^2)^{-1/2}=\ \pi(g) [\Theta^c\,(\bbbone + \DD^2)^{-c/2}](\bbbone+\DD^2)^{-1/2}
\end{equation*}
is compact for any $c >0$ by an application of Proposition \ref{prop-compactness} and this gives the result since $\pi(f)$ and $(\bbbone+\DD^2)^{-1/2}$ are bounded while the term in bracket, which is bounded by Lemma \ref{Xborne}, goes to $\bbbone$ when $c \to 0$.
\end{proof}

When $[\Theta,\DD]=0$, which will be the case in Section \ref{sec-special-case-vartheta}, then $\Theta^c(\bbbone+\DD^2)^{-c/2}=X^c$, where $X=\Theta\,(\bbbone+\DD^2)^{-1/2}$ is a positive invertible operator, which is bounded as already seen in Lemma \ref{Xborne}. 

By construction, $\Theta$ is derived from the cocycle $p=zz^*$, see \eqref{thetahat}, which defines the automorphism $\beta$ by \eqref{beta} and $z$ is the answer of $\alpha$ to the covariant representation, see \eqref{eq-UaU*}.  For $c >0$, the operator $\Theta^c$ is not necessarily derived by this procedure from an action  $\alpha_c$ of $G$ on $\bA$: only the case $c=1$ is important algebraically.

\subsection{\texorpdfstring{On the choice of $\algA$}{On the choice of A}}
\label{Remarque-Frechet}

In \cite{Schweitzer}, several constructions of smooth dense subalgebras of a $C^*$-crossed product are exhibited. We recall here some of them, and outline their importance for our construction. 

Let $A$ be a Fréchet space with a topology given by an increasing family of semi-norms  $\norm{\,}_m,\,m\in \gN$. When $A$ has an algebra structure, $A$ is said to be a Fréchet $*$-algebra if the multiplication is continuous (so is automatically jointly continuous) and the involution is also continuous. The algebra is $m$-convex when the semi-norms $\norm{\,}_m$ are submultiplicative.

The action $\alpha : r\in G \mapsto \alpha_r \in \text{Aut}(A)$ of a topological group $G$ on $A$ is continuous if, for any $a \in A$, $r\in G \mapsto \alpha_r(a)$ is continuous for each semi-norm $\norm{\,}_m,\,m\in \gN$. A scale on $G$ is a continuous map $w : G\to [0,\infty)$ (in \cite{Schweitzer}, $w$ is only assumed to be a Borel map which is bounded on compact subsets of $G$). The scale $w$ is sub-polynomial if there exists $c>0$ and $d\in \gN$ such that $w(rr')\leq c\,[1+w(r)]^d\,[1+w(r')]^d$ for any $r,r' \in G$. The action $\alpha$ is $w$-tempered if, for each $m\in \gN$, there exists $c>0$ and $k,\,l\in \gN$ such that $\norm{\alpha_r(a)}_m \leq c \,w^k(r)\norm{a}_l$ for any $r\in G$, $a\in A$. The action $\alpha$ is $*$-preserving if $\alpha_r(a^*)=[\alpha_r(a)]^*$ for any $r\in G, a\in A$.

\begin{lemma}
\label{lem-Afrechet}
Let $A$ be a Fréchet $*$-algebra endowed with a $*$-preserving continuous action $\alpha$ of a locally compact group $G$ which is  $w$-tempered for a scale $w$ on $G$. Then $C_c(G,A)$ is a $*$-algebra for the convolution product \eqref{eq-product-alpha} and involution \eqref{eq-involution-alpha}.

Moreover, if the inclusion $A \hookrightarrow \bA$ is continuous and $A$ is dense in $\bA$, then $C_c(G,A)$ is a dense $*$-subalgebra of $C_c(G, \bA)$, and so of $\,\bB=G \ltimes_{\alpha, \red} \bA$.
\end{lemma}

When the hypotheses of the lemma are satisfied, the algebra $\algA = C_c(G,A)$ is a good candidate to enter into the modular-type $\beta$-twisted spectral triple of Theorem~\ref{thm-twisted-triple}. 

\begin{proof}
First, notice that the map $(r,a) \in G \times A \mapsto \alpha_r(a) \in A$ is continuous:
\begin{align*}
\norm{ \alpha_{r'}(a') - \alpha_r(a)}_m 
&\leq \norm{ \alpha_{r'}(a') - \alpha_{r'}(a) }_m + \norm{ \alpha_{r'}(a) - \alpha_r(a) }_m
\\
&\leq  c \, w^k(r')\norm{a' - a}_l + \norm{ \alpha_{r'}(a) - \alpha_r(a) }_m,
\end{align*}
and, as $r' \to r$ and $a' \to a$, these terms go to $0$ since $w$ is continuous, and $r \mapsto \alpha_r(a)$ is continuous for each semi-norm.

Let $f,g \in C_c(G,A)$. In \eqref{eq-product-alpha}, the integrand $h(r') \vc f(r')\,\alpha_{r'}[g(r'^{-1}r)]$ is $A$-valued and has a compact support since it is included in the support of $f$. Since, for any $r \in G$, the map $r' \mapsto \alpha_{r'}[g(r'^{-1}r)]$ is continuous by continuity of $r' \mapsto g(r'^{-1}r)$ and continuity of $(r,a) \mapsto \alpha_r(a)$, the map $r'\mapsto h(r')$ is continuous.

Thus $(f\star_\alpha g)(r)=\int_G \dd\mu_G (r') \,h(r') \in A$ since it belongs to $\mu_G(\Supp h) \times \overline{\text{co}}(h(G))$ (see \cite[3.27 Theorem]{Rudin} or \cite[Section 1.5]{Will07a}) and $A$ is a completely metrizable locally convex space \cite[5.35 Theorem]{Ali}. The support of $ f\star_\alpha g$ is included in the product in $G$ of the supports of $f$ and $g$, so is compact.

Thus it remains to show that the map $r\in G \mapsto f\star_\alpha g(r)$ is continuous. For any $m \in \gN$ and any $s,r \in G$, 
\begin{align}
\label{eq-int-cont-fastg}
\norm{(f\star_\alpha g)(s)-(f\star_\alpha g)(r)}_m 
& \leq \int_G \dd\mu(r')\,\norm{f(r')\,\alpha_{r'}[g(r'^{-1}s) - g(r'^{-1} r)]}_m
\nonumber
\\
&\leq \int_G \dd\mu(r')\,\norm{f(r')}_p \,\norm{\alpha_{r'}[g(r'^{-1}s) - g(r'^{-1} r)]}_q
\nonumber
\\
& \leq c_q \int_G \dd\mu(r')\,w^{k_q}(r')\,\norm{f(r')}_p\,\norm{g(r'^{-1}s) - g(r'^{-1} r)}_q
\nonumber
\\
&\leq c'\int_G \dd \mu(r') \norm{f(r')}_p\,\norm{g(r'^{-1}s) -g(r'^{-1}r)}_q.
\end{align}
Let $S_f = \Supp f$ and $S_g = \Supp g$. Then there exist some constants $M_f, M_g > 0 $ such that $\norm{f(r')}_p \leq M_f \,\chi_{S_f}(r')$ and $\norm{g(r'^{-1}s)}_q \leq M_g \,\chi_{S_g}(r'^{-1}s)$. Let $K$ be a compact neighborhood of $r$: as $s \to r$, we may suppose that $s \in K$, and then $\chi_{S_g}(r'^{-1}s) \leq \chi_{K S_g^{-1}}(r')$, as well as $\chi_{S_g}(r'^{-1}r) \leq \chi_{K S_g^{-1}}(r')$. Then $\norm{g(r'^{-1}s) -g(r'^{-1}r)}_q \leq 2 M_g \chi_{K S_g^{-1}}(r')$, and the integral \eqref{eq-int-cont-fastg} is dominated by $2 c' M_f M_g \int_G \dd\mu(r')\,  \chi_{S_f}(r')\,\chi_{K S_g^{-1}}(r') < \infty$. Since the integrand is continuous in $r$, by dominated convergence this yields that $f\star_\alpha g$ is a continuous function from $G$ to the Fréchet space $A$, so is in $C_c(G,A)$.

Moreover, to prove that $f^*$ (which has a compact support and is $A$-valued) is continuous, notice that $\alpha$ preserves the involution and that $r \mapsto \alpha_r[f(r^{-1})]$ is continuous as a composition of continuous functions: $r \mapsto (r, f(r^{-1})) \mapsto \alpha_r[f(r^{-1})]$.

The continuity of the inclusion of $A$ in $\bA$ assures that the integral in \eqref{eq-product-alpha} (defined by duality) gives the same result in $A$ and in $\bA$, so the products of $C_c(G,A)$ and $C_c(G,\bA)$ are the same. The density of $A$ in $\bA$ gives the density of $C_c(G,A)$ in $C_c(G, \bA)$, and then in $G \ltimes_{\alpha, \red} \bA$.
\end{proof}

\begin{remark}
\label{rmk-frechet-w-sigma}
Notice that $\algA = C_c(G, A)$ is not a Fréchet algebra. It is possible to embed it as a dense $*$-subalgebra of the Fréchet $*$-algebra $L_1^w(G,A)$ provided that the continuous $*$-preserving action $\alpha$ is $w$-tempered for a sub-polynomial scale $w$ on $G$ which is equivalent to its inverse $w_-(r) \vc w(r^{-1})$ \cite[Theorem 2.2.6]{Schweitzer}. Roughly speaking, elements in $L_1^w(G,A)$ are functions $G \to A$ which are $L^1$ when multiplied by any power of $w$, \textsl{i.e.} they are functions decreasing at infinity faster than $w^{-k}$ for any $k \geq 0$ (see Section \ref{Choice-affine-A} for more details when $A=\gC$). This algebra $L^w_1(G,A)$ is $m$-convex when $A$ is $m$-convex and $\alpha$ is $m$-$w$-tempered:  for each $m\in \gN$, there exists $c>0$ such that $\norm{\alpha_r(a)}_m \leq c \,w^m(r)\norm{a}_m$ for any $r\in G$ \cite[Theorem 3.1.7]{Schweitzer}.\\ An important desired property for the smooth algebra $\algA$ in the $C^*$-algebra $\bB$ is its spectral invariance since the two algebras will have the same $K$-theory, see \cite{Schweitzer1} for sufficient conditions.
\\
Notice that the criteria used to construct $\algA$ can be mimicked to construct $Y_\algA$ when $\hrho$ is a GNS representation.
\end{remark}

\subsection{Cocycles considerations}

We gather here some facts about cocycles, which are essential in our construction. Some properties established here will be used in other sections.

\begin{definition}
For a subgroup $X$ of $M(\bA)^\times$, the set $Z^1(G,X)$ of $X$-valued $\alpha$-one-cocycles is the set of
$c\in Z^1(G,M(\bA)^\times)$ such that $c(r) \in X,\,\forall r\in G$ and 
\begin{align}
c(rr')=c(r)\,\alpha_r [c(r')],\quad \forall\,r,r'\in G.  \label{eq-relationc(rr')}
\end{align}
\end{definition}
This cocycle relation implies that 
\begin{align}
\label{cocycle relation}
c(e)=\bbbone, \quad
c(r)^{-1} = \alpha_{r}[ c(r^{-1}) ],
\quad
c(r^{-1}) &= \alpha_{r^{-1}} [c(r)^{-1}], \quad\forall r \in G 
\end{align}
and $\alpha_r[c(r')] \in X$ for all $r,r'$. The kernel of $c$ is a subgroup of $G$ and if the action $\alpha$ is trivial, $Z^1(G,X)=\text{Hom}(G,X)$.

The role of $X$ and the fact that $\alpha$ preserves $X$ or not can be made clear. If $X'$ is the subgroup of $X$ generated by the $c(r),\,r\in G$, remark that \eqref{cocycle relation} only implies that $\alpha_r (X')=X'$ but not necessarily that $\alpha_r(X)\subset X$. For instance, we will use later $Z^1(G,Z(A)^\times)$ and $c(r) \vc \gamma(r)\, \bbbone$ where $\gamma:\,G \to \gC^\times$ is a group homomorphism; thus $c\in Z^1(G,X=Z(A)^\times)$ while $X'=\gC^\times \subsetneqq X$, so that $\alpha_r$ does not have to preserve $X$.

Since $\alpha_r$ preserves $A$ for any $r \in G$, we have $\alpha_r(Z(A)^\times) = Z(A)^\times$: indeed, for $c \in Z(A)^\times$ and $a \in A$, we get 
\begin{equation*}
\alpha_r(c)\,a = \alpha_r [c \,\alpha_{r^{-1}}(a)] = \alpha_r[ \alpha_{r^{-1}}(a) \,c ] = a\, \alpha_r(c),
\end{equation*}
so the first equality proves $\alpha_r(c)\,a \in A$ since by hypothesis $\alpha_{r^{-1}}(a) \in A$, and moreover that $\alpha_r(c)$ commutes with $A$.

\medskip

For any maps $c, c' : G \to X$, the product, inverse, adjoint and positivity are defined via $(cc')(r) \vc c(r)\, c'(r)$, $c^{-1}(r) \vc c(r)^{-1}$, $c^*(r) \vc c(r)^*$ and $c(r)>0$ for any $r \in G$, and for any $c > 0$ and $z \in \gC$, define $c^z(r) \vc c(r)^z$.

\begin{lemma}
\label{group of cocycles}
One has:\\
1) $Z(A)^\times \subset Z(M(\bA))^\times,$ so $Z(A)^\times$is an abelian group;\\
2) $Z^1(G,Z(A)^\times)$ is an abelian group stable by adjoint and polar decomposition and for any positive $c \in Z^1(G, Z(A)^\times)$, $c^z \in Z^1(G, Z(A)^\times)$ for any $z \in \gC$.
\end{lemma}

\begin{proof}
1) The abelianness of $Z(A)^\times $ will follow from the inclusion $Z(A)^\times\subset Z(M(\bA))^\times$. \\
It is sufficient for this inclusion to prove that for $m \in M(\bA),\,c \in Z(A)^\times,\, b\in \bA$, $mc\,b=cm\,b$ since, $\bA$ being an essential ideal in $M(\bA)$, this will imply $mc=cm$. \\For a given $b\in \bA$, let $a_\alpha\in A$ be a net norm-converging to $b$. Then $mc\,b=\lim_\alpha mc\,a_\alpha=cm \,b$ because $m(ca_\alpha)=m(a_\alpha c)=(ma_\alpha)c=c(ma_\alpha)$.

Now, for $c,\,c'\in Z^1(G,Z(A)^\times)$ and $a\in A$, since all elements commute,
\begin{align*}
(cc')(rr')&= c(rr')\,c'(rr')=c(r) \alpha_r[c(r')]\,c'(r)\alpha_r[c'(r')] =c(r)c'(r) \,\alpha_r[c(r')]\alpha_r [c'(r')]\\
 &=cc'(r)\,\alpha_r[cc'(r')],\\
c^{-1}(r r') &= c(r r')^{-1}= \left[ \,c(r)\, \alpha_{r}[c(r')] \,\right]^{-1} = \alpha_r[c(r')]^{-1} \, c(r)^{-1} = c^{-1}(r) \,\alpha_r[c^{-1}(r')].
\end{align*}
It is obvious that $r \mapsto \bbbone \in M(\bA)$ is the unit for the product and that $c^{-1}$ is the inverse of $c$, so $Z^1(G,Z(A)^\times)$ is an abelian group.
  
2) If $c \in Z^1(G,Z(A)^\times)$ and $a \in A$, $c^*(r)\, a =[ a^*\,c(r) ]^*= [c(r)\, a^* ]^* = a\, c(r)^* = a \,c^*(r) \in A$ 
so that $c^*(r) \in Z(A)^\times$ for $r \in G$. On the other hand,
\begin{align}
\label{cocyclerelations}
c^*(rr') 
&= c(rr')^* 
= [ c(r)\, \alpha_r (c(r')) ]^* 
= \alpha_r[c(r')]^* \, c(r)^*
= c^*(r)\, \alpha_r [c^*(r')],
\end{align}
so that $c^*$ is a one-cocycle. 

Let $c \in Z^1(G,Z(A)^\times)$ be positive. Using the continuous functional calculus in the $C^*$-algebra $M(\bA)$, for any $z \in \gC$, $c(r)^z \in M(\bA)$ is well defined, and for any $a \in A \subset M(\bA)$ one has $[c(r), a] = 0$, so that $[c(r)^z, a] = 0$. This implies that $c(r)^z \in Z(A)^\times$. For $r,r' \in G$, $c(r) \,\alpha_r(c(r')) > 0$ since $c(r)$ and $\alpha_r[c(r')]$ commute,
 so
\begin{align*}
c^z(rr')
&= \big[ c(r) \,\alpha_r[c(r')] \big]^z
= c(r)^z \, \alpha_r[c(r')]^z
=c^z(r)\, (\alpha_r [c(r')])^z
\end{align*}
and $c^z$ is a one-cocycle. Thus if $c \in Z^1(G,Z(A)^\times)$, then $\abs{c} \vc (c^* c)^{1/2}$ is a positive one-cocycle, and $u \vc c\, \abs{c}^{-1}$ is also a one-cocycle, which is unitary because $[c ,\,\abs{c}^{-1} ]=0$.
\end{proof}

\begin{lemma}
\label{strict continuity}
Let $c$ be in $Z^1(G,M(\bA)^\times)$ such that $c$ is continuous. \\
(i) If $c^{-1}$ is also continuous, then the map $r\in G\mapsto \alpha_r(a)\,c(r) \in M(\bA)$  is continuous for any $a\in \bA$ (for the strict topology). \\
(ii) The same conclusion holds for any continuous $c\in Z^1(G,Z(M(\bA))^\times)$.
\end{lemma}

\begin{proof}
Let $r,r'\in G$ and $b\in \bA$. 

(i) We have 
\begin{align*}
\norm*{\alpha_r(a)\,c(r)\,b-\alpha_{r'}(a)\,c(r')\,b}
& = \norm*{[\alpha_r(a)-\alpha_{r'}(a)] \, c(r) \, b + \alpha_{r'}(a) [c(r)-c(r')] \,b }\\
& \leq \norm*{\alpha_r(a)-\alpha_{r'}(a)}\, \norm*{c(r)\,b} + \norm*{a} \,\norm*{[c(r)-c(r')]\,b}
\end{align*}
so we get the continuity of $r \in G \mapsto\norm{\alpha_r(a)\,c(r)\,b}$ from the continuities of $\alpha$ and $c$.\\
Moreover, thanks to \eqref{cocycle relation},
\begin{align*}
\norm{b\,\alpha_r(a)\,c(r)\,-b\,\alpha_{r'}(a)\,c(r')} &=\norm{c^*(r)\,\alpha_r(a^*)\,b^*-c^*(r')\,\alpha_{r'}(a^*)\,b^*}\\
&\hspace{-1cm}= \norm{\alpha_r[c^{*-1}(r^{-1})\,a^*]\,b^*-\alpha_{r'}[c^{*-1}(r'^{-1})\,a^*]\,b^*}\\
&\hspace{-1cm}=\norm{[\alpha_r-\alpha_{r'}][c^{*-1}(r^{-1})\,a^*]\,b^* + \alpha_{r'}[c^{*-1}(r^{-1})\,a^*-c^{*-1}(r'^{-1})\,a^*]\,b^*}\\
&\hspace{-1cm} \leq \norm{[\alpha_r-\alpha_{r'}][c^{*-1}(r^{-1})\,a^*]}\,\norm{b^*}+ \norm{[c^{*-1}(r^{-1})-c^{*-1}(r'^{-1})]\,a^*}\,\norm{b^*}
\end{align*}
and we get the continuity of $r\in G \mapsto \norm{b\,\alpha_r(a) \,c(r)}$ since $c^{-1}$, so $(c^{-1})^*$, is  continuous.

(ii) When $c\in Z^1(G,Z(M(\bA))^\times)$, we use 
\begin{equation*}
\norm{b\,\alpha_r(a)\,c(r)\,-b\,\alpha_{r'}(a)\,c(r')} =\norm{\alpha_r(a^*)\,c^*(r)\,b^*-\alpha_{r'}(a^*)\,c^*(r')\,b^*}
\end{equation*}
and conclude as in the first part of the proof of (i).
\end{proof}

In some examples that will be given in Sections~\ref{sec-C*-algebra-semidirect-product} and \ref{sec-examples}, some cocycles will be of the form $c(r) = \gamma(r)\, \bbbone$ where $\bbbone$ is the unit in $M(\bA)$ and $\gamma : G \to \gR_+^\times$ is a continuous morphism of groups. This is why in Section \ref{sec-special-case-vartheta} we specify the construction for this situation.

\begin{remark}[The cohomology class of a cocycle]
\label{cohomo}
Two cocycles $c,\,c' \in Z^1(G,X)$ are said to be cohomologous if there exists $b\in X$ such that $c'(r)= b^{-1}\,c(r)\,\alpha_r(b)$. The one-cocycle $z$ is a coboundary if it is cohomologous to the trivial cocycle $c(r)=\bbbone$ for all $r\in G$. Now define the two cohomology multiplicative groups 
\begin{align*}
H^0(G, X)\vc \{b\in X  \ \mid \  \alpha_r(b)=b, \,\forall r\in G\},\qquad H^1(G,X) \vc Z^1(G,X)/\{\text{coboundaries}\}.
\end{align*}
If $c\in Z^1(G,X)$ and $b \in X$, then $c'(r)=c(r) \,b^{-1}\,\alpha_r(b)$ defines a cocycle cohomologous to $z$. When $c(r)=\vartheta(r) \,\bbbone$, $c'$ is not necessarily a multiple of $\bbbone$, so that the class of $c$ in $H^1(G,X)$ may contain cocycles not multiple of $\bbbone$. Some cohomological aspects of the construction will be exemplify in Proposition~\ref{essential}.
\end{remark}

When the group $G$ is compact, the cocycles are trivial in some generic situations:
\begin{lemma}
Let $X \subset M(\bA)$ be a cone which is a quasi-complete topological space, and $X^+ \vc X \cap M(\bA)^\times_ +$. Let $\alpha$ be an action of $G$ on $X$ such that $\alpha_r : X \to X$ is continuous for any $r\in G$, and let $c \in Z^1(G, X^+)$ be a $\alpha$-one-cocycle. \\Then this cocycle $c$ is trivial when $G$ is compact.
\end{lemma}

\begin{proof}
Since $X$ is quasi-complete, $X$ is a cone, integration preserves positivity, and $c^{-1}$ is continuous, one has
$b \vc \int_G \dd\mu_G(r') \, c(r')^{-1} \in X^+$. 
Since the bounded positive operators $c(r')$ have a bounded inverse, $c(r')\geq \epsilon_{r'}\,\bbbone$ for $\epsilon_{r'}>0$ and $b$ is invertible since $G$ is compact.
Then, by continuity of $\alpha_r$, 
\begin{align*}
\alpha_r(b) 
= \int_G \dd\mu_G(r') \, \alpha_r[c(r')^{-1}]
= \int_G \dd\mu_G(r') \, c(r r')^{-1} \, c(r)
= b \, c(r)
\end{align*}
and $c(r) = b^{-1} \alpha_r(b)$ for any $r \in G$.
\end{proof}

Two actions $\alpha$ and $\alpha'$ of $G$ on $\bA$ are said to be exterior equivalent \cite[8.11.3]{pedersen1979c} 
if there exists a continuous one-cocycle $u\in Z^1(G,UM(\bA))$ such that $\alpha'_r (a)\vc u(r)\,\alpha_r(a)\,u(r)^*$ for any $a \in \bA$ and $r \in G$. Then the crossed product $C^*$-algebras $G\ltimes_\alpha \bA$ and $G\ltimes_{\alpha'} \bA$ are isomorphic as well as the reduced ones. This isomorphism is induced by the morphism of $*$-algebras $\varphi : C_c(G, \bA)_\alpha \to C_c(G, \bA)_{\alpha'}$ defined by $\varphi(f)(r) \vc f(r) \,u(r)^*$, where $C_c(G, \bA)_\alpha$ (resp.  $C_c(G, \bA)_{\alpha'}$) is the vector space $C_c(G, \bA)$ equipped with the product \eqref{eq-product-alpha} and the involution \eqref{eq-involution-alpha} for $\alpha$ (resp. $\alpha'$) \cite[Lemma~2.68]{Will07a}.

\begin{proposition}
\label{exterior}
Suppose that Hypothesis~\ref{hyp-dirac} is satisfied and produces the modular-type $\beta$-twisted spectral triple $(\algA, \caH,\DD)$ of Theorem~\ref{thm-twisted-triple}. Let $u \in Z^1(G,UM(A))$ be a continuous one-cocycle. Then, there exists a modular-type $\beta$-twisted spectral triple $(\algA', \caH,\DD')$ where $\algA' \vc \varphi(\algA) \subset C_c(G, \bA)_{\alpha'}$, $\DD' \vc \hD' \otimes \gamma^1 + \caT'_{\cone,\,\ct} \otimes \gamma^2$ with $(\hD' \,\hxi\,)(r) \vc U_r^*\, u(r)^* D u(r) \,U_r \,\hxi(r)$ and $\caT'_{\cone,\,\ct} \vc \caT_{\cone,\,\ct}$.

\end{proposition}

\begin{proof}
The operator associated to $\alpha'$ in \eqref{eq-UaU*} is $U'_r = \rho[u(r)]\, U_r$ for $r \in G$ which is still a unitary representation, and \eqref{eq-comp-D-alpha} is satisfied for the same operator $D$ with $z'(r) \vc \rho[u(r)]\, z(r)$. Remark that if $\mathcal{C}_r$ is the common core of $U_r\,D\,U_r^*$ and $z(r)^{*-1}\,D\,z(r^{-1})$ then $\rho[u(r)]\,\mathcal{C}_r$ is a common core of $U'^*_r\,D\,U'^*_r$ and $z'(r)^{*-1}\,D\,z'(r^{-1})$. 

Notice that $\rho[p'(r)] = z'(r) \,z'(r)^* = \rho[u(r)]\, \rho[p(r)]\, \rho[u(r)^*] = \rho[p(r)]$ since $p(r)$ is in the center of $M(\bA)$ by Lemma \ref{group of cocycles}.

$\algA'$ is a $*$-subalgebra of $C_c(G, A)_{\alpha'}$, and it has the following properties: $p'$ satisfies \eqref{eq-hyp-p(r)} for $\algA'$, and \eqref{eq-hyp-comm-bounded} is satisfied with $z'$ for any $f' = \varphi(f) \in \algA'$ with the same $M_{f,z}$, since
\begin{align*}
[D, \rho(f'(r)) z'(r) ]\, z'(r)^{-1} &= [D, \rho(f(r) u(r)^* \, u(r)) \,z(r) ]\, z(r)^{-1} \rho (u(r)^{-1})
\\
&= [D, \rho(f(r))\, z(r) ]\, z(r)^{-1} \, \rho(u(r)^{-1}).
\end{align*}
Hypothesis~\ref{hyp-dirac} is then satisfied for these new objects, and Theorem~\ref{thm-twisted-triple} produces a modular-type twisted spectral triple $(\algA', \caH,\DD')$ with the same twist $\beta$, the same operator $\htheta$, and the quoted operators $\DD'$ and $\caT'_{\cone,\,\ct}$.
\end{proof}

\begin{proposition}
\label{Dbar}
In Hypothesis~\ref{hyp-dirac}-\ref{hyp-comp-D-alpha}, suppose that the map $z : G \mapsto \caB(H)$ can be written as $z(r) = \rho[ p(r)^{1/2} u(r)^* ]$ with $u \in Z^1(G,UM(A))$ a continuous one-cocycle, and suppose that the one-cocycle $p$ is trivial.\\
Then, there is a spectral triple $(A, H, \bar{D})$ and an action $\alpha'$ of $G$ on $\bA$, exterior equivalent to $\alpha$, which is implemented by the unitary representation $r \mapsto U'_r$ on $H$ satisfying 
\begin{equation*}
U'_r \,\bar{D} \,U'^*_r = \bar{D} \quad\text{for any $r \in G$.}
\end{equation*}
\end{proposition}
Notice that if $z$ can be written as $z(r) = \rho[ p(r)^{1/2} u(r)^*]$, then the cocycle relations on $p$ and $u$ are natural requirements to be compatible with \eqref{eq-comp-D-alpha}.
\\
The situation of this proposition will be called \emph{inessential} for two reasons. Firstly, this implies that the triple $(\algA', \caH,\DD')$ of Proposition~\ref{exterior} constructed with the spectral triple $(A, H, \bar{D})$ and the action $\alpha'$ reduces to $\hD' = \bar{D}$ and $\caT'_{\cone,\,\ct} \vc (\cone + \ct) \,\bbbone$. Secondly, in the case of the conformal group exposed in Section~\ref{conformal}, this situation means that the group $G$ is inessential in the terminology of conformal groups (see Proposition~\ref{essential}).

\begin{proof}
By assumption, there exists a positive $b \in Z(A)^\times$ such that $p(r) = b^{-1} \alpha_r(b)$ for any $r \in G$, so that $z(r) = b^{-1/2}\, \alpha_r(b^{1/2})\, u(r)^*$. \\
Let $\bar{D} \vc \rho(b^{1/2})\, D \, \rho(b^{1/2})$ acting on $H$. Then $(A, H, \bar{D})$ is a spectral triple: 
since $[b,a]=0$ for any $a \in A$, the operator $[\bar{D}, \rho(a)] = \rho(b^{1/2})\,[D,\rho(a)]\, \rho(b^{1/2})$ is bounded. Moreover, \eqref{compactresol} holds true since $\rho(a)(\bar D-z)^{-1}$ is compact: if $x\vc \rho(b^{1/2})$ and $z,z'\in i\gR^+$,
\begin{equation*}
\rho(a)(xDx-z)^{-1}=x^{-1}\rho(a)(D-z')^{-1}[\bbbone-(x^{-2}z-z')(D-z')^{-1}]^{-1}x^{-1}.
\end{equation*}
When $\vert z'\vert>\norm{x^{-2}}\,\vert z\vert=\norm{\rho(b^{-1})}\,\vert z\vert$, the term in bracket is bounded since
\begin{equation*}
\norm{(x^{-2}z-z')(D-z')^{-1}} \leq \norm{x^{-2}z-z'}\,\tfrac{1}{\vert z'\vert}\leq \norm{x^{-2} \vert\tfrac{z}{z'}\vert- \bbbone}<1
\end{equation*}
because there exists $\epsilon>0$ such that $\epsilon\, \bbbone \leq x^{-2} \leq \norm{x^{-2}}\,\bbbone$, so $0<\epsilon \vert \tfrac{z}{z'}\vert \leq x^{-2} \vert\tfrac{z}{z'}\vert \leq \bbbone$ by hypothesis on $z'$. Since $\rho(a)(D-z')^{-1}$ is compact by hypothesis, so is  $\rho(a)(\bar D-z)^{-1}$.

Define the action $\alpha'_r (a)\vc u(r)\,\alpha_r(a)\,u(r)^*$ and its unitary implementation $U'_r = u(r) \,U_r$ as in Proposition~\ref{exterior}. Then, omitting the $\rho$'s,
\begin{align*}
U'_r\, \bar{D}\, U'^*_r
& = u(r) U_r \, b^{1/2} D b^{1/2}\, U^*_r u(r)^*
= u(r)\, \alpha_r(b^{1/2}) \, U_r D U^*_r \, \alpha_r(b^{1/2})\, u(r)^*
\\
& = u(r) \alpha_r(b^{1/2}) \, b^{1/2} \alpha_r(b^{-1/2}) \,u(r)^* \, D\, u(r) \alpha_r(b^{-1/2}) b^{1/2}   \, \alpha_r(b^{1/2}) u(r)^*\\
&= b^{1/2} \,D \,b^{1/2} = \bar{D}.
\end{align*}
\end{proof}

\subsection{The proof of main result}

We now show that $\beta$ defines an automorphism of $\algA$ and then that the $\beta$-twisted commutators with $\hD$ and $\caT_{\cone,\ct}$ are bounded.

\begin{proposition}
\label{prop-beta}
For any $z \in \gC$, $\beta_z$ defined by $\beta_z(f)(r) \vc p^z(r) f(r)$ for any $r \in G$,  is an automorphism of $C_c(G, \bA)$ satisfying $\beta_z(f^*)=[\beta_{-\bar{z}}(f)]^*$ for $f\in C_c(G, \bA)$ and which is implemented on the Hilbert space $\hH$ by the operator $\htheta^z$ defined by $(\htheta^z \hxi\,)(r) \vc \rho (\alpha_{r^{-1}}[p^z(r)])\, \hxi(r)$  (see also \eqref{thetahat}):
\begin{equation*}
\hrho\,[\beta_z(f)]\,\hxi = \htheta^z \,\hrho(f)\, \htheta^{-z}\, \hxi \quad \text{for any $f \in C_c(G, \bA)$ and $\hxi \in C_c(G, H)$}.
\end{equation*}
Moreover, $\beta \vc \beta_1$ reduces to an automorphism of $\algA$.
\end{proposition}

A priori, $\beta_z$ does not always extend to an automorphism of $G \ltimes_{\alpha, \red} \bA$, as seen in the affine case, see Section \ref{General considerations}.

\begin{proof}
From Lemma~\ref{group of cocycles}, we know that $p^z \in Z^1(G, Z(A)^\times) \subset Z^1(G,M(\bA))$, so that for any $f \in C_c(G, \bA)$, $\beta_z(f) \in C_c(G, \bA)$.

Let us show that $\beta_z$ is an automorphism.  For $f,g \in C_c(G, \bA)$, one has
\begin{align*}
\beta_z(f \star_\alpha g) (r)
&= p^z(r) (f \star_\alpha g)(r)
= p^z(r) \int_G \dd\mu_G(r') \, f(r')\, \alpha_{r'}[g(r'^{-1} r)].
\end{align*}
On the other hand,
\begin{align*}
(\beta_z(f) \star_\alpha \beta_z(g))(r)
&= \int_G \dd\mu_G(r') \,   p^z(r') f(r')\, \alpha_{r'} [p^z(r'^{-1} r) \,g(r'^{-1} r)]
\\
&= \int_G \dd\mu_G(r') \,  p^z(r')\, \alpha_{r'} [ p^z(r'^{-1} r)] \, f(r') \,\alpha_{r'}[g(r'^{-1} r)]
\\
&= \int_G \dd\mu_G(r') \,  p^z(r) f(r') \,\alpha_{r'}[g(r'^{-1} r)],
\end{align*}
so that $\beta_z(f \star_\alpha g) = \beta_z(f) \star_\alpha \beta_z(g)$. Moreover
\begin{align*}
\beta_z(f^*)(r)&=p^z(r)\,f^*(r)
=p^z(r)\,\Delta_G(r^{-1})\,\alpha_r[f(r^{-1})^*]
=\Delta_G(r^{-1})\,\alpha_r[p(r^{-1})^{-z}\,f(r^{-1})^*]\\
&=\Delta_G(r^{-1})\,\alpha_r([p(r^{-1})^{-\bar{z}}\,f(r^{-1})]^*)
=\Delta_G(r^{-1})\,\alpha_r([\beta_{-\bar{z}}(f)(r^{-1})]^*)=[\beta_{-\bar{z}}(f)]^*(r).
\end{align*}
Notice that $(\htheta^{-z} \hxi\,)(r) = \alpha_{r^{-1}}[p(r)^{-z}] \, \hxi(r) = p^z(r^{-1})\, \hxi(r)$ by \eqref{cocycle relation}. 
\\
For $\hxi \in C_c(G, H)$, $f \in C_c(G, \bA)$ and $r \in G$, 
\begin{align*}
(\htheta^z \,\hrho(f)\, \htheta^{-z} \,\hxi\,)(r) 
&=  \int_G \dd\mu_G(r') \, \alpha_{r^{-1}}[p^z(r)] \, \alpha_{r^{-1}}[f(r')] \, p^z(r^{-1} r')\, \hxi(r'^{-1} r)
\\
&= \int_G \dd\mu_G(r') \, \alpha_{r^{-1}} \big[p^z(r) \, \alpha_r[p^z(r^{-1} r')] \,f(r')\big] \, \hxi(r'^{-1} r)
\\
&= \int_G \dd\mu_G(r') \, \alpha_{r^{-1}}[p^z(r') f(r')] \, \hxi(r'^{-1} r)
= \int_G \dd\mu_G(r') \, \alpha_{r^{-1}} [\beta_z(f)(r') ] \, \hxi(r'^{-1} r).
\end{align*}
This proves the implementation.
\\
By Hypothesis~\ref{hyp-dirac}-\ref{hyp-comp-D-alpha}, $r\mapsto p(r)^{\pm 1}\,f(r)$ is in $\algA$ for $f \in \algA$, so $\beta \vc \beta_1$ reduces to an automorphism of $\algA$.
\end{proof}

\begin{remark}\label{rmk-betaz}
Hypothesis~\ref{hyp-dirac}-\ref{hyp-comp-D-alpha} is essential to ensure that $\beta = \beta_1$, and then that $\beta_n = \beta^n$ for any $n \in \gZ$, is an automorphism of $\algA$. It can happen that $\beta_z$ reduces to an automorphism of $\algA$ for a larger class of values of $z \in \gC$. This will be the case for instance in the affine case, see Section \ref{sec-extended-affine}.
\end{remark}

\begin{proposition}
\label{prop-twistcom}
For any $f \in\algA$, the twisted commutator $[\DD, \pi(f)]_\beta \vc \DD\, \pi(f) - \pi (\beta[f]) \DD$ extends to a bounded operator on $\caH$ and $\pi(f) \Dom(\DD) \subset \Dom(\DD)$.
\end{proposition}
As outlined in \cite{FMR14}, the constraint $\pi(f)\Dom \DD \subset \Dom \DD$ can be crucial. Here, it is controlled by the Hypothesis \ref{hyp-dirac}-\ref{preservdom} and the inequality \eqref{eq-hyp-comm-bounded} which are used to proved that the twisted commutators $[\DD,\pi(f)]_\beta$ are bounded on the dense core $Y_\algA\otimes \gC^2$ of $\DD$ (see proof of Lemma \ref{prop-hD}).

\begin{proof}
To show that $[\DD,\pi(f)]_\beta$ is bounded for $f\in \algA$, it is sufficient to prove the boundedness of the twisted commutators of $\hD$ and $\caT_{\cone,\ct}$ on $\hH$, a result obtained in the next two lemmas.\\
Using a modified version of \cite[Proposition 2.1]{FMR14} for twisted commutators, which can be obtained using the following inequality in the original proof (with same notations),
\begin{align*}
\norm{\DD a x_n - \DD a x_m} &= \norm{\beta(a) \DD x_n - \beta(a) \DD x_m + [\DD, a]_\beta x_n - [\DD, a]_\beta x_m}\\
& \leq \norm{\beta(a)} \norm{\DD x_n -\DD x_m} + \norm{[\DD, a]_\beta} \norm{x_n - x_m},
\end{align*}
one gets $\pi(f) \Dom \DD \subset \Dom \DD$ for any $f \in \algA$. 
\end{proof}

\begin{lemma}
\label{prop-hD}
For any $f \in\algA$, the twisted commutator $[\hD, \hrho(f)]_\beta \vc \hD\, \hrho(f) - \hrho (\beta[f]) \hD$ extends to a bounded operator.
\end{lemma}

\begin{proof}
For $f \in \algA \subset C_c(G, A)$ and $\hxi \in Y_\algA$, one has (forgetting the writing of few $\rho$'s)
\begin{align*}
(\hD \,\hrho(f) \,\hxi\,)(r)
&= U_r^* D U_r \int_G \dd\mu_G(r') \, \alpha_{r^{-1}}[f(r')]\, \hxi(r'^{-1} r)
= U_r^* \int_G \dd\mu_G(r') \, D f(r') \,U_r\, \hxi(r'^{-1} r).
\end{align*}
On the other hand, using successively \eqref{eq-UaU*} and \eqref{eq-comp-D-alpha},
\begin{align*}
([\hrho(\beta(f)] \hD \,\hxi \,)(r)
&= \int_G \dd\mu_G(r') \,  \alpha_{r^{-1}}[ p(r') f(r') ] \, U_{r'^{-1} r}^* D U_{r'^{-1} r} \, \hxi(r'^{-1} r)
\\
&= U_r^* \int_G \dd\mu_G(r') \,  p(r') f(r') U_{r'^{-1}}^* D U_{r'^{-1}} U_{ r} \, \hxi(r'^{-1} r)
\\
&= U_r^* \int_G \dd\mu_G(r') \,  p(r') f(r')  {z(r')^*}^{-1} D z(r')^{-1} U_{ r} \, \hxi(r'^{-1} r)
\\
&= U_r^* \int_G \dd\mu_G(r') \,  f(r') z(r') D z(r')^{-1} U_{ r} \, \hxi(r'^{-1} r)
\end{align*}
so that
\begin{align}
\label{eq-hD-hrho(f)}
([\hD, \hrho(f)]_\beta\, \hxi\,)(r) 
&= U_r^* \int_G \dd\mu_G(r') \,  [D, f(r') z(r')] \,z(r')^{-1} U_{ r} \, \hxi(r'^{-1} r)
\\
&= \int_G \dd\mu_G(r') \,  K(r,r') \,\hxi(r') \nonumber
\end{align}
with
\begin{equation*}
K(r,r')\vc \Delta_G(r'^{-1}) \,U_r^* \,[D, f(r r'^{-1}) z(r r'^{-1}) ] \, z(r r'^{-1})^{-1} \, U_r.
\end{equation*}
Then, using equation \eqref{eq-hyp-comm-bounded} of Hypothesis \ref{hyp-dirac}, 
\begin{equation*}
\norm{K(r,r') \,\hxi(r')}_{H} \leq M_{f,z}\, \Delta_G(r')^{-1} \, \chi_{S_f}(r r'^{-1}) \, \norm{\hxi(r')}_{H} \text{ for any $r,r' \in G.$}
\end{equation*}
Thus
\begin{align*}
\norm{[\hD, \hrho(f)]_\beta\,  \hxi \,}^2_{\hH} 
	&= \int_G \dd\mu_G(r) \norm{( [\hD, \hrho(f)]_\beta\,  \hxi \,)(r)}_{H}^2  
	=\int_G \dd\mu_G(r)\,\norm{\int_G \dd\mu_G(r') \,K(r,r')\, \hxi(r')}_{H}^2  
	\\
	&\leq \int_G \dd\mu_G(r) \big[ \int_G  \dd\mu_G(r') \,\norm{K(r,r')\, \hxi(r')}_{H} \big]^2 	\\
	&\leq M^2_{f,z} \int_G \dd\mu_G(r)\, \big[ \int_G  \dd\mu_G(r')\, \Delta_G(r')^{-1} \, \chi_{S_f}(r r'^{-1}) \, \norm{\hxi(r')}_{H} \big]^2 .
\end{align*}
Using the Hölder inequality $(\int \abs*{g_1} \, \abs*{g_2} )^{2} \leq \int \abs*{g_1}^2 \times\int \abs*{g_2}^2$), we get
\begin{equation*}
\big[ \int_G \dd\mu_G(r')\Delta_G(r')^{-1} \, \chi_{S_f}(r r'^{-1}) \, \norm{\hxi(r')}_{H} \big]^2
	\leq 
	\begin{multlined}[t]
		\underbrace{\big[ \int_G \dd\mu_G(r') \,\Delta_G(r')^{-1} \, \chi_{S_f}(r r'^{-1}) \big]}_{= \mu_G(S_f)}
		\\
		\hspace{-0.5cm}
		\times
		\big[ \int_G \dd\mu_G(r')\,\Delta_G(r')^{-1} \, \chi_{S_f}(r r'^{-1}) \, \norm{\hxi(r')}_{H}^2 \big],
	\end{multlined}
\end{equation*}
thus
\begin{align*}
\norm{[\hD, \,\hrho(f)]_\beta\,  \hxi\,}^2_{\hH} 
	&\leq  
	\begin{multlined}[t]
		M^2_{f,z} \, \mu_G(S_f)
		\times \int_G  \dd\mu_G(r')\,\underbrace{\Delta_G(r')^{-1} \,
			\big[ \int_G  \dd\mu_G(r)\,\chi_{S_f}(r r'^{-1})  \big]}_{= \mu_G(S_f)}
			 \norm{\hxi(r')}_{H}^2
	\end{multlined}
	\\
	&\leq M^2_{f,z} \, \mu_G(S_f)^2 \,  \norm{\hxi\,}_{\hH}^2\,.
\end{align*}
Thus $[\hD, \hrho(f)]_\beta$ extends to a bounded operator on $\hH$.
\end{proof}

\begin{lemma}
\label{prop-caT}
For any $f \in\algA$, we have
\begin{align}
\label{combounded}
[\caT_{\cone,\,\ct},\,\hrho(f)]_\beta=\cone\hrho(f)-\cone\hrho(\beta[f]),
\end{align}
so this twisted commutator is a bounded operator.
\end{lemma}

\begin{proof}
The operator $[\bbbone,\,\hrho(f)]_\beta=\hrho(f) - \hrho(\beta[f])$ is bounded and moreover we get the simple relation $[\htheta,\,\hrho(f)]_\beta = \htheta\,\hrho(f)-\hrho(\beta[f])\,\htheta=0.$
\end{proof}

This ends the proof of Theorem \ref{thm-twisted-triple}. Notice that for $n\in \gN$ and $f \in \algA$, one has $[\htheta^n,\,\hrho(f)]_\beta = \htheta^n[\hrho(f)-\hrho(\beta^{1-n}[f])]$, which shows that \eqref{eq-DH} is the most generic polynomial expression in $\htheta$ which ensures that the twisted commutator is bounded.

\subsection{\texorpdfstring{The special case $z(r) = \vartheta(r)\, \bbbone$}{The special case z(r)=vartheta(r) 1}}
\label{sec-special-case-vartheta}

This section is motivated by the scaling automorphisms considered for instance in \cite{Moscovici2010} and our example of the affine group in Sections \ref{sec-example-affine-group} and \ref{sec-example-discrete-affine-group}. 

We suppose that the Hypothesis~\ref{hyp-dirac} are satisfied and assume moreover that $z(r) = \vartheta(r)\, \bbbone$, with $\vartheta(r) \in \gC^\times$ for any $r \in G$. This assumption permits to get new results. The one-cocycle $p$ is then given by $p(r) = \abs{\vartheta(r)}^2 \bbbone$, where $r \mapsto \abs{\vartheta(r)}^2$ is a continuous group homomorphism. We identify $\bbbone \in \caB(H)$ and $\bbbone \in M(\bA)$ and we omit few $\rho$'s.

Since $D z(r) = z(r) D$, we get from \eqref{eq-comp-D-alpha}
\begin{equation}
\label{eq-comp-D-alpha-simplified}
U_r \,D \,U_r^* = [z(r) z(r)^*]^{-1} D = \rho[p(r)^{-1}] D = \abs{\vartheta(r)}^{-2} D.
\end{equation}
 The constraint \eqref{eq-comp-D-alpha-simplified} has some implication: if $(A,H,D)$ is a unital spectral triple, then $D$ has a discrete spectrum and when $\lambda$ is a non-zero eigenvalue of $D$, the continuous map $r\in G \mapsto \abs{\vartheta(r)}^2 \lambda\in\text{Spectrum}(D)$ has a discrete image. This implies that $\vartheta$ is constant on the connected components of $G$. \\
 Moreover, equation \eqref{eq-comp-D-alpha-simplified} implies that  \eqref{Dchap} simplifies to  
 \begin{equation*}
(\hD \,\hxi\,)(r) = \abs{\vartheta(r)}^2 D \,\hxi(r), \text{ for any $\xi \in C_c(G,Y)$}
\end{equation*}
and, since $(\htheta \,\hxi\,)(r) = \abs{\vartheta(r)}^2 \,\hxi(r)$,
\begin{equation*}
(\caT_{\cone,\,\ct}\, \hxi\,)(r) = [\cone  + \ct \abs{\vartheta(r)}^2 ] \,\hxi(r).
\end{equation*}
In particular, $\hD$ and $\caT_{\cone,\,\ct}$ commute and so the criteria of Proposition \ref{prop-compactness} can be used. From \eqref{eq-hD-hrho(f)}, one gets 
\begin{equation}
\label{eq-hD-hrho(f)-simplified}
(\,[\hD, \,\hrho(f)]_\beta \,\hxi\,)(r) = U_r^* \int_G \dd\mu_G(r') \,  [D, f(r')]\, U_{ r} \, \hxi(r'^{-1} r),
\end{equation}
and the relation \eqref{eq-hyp-comm-bounded} is automatically satisfied if we suppose that $a \in A \mapsto [D, \rho(a)] \in \caB(H)$ is continuous: then, by hypothesis on the original triple, $r \in G \mapsto [D, f(r)] \in \caB(H)$ is continuous, and since $f$ is compactly supported, just take
\begin{equation*}
M_{f,z} = \sup_{r \in S_f} \norm*{[D, \rho(f(r))]}_{\caB(H)} < \infty.
\end{equation*}

\begin{proposition}
\label{prop-hJ-operator}
Suppose that Hypothesis~\ref{hyp-dirac} are satisfied and $z(r) = \vartheta(r) \bbbone$ for  $\vartheta(r) \in \gC^\times$. Suppose that the original triple $(A,H,D)$ is real with  a reality operator $J : H \to H$ satisfying 
\begin{equation*}
JD = \epsilon DJ
\quad\text{and}\quad
J^2 = \epsilon',
\quad\text{for $\epsilon, \epsilon' \in \{\pm 1\}$,}
\end{equation*}
and
\begin{equation*}
J \,U_r = U_r \,J \quad\text{for any $r \in G$,}
\end{equation*}
and assume $\cone = \epsilon\, \ct$.

Then the conjugate-linear isometry
\begin{equation*}
(\hJ \,\hxi\,)(r) \vc \Delta_G(r)^{-\onehalf}\, U_{r^{-1}} \,J\, \hxi(r^{-1}),\quad \hxi \in C_c(G,H)
\end{equation*}
defines a reality operator $\caJ \vc \hJ \otimes \bbbone$ for the modular-type $\beta$-twisted spectral triple $(\algA, \pi,\DD)$, which satisfies, for $f,g\in \algA$,
\begin{equation*}
\caJ \,\DD = \epsilon \,\Theta^{-1} \DD\, \caJ,
\quad
\caJ\,^2 = \epsilon',
\quad
[\pi(f), \pi^\circ(g)] = 0,
\quad
\big[ \,[\DD,\pi(f)]_\beta, \,\pi^\circ(g)\big] = 0.
\end{equation*}
where $\Theta$ is defined in \eqref{defTheta} and $\pi^\circ(g) \vc \caJ\, \pi(g) \,\caJ$.
\end{proposition}

\begin{proof}
When $\hxi_1, \hxi_2 \in C_c(G, H)$, one has
\begin{align*}
\left\langle \hJ\, \hxi_1, \hJ \,\hxi_2 \right\rangle_{\hH}
&= \int_G \dd\mu_G(r) \, \left\langle (\hJ \,\hxi_1)(r), (\hJ \,\hxi_2)(r) \right\rangle_{H}
\\
&= \int_G \dd\mu_G(r) \, \Delta_G(r)^{-1} \, \left\langle U_{r^{-1}} J \hxi_1(r^{-1}), U_{r^{-1}} J \hxi_2(r^{-1}) \right\rangle_{H}
\\
&= \int_G \dd\mu_G(r) \, \Delta_G(r)^{-1} \, \left\langle J \,\hxi_1(r^{-1}), J \,\hxi_2(r^{-1}) \right\rangle_{H}
= \int_G \dd\mu_G(r) \, \left\langle J \,\hxi_1(r), J \,\hxi_2(r) \right\rangle_{H}
\\
&= \int_G \dd\mu_G(r) \, \left\langle \hxi_2(r), \hxi_1(r) \right\rangle_{H}
= \left\langle \hxi_2, \hxi_1 \right\rangle_{\hH},
\end{align*}
so that $\hJ$ is a  conjugate-linear isometry. One has
\begin{align*}
(\hJ\, \hJ \,\hxi\,)(r) 
&= \Delta_G(r)^{-\onehalf} U_{r^{-1}} J \Delta_G(r^{-1})^{-\onehalf} U_r J\, \hxi(r)= J^2\, \hxi(r) = \epsilon'\, \hxi(r).
\end{align*}
On the one hand, for $\hxi \in C_c(G, Y)$,
\begin{align*}
(\hJ \,\hD \,\hxi\,)(r)
&= \Delta_G(r)^{-\onehalf} U_{r^{-1}} J \abs{\vartheta(r^{-1})}^2 D \,\hxi(r^{-1})
= \epsilon\, \Delta_G(r)^{-\onehalf} \abs{\vartheta(r^{-1})}^2 U_{r^{-1}} D J \,\hxi(r^{-1})
\\
&= \epsilon \,\Delta_G(r)^{-\onehalf} \abs{\vartheta(r^{-1})}^2 \abs{\vartheta(r^{-1})}^{-2} D U_{r^{-1}} J \,\hxi(r^{-1})
= \epsilon\, \Delta_G(r)^{-\onehalf} D U_{r^{-1}} J\, \hxi(r^{-1}),
\end{align*}
and on the other hand,
\begin{align*}
( \hD \,\hJ \,\hxi\,)(r) 
&= \abs{\vartheta(r)}^2 D \Delta_G(r)^{-\onehalf} U_{r^{-1}} J \,\hxi(r^{-1}),
\end{align*}
so that $\hJ \,\hD = \epsilon \,\htheta^{-1} \,\hD \,\hJ$. Now,
\begin{align*}
(\hJ \,\caT_{\cone,\,\ct}\, \hxi\,)(r)
&= \Delta_G(r)^{-\onehalf} U_{r^{-1}} J \,[\cone  + \ct \abs{\vartheta(r^{-1})}^2 ] \,\hxi(r^{-1})
\\
&= \abs{\vartheta(r)}^{-2} [\epsilon\, \ct\, \abs{\vartheta(r)}^2  + \epsilon \,\cone]\, \Delta_G(r)^{-\onehalf} U_{r^{-1}} J \,\hxi(r^{-1}),
\end{align*}
while, for $\hxi \in C_c(G, H)$,
\begin{align*}
(\caT_{\cone,\,\ct} \,\hJ\, \hxi\,)(r)
&= [ \cone + \ct \abs{\vartheta(r)}^2 ] \,\Delta_G(r)^{-\onehalf} U_{r^{-1}} J \,\hxi(r^{-1}),
\end{align*}
so that $\hJ \,\caT_{\cone,\,\ct} = \epsilon \,\htheta^{-1} \,\caT_{\cone,\,\ct}\, \hJ$. Finally, this proves $\caJ \DD = \epsilon \,\Theta^{-1}\,\DD \caJ$. 

Let us use the notation $\hrho^{\,\circ}(g) \vc \hJ\, \hrho(g) \,\hJ$ for any $g \in C_c(G, \bA)$. Then one has successively (omitting some $\hrho$'s)
\begin{align*}
(\hrho^{\,\circ}(g) \,\hxi\,)(r)
&= \int_G \dd\mu_G(r') \, \Delta_G(r)^{-\onehalf} U_{r^{-1}} J \alpha_r[g(r')] \Delta_G(r'^{-1} r^{-1})^{-\onehalf} U_{r r'} J\, \hxi(r r')
\\
&= \int_G \dd\mu_G(r') \, \Delta_G(r')^{\onehalf}  J g(r') J U_{r'}\, \hxi(r r').
\\
(\hrho(f) \hrho^{\,\circ}(g)\, \hxi\,)(r)
&= \int_G \dd\mu_G(r') \, \alpha_{r^{-1}}[f(r')] \, (\hrho^{\,\circ}(g)\, \hxi\,)(r'^{-1} r)
\\
&= \int_G \dd\mu_G(r') \dd\mu_G(r'') \, \Delta_G(r'')^{\onehalf} \, \alpha_{r^{-1}}[f(r')] \, J g(r'') J \, U_{r''} \, \hxi(r'^{-1} r r'').
\\
(\hrho^{\,\circ}(g) \hrho(f) \,\hxi\,)(r)
&= \int_G \dd\mu_G(r'') \, \Delta_G(r'')^{\onehalf} \, J g(r'') J U_{r''} \, (\hrho(f) \,\hxi\,)(r r'')
\\
&= \int_G \dd\mu_G(r') \,\dd\mu_G(r'') \, \Delta_G(r'')^{\onehalf} \, J g(r'') J U_{r''} \, \alpha_{r''^{-1} r^{-1}}[f(r')] \, \hxi\,(r'^{-1} r r'')
\\
&= \int_G \dd\mu_G(r')\, \dd\mu_G(r'') \, \Delta_G(r'')^{\onehalf} \, J g(r'') J \, \alpha_{r^{-1}}[f(r')] \, U_{r''}  \,\hxi(r'^{-1} r r'').
\end{align*}
The hypothesis $[\rho(a), J \rho(b) J] = 0$ on the original triple, with $a = \alpha_{r^{-1}}(f(r'))$ and $b = g(r'')$, shows that the last  two relations are equal, so that $[\hrho(f), \hrho^{\,\circ}(g)] = 0$.

Because $[\caT_{\cone,\ct},\,\hrho(f)]_\beta=\cone\hrho(f)-\cone\hrho(\beta[f])$, one has $\left[\, [\caT_{\cone,\ct},\,\hrho(f)]_\beta, \,\hrho^{\,\circ}(g) \right] = 0$ by the previous relation. Using \eqref{eq-hD-hrho(f)-simplified}, for $\hxi \in Y_\algA$,
\begin{align*}
([\hD, \hrho(f)]_\beta \,\hrho^{\,\circ}(g) \,\hxi\,)(r)
&= U_r^* \int_G \dd\mu_G(r') \,  [D, f(r')] U_{ r} \, (\hrho^{\,\circ}(g)\hxi\,)(r'^{-1} r)
\\
&\hspace{-1cm}= U_r^* \int_G \dd\mu_G(r') \,\dd\mu_G(r'') \, \Delta_G(r'')^{\onehalf} \, [D, f(r')] U_r \, J g(r'') J U_{r''} \, \hxi(r'^{-1} r r'')
\\
&\hspace{-1cm}= U_r^* \int_G \dd\mu_G(r') \,\dd\mu_G(r'') \, \Delta_G(r'')^{\onehalf} \, [D, f(r')] \, J\, \alpha_{r^{-1}}[g(r'')] J \, U_{r r''} \, \hxi(r'^{-1} r r''),
\end{align*}
and
\begin{align*}
(\hrho^{\,\circ}(g) [\hD, \hrho(f)]_\beta\, \hxi\,)(r)
&= \int_G \dd\mu_G(r'') \, \Delta_G(r'')^{\onehalf} \, J g(r'') J U_{r''} \, ([\hD, \hrho(f)]_\beta \,\hxi\,)(r r'')
\\
&\hspace{-1cm}= \int_G \dd\mu_G(r') \,\dd\mu_G(r'') \, \Delta_G(r'')^{\onehalf} \, J g(r'') J U_{r''} \, U_{r r''}^* [D, f(r')] U_{r r''} \, \hxi(r'^{-1} r r'')
\\
&\hspace{-1cm}= U_r^* \int_G \dd\mu_G(r') \,\dd\mu_G(r'') \, \Delta_G(r'')^{\onehalf} \, J \,\alpha_{r^{-1}}[g(r'')] J \, [D, f(r')] \, U_{r r''} \, \hxi(r'^{-1} r r'').
\end{align*}
The hypothesis $\left[\, [D, \rho(a)], J\rho(b) J\right] = 0$ on the original triple, for $a = f(r')$, $b = \alpha_{r^{-1}}[g(r'')]$, shows that these two relations are equal, so that  $\left[ [\hD, \hrho(f)]_\beta,\,  \hrho^{\,\circ}(g) \right] = 0$, and this proves that $\left[ \,[\DD,\pi(f)]_\beta, \,\pi^\circ(g) \right] = 0$.
\end{proof}

\bigskip
The summability of the modular-type $\beta$-twisted spectral triple $(\algA, \caH, \DD)$ is determined by the traceability of the operator $\Theta^c\, \pi(f) (\bbbone + \DD^2)^{-s/2} $
 
\begin{lemma}
Assume that $\Theta^c\, \pi(f) (\bbbone + \DD^2)^{-s/2}$ is trace-class for some $c\geq 0, s\geq1$ and $f \in \algA$. Then 
\begin{align}
\label{eq-tracecomputationagain}
\Tr \Theta^c\, &\pi(f)\,(\bbbone + \DD^2)^{-s/2}  \nonumber \\
&=2 \int_G \dd\mu_G(r)\, \Delta_G(r)^{-1}\,\abs{\vartheta(r)}^{2c}\, \Tr_{H}\, f(e_G) \left[\bbbone+D^2 +  [\cone + \ct \abs*{\vartheta(r)}^2]^2 \right]^{-s/2}
\end{align}
\end{lemma}
 
\begin{proof}
We have $\Theta^c\, \pi(f) (\bbbone + \DD^2)^{-s/2}=L(c,s, f) \otimes \bbbone_2$ on $\caH$ where
\begin{equation}
\label{L(s,f)}
L(c,s, f) \vc \htheta^c\, \hrho(f) [\bbbone + \hD^2 + \caT_{\cone,\,\ct}^{\,2}]^{-s/2}.
\end{equation}
 Using the fact that $\htheta$, $D$ and $\caT_{\cone,\,\ct}$ commute, one has
\begin{equation*}
\big[ (\bbbone + \hD^{\,2} + \caT_{\cone,\,\ct}^{\,2})^{-s/2} \hxi \,\big](r) 
= \abs{\vartheta(r)}^{-2s} \big[ D^2 + t(r) \big]^{-s/2} \,\hxi(r),\quad \text{ for $\hxi \in \hH$, $r \in G$,}
\end{equation*}
with
\begin{equation*}
t(r) \vc \abs{\vartheta(r)}^{-4} \left[ 1 + [\cone + \ct \abs*{\vartheta(r)}^2]^2 \right]  > 0.
\end{equation*}
Then,
\begin{align*}
(L(c,s,f)\,\hxi\,)(r)
&= \abs{\vartheta(r)}^{2c} \int_G \dd\mu_G(r') \, \alpha_{r^{-1}}[f(r')] \, \abs{\vartheta(r'^{-1} r)}^{-2s} [ D^2 + t(r'^{-1} r) ]^{-s/2} \, \hxi(r'^{-1} r)
\\
&= \int_G \dd\mu_G(r') \,K_{L(c, s, f)}(r,r') \,\hxi(r')
\end{align*}
where $K_{L(c,s, f)}(r,r')$ is the function with values in $\algB(\hH\hspace{0.05cm})$ given by
\begin{equation*}
K_{L(c,s, f)}(r,r') \vc \abs{\vartheta(r)}^{2c} \, \abs{\vartheta(r')}^{-2s} \, \Delta_G(r')^{-1} \, \alpha_{r^{-1}}[ f(r r'^{-1})] \,[D^2 + t(r') ]^{-s/2}.
\end{equation*}
Thus the trace of $L(c,s,f)$ is
\begin{equation*}
\Tr_{\hH} \,L(c,s, f) = \int_G \dd\mu_G(r)\, \abs{\vartheta(r)}^{2c-2s} \, \Delta_G(r)^{-1} \, \Tr_{H}\, \alpha_{r^{-1}}[ f(e_G)]\,[D^2 + t(r)]^{-s/2} .
\end{equation*}
Using \eqref{eq-comp-D-alpha-simplified} and the integral representation
$\abs*{A}^{s} = \Gamma(s/2)^{-1} \int_0^\infty t^{s/2-1} e^{-t A^2} \, \dd t$, 
one can show that 
\begin{equation*}
U_r \, [D^2 + t(r) ]^{-s/2} \,U_r^* =  \abs*{\vartheta(r)}^{2s} \, [D^2 + \abs*{\vartheta(r)}^{4}\, t(r) ]^{-s/2}.
\end{equation*}
Thanks to this relation, \eqref{eq-UaU*} and $\abs*{\vartheta(r)}^{4} \,t(r)= 1 + [\cone + \ct \abs*{\vartheta(r)}^2]^2$, one finally obtains \eqref{eq-tracecomputationagain} since $\Tr \bbbone_2=2$.
\end{proof}

\begin{remark}[About the kernel of $\vartheta$]
\label{pb du noyau 1} 
If the kernel $K$ of $\vartheta$ is such that the $\mu_G(K)$ is infinite, so either replace the group $G$ by the group $G / K$ or add to $\caT_\theta$ a Dirac-like operator $\DD_{K}$ which is sensitive to the space $K$ (for instance using derivations along $K$), in order that the modified integral in \eqref{eq-tracecomputationagain} is well-behaved in the directions of $K$ so that the dimension of $K$ is taken into account at the end. This point is not be further developed here. 
\end{remark}

\section{A construction via a relatively invariant weight}
\label{relatively invariant weight}

In this Section, we propose a variant to Hypothesis~\ref{hyp-dirac}, where the representation $\rho$ stems from a GNS construction based on a weight $\varphi$ on which a part of the hypotheses are reported.

\begin{hypothesis}
\label{hypo-weight}
\begin{enum-hypothesis}
\item $(A,H,D)$ is a spectral triple where $A$ is a dense $*$-subalgebra in a $C^*$-algebra $\bA$, and such that the representation $\rho$ of $\bA$ on $H$ is the GNS representation $\pi_\varphi$ of $\bA$ on $H = H_\varphi$ for a weight $\varphi$ on $\bA$ which is  faithful lower semi-continuous and densely defined (thus semifinite
and the GNS representation is faithful and nondegenerate).

\item The weight $\varphi$ is KMS with respect to a (necessarily unique) norm continuous one-parameter group $\sigma$, which means (see \cite{KV1999}):
\begin{align*}
&\varphi \circ\sigma_t=\varphi,\quad\forall t\in \gR,\\
&\varphi[a^*a] = \varphi [\sigma_{i/2}(a)\,\sigma_{i/2}(a)^*], \quad \text{for any }a\in \Dom(\sigma_{i/2}).
\end{align*}
Let $(\bA, G, \alpha)$ be a $C^*$-dynamical system, such that $\alpha_r(A) = A$ for any $r \in G$.

\item \label{hyp-q} The weight $\varphi$ and the action $\alpha$ of $G$ are related by
\begin{equation}
\label{eq-hyp-omega-q-alpha}
\varphi\circ\alpha_{r^{-1}}[a] 
= \varphi_{q(r)}[a] 
\vc \varphi\left[q(r)\, a\, q(r)^*\right], \quad \text{ for any $r \in G$ and $a\in \bA_+$,}
\end{equation}
where $q$ is a one-cocycle in $Z^1(G,M(\bA)^\times)$ such that the maps $r\in G \to q(r)^{-1}\,a$ are continuous for each $a \in \bA$.
\end{enum-hypothesis}
\end{hypothesis}
Note that if $q(r)=\vartheta(r) \,\bbbone$ where $\vartheta : G \to \gR_+^\times$ is a continuous morphism of groups, see Section \ref{sec-special-case-vartheta}, then $\varphi \circ \alpha_r=\abs*{q(r)}^{-2}\varphi$ yielding the commutation $\alpha_r\circ \sigma =\sigma \circ\alpha_r$.

\subsection{\texorpdfstring{The unitary representation of $G$ based on a weight }{The unitary representation of G based on a weight }}
\label{rhypothesis-weight}

Let us give some general notations related to these structures. As usual we define
\begin{align*}
 \caN_\varphi \vc \{ a \in \bA \mid \varphi(a^* a) < \infty \}, \quad \caM_\varphi \vc  \caN_\varphi^* \,\caN_\varphi= \Span \{a\in \bA^+ \mid \varphi(a) <\infty \}
\end{align*} 
so $\caN_\varphi$ is a left ideal in the multiplier algebra $M(\bA)$.

The GNS construction defines a linear injective map $\Lambda_\varphi : \caN_\varphi \to H$, for which $\Lambda_\varphi(\caN_\varphi)$ is a dense subset of $H$, such that $\pi_\varphi(a)\, \Lambda_\varphi(b) = \Lambda_\varphi(ab)$ for any $a \in \bA, b \in \caN_\varphi$, and one has 
\begin{equation*}
\langle \Lambda_\varphi(a), \Lambda_\varphi(b) \rangle_H \vc \varphi(b^* a)\, \text{ for any $a, b \in \caN_\varphi$.}
\end{equation*}
There is a antiunitary operator $J$ defined by $J \,\Lambda_\varphi(a)=\Lambda_\varphi [\sigma_{i/2}(a)^*]$ for $a \in \caN_\varphi \cap \Dom(\sigma_{i/2})$ (and extended to $H$) and a strictly positive operator $\Delta$ on $H$ such that $\Delta^{it}\Lambda_\varphi(a)=\Lambda_\varphi[\sigma_t(a)]$, for $t\in \gR,\,a\in \caN_\varphi$.

Define $S$ by $S\,\Lambda_\varphi(a)\vc\Lambda_\varphi(a^*)$ for $a\in \caN_\varphi \cap \caN_\varphi^*$, then $S^*\,\Lambda_\varphi(a)=\Lambda_\varphi[\sigma_i(a)^*]$. We have $\Delta=S^*S$, $S=J\Delta^{1/2}=\Delta^{-1/2}J$, $J\Delta^{t}J=\Delta^{-t}$ and $J\Delta^{it}J=\Delta^{it}$ for $t\in \gR$.

To avoid cumbersome notations, we identify from times to times $a \in \caN_\varphi$ with its image $\xi \vc \Lambda_\varphi(a) \in H$. In particular, one uses the notation $\alpha_r(\xi)$ for any $\xi \in \caN_\varphi \subset H$. In the same way, the representation $\pi_\varphi$ can be written $\pi_\varphi(a) \,\xi = a\, \xi$ for any $a \in \bA$ and $\xi \in \caN_\varphi$. 
\begin{proposition}
\label{prop-def-unitary}
Suppose Hypothesis~\ref{hypo-weight} is satisfied. Then for $r \in G$ and $a \in \caN_\varphi$, the operator defined by
\begin{equation}
\label{eq-Ur}
U_r \,\Lambda_\varphi[a] \vc \Lambda_\varphi [\alpha_r(a)\, q(r)^*]
\end{equation}
extends to a strongly continuous unitary representation of $G$ on $H$ which implements the action $\alpha$ on $M(\bA)$ (and so on $\bA$):
\begin{align}
U_r^* \,\pi_\varphi(m) \,U_r &= \pi_\varphi [ \alpha_{r^{-1}}(m) ], \quad \text{$r\in G$ and $m\in M(\bA)$.} \label{eq1-UmU*}
\end{align}
\end{proposition}
\begin{proof}
To show that \eqref{eq-Ur} is well-defined, one first need to show that for any $a \in \caN_\varphi$ and $r \in G$, one has $\alpha_r(a)\, q(r)^* \in  \caN_\varphi$:
\begin{align*}
\varphi\big[\left[\alpha_r(a)\, q(r)^*\right]^* \alpha_r(a)\, q(r)^* \big]
&= \varphi [q(r) \alpha_r(a^*) \alpha_r(a) q(r)^*]
= \varphi[a^* a] <\infty.
\end{align*}
Then, for any $\xi = \Lambda_\varphi(a)$, this computation gives
\begin{align*}
\langle U_r\, \xi, U_r \,\xi \rangle
&= \varphi\big(\left[\alpha_r(a)\, q(r)^*\right]^* \alpha_r(a)\, q(r)^* \big)
= \varphi(a^* a)
= \langle \xi, \xi \rangle
\end{align*}
so that $U_r$ extends to a unitary operator.

Since $q$ is a cocycle, one has, for any $r_1, r_2 \in G$ and $a \in \caN_\varphi$, 
\begin{align*}
U_{r} U_{r'} \,\Lambda_\varphi[a] 
&=  U_{r} \Lambda_\varphi [ \alpha_{r'}(a)\, q(r')^*]
= \Lambda_\varphi\big[ \alpha_{r}[ \alpha_{r'}(a)\, q(r')^* ] \,q(r)^* \big]\\
&= \Lambda_\varphi\big[ \alpha_{r r'}(a) \, ( q(r) \alpha_{r}[q(r')])^* \big]
= \Lambda_\varphi[ \alpha_{r r'}(a) \, q(r r')^*]
= U_{rr'} \,\Lambda_\varphi(a),
\end{align*}
and $r \mapsto U_r$ is a representation of $G$. Thus, $U_r^* \, \Lambda_\varphi(a) = U_{r^{-1}}\, \Lambda_\varphi(a) =  \Lambda_\varphi [\alpha_{r^{-1}}(a)\, q(r^{-1})^*]$.  

So, for any $r \in G$, $a \in \bA$ and $b \in \caN_\varphi$, 
\begin{align*}
U^*_r \,\pi_\varphi(a) \,U_r \,\Lambda_\varphi(b)
&= U^*_r \,\Lambda_\varphi [ a\, \alpha_r(b)\, q(r)^* ]
= \Lambda_\varphi \big[ \alpha_{r^{-1}}[ a\, \alpha_r(b)\, q(r)^*] \, q(r^{-1})^*\big]
\\
&= \Lambda_\varphi\big[ \alpha_{r^{-1}}(a)\, b\, [ q(r^{-1}) \alpha_{r^{-1}}( q(r)) ]^*\big]
= \pi_\varphi[\alpha_{r^{-1}}(a)] \,\Lambda_\varphi(b).
\end{align*}
The extension $\alpha$ to $M(\bA)$ is characterized by $\alpha_r(m\,a) = \alpha_r(m)\,\alpha_r(a)$ for all $m \in M(\bA)$, $a \in \bA$, and \eqref{eq1-UmU*} follows directly from the previous result.

To show the strong continuity of the representation $U$, we adapt an argument given in the proof of \cite[Lemma 3.1]{QV1999}: 
let $\mathcal{G}_\varphi \vc \{\lambda \,\varpi \mid \varpi \in \mathcal{F}_\varphi,\,\lambda \in ]0,1[\}$ be a directed subset of the set $\mathcal{F}_\varphi \vc \{\varpi \mid  \varpi \in \bA^*_+ ,\,\varpi \leq \varphi\}$. By \cite[Proposition 2.4]{Combes}, given $\varpi\in \mathcal{G}_\varphi$, 
there exists a unique $T_\varpi \in \pi_\varphi(\bA)'$ such that $0\leq T_\varpi\leq \bbbone$ and 
$\langle \Lambda_\varphi(a), \,T_\varpi\, \Lambda_\varphi(b)\rangle=\varpi(b^*a)$ for all $a,\,b\in\caN_\varphi$. 
Moreover $(T_\varpi)_{\varpi\in \mathcal{G}_\varphi}$ strongly converges to $\bbbone$ since $\varphi$ is 
lower-semicontinuous. Thus $\{T_\varpi\, \Lambda_\varphi(b) \ \mid \  \varpi\in\mathcal{G}_\varphi,\,b \in \caN_\varphi\}$ is dense in $H$. For $\varpi \in \mathcal{G}_\varphi,\,a,\,b\in\caN_\varphi$, and $r\in G$, we have
\begin{align*}
\langle U_r \,\Lambda_\varphi(a),\,T_\varpi\,\Lambda_\varphi(b) \rangle 
= \langle \Lambda_\varphi [ \alpha_r(a)\,q(r)^*],T_\varpi \,\Lambda_\varphi (b) \rangle 
= \varpi [b^*\,\alpha_r(a)\,q(r)^*].
\end{align*}
Thus the continuity of the map $r\in G \mapsto \langle U_r\, \Lambda_\varphi(a),\,T_\varpi \,\Lambda_\varphi(b) \rangle$ follows from the continuity of $\varpi$ and $\alpha$, the continuity of $r \to q(r^{-1})^{-1}a^*$ by hypothesis and the following
\begin{align*}
\norm{b^*\alpha_r(a) q(r)^*-b^* \alpha_s(a)q(s)^*}&=\norm{q(r) \alpha_r(a^*)\,b-q(s)\alpha_s(a^*)\,b}\\
&= \norm{\alpha_r[q(r^{-1})^{-1}\,a^*]b-\alpha_s[q(s^{-1})^{-1}\,a^*]b}\\
& \leq \norm{(\alpha_r - \alpha_s) [q(r^{-1})^{-1}a^*] + \alpha_s[q(r^{-1})^{-1} a^* - q(s^{-1})^{-1} a^*]} \,\norm{b}\\
&\leq \norm{(\alpha_r-\alpha_s)[q(r^{-1})^{-1}a^*]}\,\norm{b} + \norm{[q(r^{-1})^{-1}-q(s^{-1})^{-1}]a^*}\norm{b}.
\end{align*}

Since $\norm{U_r}=1$, we deduce by the $\epsilon/3$ arguments
\begin{align*}
\vert \langle (U_r-U_s) \,\xi,\,\xi' \rangle\vert  &= \lvert \langle (U_r-U_s) (\xi- \Lambda_\varphi (a)),\xi'\rangle +\langle (U_r-U_s) \, \Lambda_\varphi(a),\xi'-T_\varpi \Lambda_\varphi(b) \rangle\\
&\hspace{5.32cm}+\langle  (U_r-U_s)  \,\Lambda_\varphi(a),T_\varpi \Lambda_\varphi(b) \rangle \rvert\\
&\leq 2 \norm{\xi- \Lambda_\varphi (a)}\,\norm{\xi'} + 2 \norm{\Lambda_\varphi(a)}\,\norm{\xi'-T_\varpi \Lambda_\varphi(b)}\\
& \hspace{5.32cm}+ \varpi[b^*(\alpha_r(a)q(r)^*-\alpha_s(a)q(s)^*)]\,,
\end{align*}
that the map $r \in G\mapsto\langle U_r\, \xi,\,\xi' \rangle$ is also continuous for any $\xi,\xi' \in H$ yielding the weak continuity of the unitary representation $U$, so its strong continuity.
\end{proof}

\medskip
The map $\Lambda_\varphi$ induces a linear injection $\tLambda : C_c(G, \caN_\varphi) \to \hH = L^2(G, \dd\mu_G) \otimes H$ defined by $\tLambda(f)(r) \vc \Lambda_\varphi[f(r)]$ for any $r \in G$. Then the space $\Lambda_\varphi[C_c(G, \caN_\varphi)]$ can be seen as a dense subspace of $\hH$ and, for any $\hxi_i=\tLambda [f_i] \in \Lambda_\varphi[C_c(G, \caN_\varphi)] \subset \hH$, one has
\begin{align*}
\langle \hxi_1, \hxi_2 \rangle_{\hH} 
&= \int_G \dd\mu_G(r) \, \langle \hxi_1(r), \hxi_2(r) \rangle_H 
= \int_G \dd\mu_G(r) \, \varphi \,[ f_2(r)^* \,f_1(r)].
\end{align*}

\begin{remark}
With the notations of Proposition \ref{exterior}, the weight $\varphi$ satisfies equation  \eqref{eq-hyp-omega-q-alpha} for $q'(r) \vc q(r)\, u(r)^*\in M(\bA)$. One has
\begin{align*}
q'(rr')
&= q(r r') \, u(rr')^*
= q(r) \,\alpha_{r}[q(r')] \, [u(r)\,\alpha_{r}(u(r'))]^*
= q(r) \,\alpha_{r}[ q'(r')]\, u(r)^*
\\
&= q'(r)\, u(r) \,\alpha_{r}[q'(r')] \,u(r)^*
= q'(r) \,\alpha'_{r}[q'(r')],
\end{align*}
so that $q'$ is an $\alpha'$-one cocycle in $Z^1(G,M(\bA)^\times)$. Moreover, $r \mapsto q'(r)^{-1}\, a$ is continuous for each $a\in \bA$ since
\begin{align*}
\norm{q'(r)^{-1}\,a -q'(s)^{-1}\,a} 
\leq 
\norm{[u(r)-u(s)] \, q(r)^{-1}\,a} + \norm{u(s)[ q(r)^{-1} -q(s)^{-1}]\,a}.
\end{align*}
\end{remark}

The last proposition shows that 
\begin{corollary}
\label{cor-twisted-triple-weight}
Assume that hypotheses~\ref{hypo-weight}, \ref{hyp-dirac}-\ref{hyp-algA}-\ref{hyp-comp-D-alpha}-\ref{preservdom}-\ref{hyp-comm-bounded} are satisfied for the representation $U\,$of Proposition~\ref{prop-def-unitary}. Then the conclusion of Theorem~\ref{thm-twisted-triple} is valid.
\end{corollary}

\subsection{The dual weight}
\label{The dual weight}

There exists a faithful normal semifinite weight $\tvarphi$ on the von Neumann algebra $\algM \vc \pi_\varphi(\bA)''$ such that $\tvarphi\, \circ\pi_\varphi= \varphi$ and a unique strongly continuous one-parameter modular group $\tilde \sigma_t$ of the weight $\tvarphi$ which satisfies $\tilde \sigma_t\, \pi_\varphi=\pi_\varphi \,\sigma_t$ for $t\in \gR$ and $\tilde \sigma_t(x)=\Delta^{it}\,x\,\Delta^{-it}$ for $x\in \pi_\varphi(\bA)''$. Moreover, $\Lambda_\varphi[\tilde \sigma_t(x)]=\Delta^{it}\Lambda_\varphi(x)$ for $t\in\gR$ and $x\in \caN_{\tvarphi}$ , see \cite{KV1999}.

The action $\alpha$ has an extension $\talpha$ which defines a covariant W$^*$-system $(\algM, G, \talpha)$ namely, $\talpha$ is a strongly continuous action $\talpha$ of $G$ on $\algM$ with
\begin{equation*}
\talpha_r(x) \vc U_r\, x\, U_r^*\,\,\text{ for }x \in \algM.
\end{equation*}
The crossed product von~Neumann algebra $G \ltimes_\talpha  \algM$ is defined as the von~Neumann algebra generated by $\algM$ and $\lambda_G(G)$ acting on $\hH = L^2(G, \dd\mu_G) \otimes H_\varphi$ (notice that $H_\varphi = H_\tvarphi$). We denote by $\hpi_\tvarphi$ the defining representation of $G \ltimes_\talpha  \algM$ on $\hH$, which is \eqref{eq-defrepresentationhrho} for $\rho = \pi_\tvarphi$ (GNS representation of $\algM$ on $H_\tvarphi$).

In order to define the dual weight $\hvarphi$ of $\tvarphi$ on $G \ltimes_\talpha  \algM$, we follow \cite[Chap.~X]{Takes2003}. There, the von~Neumann crossed product algebra is defined using another convention. Let  $C_c(G, \algM)$ be  the space of $\sigma$-strong$^*$ continuous functions $G \to \algM$. Then one defines on $C_c(G, \algM)$ a product and involution \cite[X.1.(15)]{Takes2003}
\begin{align}
\label{eq-prod-conv-takesaki}
(x \ast y)(r) &\vc \int_G \dd\mu_G(r')\, \alpha_{r'}[ x(r r') ] \,y(r'^{-1}),
&
x^\sharp(r) &\vc \Delta_G(r)^{-1} \, \alpha_{r^{-1}}[ x(r^{-1})^*],
\end{align}
for any $x, y \in C_c(G, \algM)$. Equipped with this product and involution, this involutive algebra is denoted by $C_c(G, \algM,\ast)$. Let $\caB_\tvarphi \vc C_c(G, \algM) \cdot \caN_\tvarphi\,$ where $\caN_\tvarphi \vc \{x\in \algM  \mid  \tvarphi(x^*x)<\infty \}$.  The linear map $\ttLambda : \caB_\tvarphi \to \hH$ defined by 
\begin{equation*}
\ttLambda\,(f \cdot x)(r) \vc f(r)\, \Lambda_\tvarphi(x),\,\,\text{ for any }f \cdot x \in \caB_\tvarphi \text{ and }r \in G
\end{equation*}
 is such that $\ttLambda\,(\caB_\tvarphi \cap \caB_\tvarphi^\sharp)$ is dense in $\hH$. \\For $x \in C_c(G, \algM,\ast)$ and $y \in \caB_\tvarphi$, we define the representation \cite[X.1.(24)]{Takes2003}
\begin{equation}
\label{eq-deftpitalpha}
\tpi_\talpha(x)\, \ttLambda(y) \vc \ttLambda(x \ast y),
\end{equation}
and we let $G \ltimes_\talpha  \algM \vc \tpi_\talpha[C_c(G, \algM,\ast)]'' \subset \caB(\hH\,)$. Then $\tpi_\talpha(\caB_\tvarphi \cap \caB_\tvarphi^\sharp)$ generates $G \ltimes_\talpha  \algM$. \\
The dual weight $\hvarphi$ on $G \ltimes_\talpha  \algM$ is defined by \cite[X.1.(42)]{Takes2003}:
\begin{equation}
\label{eq-dualweightTake}
\hvarphi \,[ \tpi_\talpha(x)^*\, \tpi_\talpha(x) ] \vc \tvarphi\, [ (x^\sharp \ast x)(e_G) ],
\end{equation}
and is normal and semifinite since $\tvarphi$ is (\cite[III.3.2.9]{Blac06a}).

From \cite[X.1.(25)]{Takes2003} one has $\tvarphi\,[ (x^\sharp \ast x)(e_G)] = \left\langle \ttLambda(x), \ttLambda(x) \right\rangle_{\hH}\,$ and $\tpi_\talpha$ is a semi-cyclic representation of $G \ltimes_\talpha  \algM$ on $\hH$ associated to the weight $\hvarphi$.

From \cite[Lemma~X.1.13]{Takes2003}, there is a conjugate linear involutive isometry $\hJ$ given by
\begin{equation*}
(\hJ \,\hxi\,)(r) = \Delta_G(r)^{-\onehalf} U_{r^{-1}} J\, \hxi(r^{-1}), \quad \text{for any $\hxi \in \hH$,}
\end{equation*}
where $J$ is the conjugate linear involutive isometry associated to the GNS representation $\pi_\tvarphi$ of $\algM$ on $H_\tvarphi$. This is the expression given in Proposition \ref{prop-hJ-operator}, but there we did not suppose that $J$ was given by the modular theory on the original spectral triple.

In \ref{sec-plancherel-extension-group}, in the specific case of a group $G$ acting on a group $N$, we will explicitly give a unitary correspondence between the representation $\hpi_\varphi$ of $G \ltimes_\alpha  \bA$ on $\hH$ and the GNS representation $\pi_\hvarphi$ on $H_\hvarphi$ associated to the dual weight $\hvarphi$.

\subsection{Derivations and twisted derivations}
\label{sec-candidateforD}

The existence of the unitary representation in Hypothesis~\ref{hyp-dirac}-\ref{hyp-U}-\ref{hyp-comp-D-alpha} is a consequence of Hypothesis~\ref{hypo-weight}, but, to fulfill Hypothesis~\ref{hyp-dirac}, one also has to assume Hypothesis~\ref{hyp-dirac}-\ref{hyp-comp-D-alpha} and \ref{hyp-dirac}-\ref{hyp-comm-bounded}. Here, we consider another hypothesis which can replace Hypothesis~\ref{hyp-dirac}-\ref{hyp-comp-D-alpha} and \ref{hyp-dirac}-\ref{hyp-comm-bounded}.

The relation \eqref{eq-comp-D-alpha} is a compatibility condition between $D$, the action of $G$, and the cocycle $z$. It is of the same kind as the second relation in \cite[eq. (4.1)]{Moscovici2010}. Another equivalent relation can be assumed (see \eqref{commutationalphadelta} below), when in the triple $(A,H,D)$, $D$ is based on derivations on $A$.  
We begin with some remarks on the twisting of derivations using cocycles.

\begin{lemma}
\label{lem-twistedderivation}
We assume that $A$ and $\algA = C_c(G, A)$ are as in Lemma~\ref{lem-Afrechet}. \\
Suppose that there exists a (positive) continuous $\alpha$-one-cocycle valued in $Z(A)^\times$, $r \mapsto p(r)$, such that for a continuous derivation $\delta$ of the algebra $A$
\begin{align}
\label{commutationalphadelta}
p(r)\,\alpha_r \circ\delta=\delta\circ \alpha_r\,\text{ for any } r\in G.
\end{align}
Then, for any $\alpha$-one-cocycle valued in $Z(A)^\times$, $r \mapsto c(r)$, such that $c$ and $c^{-1}$ are continuous,  $\hdelta_c$ defined on $f\in \algA$ by 
\begin{align}
\label{deltahat}
(\hdelta_c \,f)(r)\vc c(r)^{-1}\,\delta [c(r)\,f(r)], \, r\in G
\end{align}
is a $\beta$-twisted derivations on $\algA$, where $\beta$ is defined as in \eqref{beta}.
\end{lemma}
Notice that $r \mapsto c(r)f(r)$ is in $\algA$ for any $f \in \algA$, so that $\hdelta_c \,f \in \algA$.

\begin{proof}
One has
\begin{align*}
\hdelta_c (f \star g)(r)
&= \int_G\dd \mu(r')\,\delta\big[f(r')\,\alpha_{r'} [g(r'^{-1}r)] \, c(r) \big] c(r)^{-1},
\end{align*}
and, using $\alpha_{r'}(c(r'^{-1}r)) = c(r) c(r')^{-1}$, 
\begin{align*}
[\hdelta_c(f) \star g ](r)
&= \int_G\dd \mu(r')\,\delta [f(r')\,c(r')] c(r')^{-1}\, \alpha_{r'} [g(r'^{-1}r)],
\\
[\beta(f) \star \hdelta_c(g) ](r)
&= \int_G\dd \mu(r')\,f(r')\, p(r')\, \alpha_{r'}\circ \delta\big[ g(r'^{-1}r)  \, c(r'^{-1}r) \big] \alpha_{r'}(c(r'^{-1}r)^{-1})
\\
&= \int_G\dd \mu(r')\,f(r') \, \delta\big[ \alpha_{r'} [g(r'^{-1}r)] \, c(r) c(r')^{-1} \big] c(r') c(r)^{-1}.
\end{align*}
With the simplified notations $a \vc f(r')$, $b \vc \alpha_{r'} [g(r'^{-1}r)]$, $c \vc c(r)$ and $c' \vc c(r')$, the integrand in the sum of the last two expressions is
\begin{equation*}
\delta[ a c'] \, c'^{-1} b + a \, \delta[b c c'^{-1}] \, c' c^{-1}
= \delta[a c' \, b c c'^{-1}] \, c^{-1}
= \delta[a b c] \, c^{-1}
\end{equation*}
which is the integrand for $\hdelta_c (f \star g)(r)$.
\end{proof}

This twisting of derivations by cocycles is related to \cite{FaKha2010,FaKhaNiRo2013}. When $c(r) = \vartheta(r) \,\bbbone$ (see Section~\ref{sec-special-case-vartheta}), this twisted derivation reduces to 
\begin{equation}
\label{eq-hdeltadef}
(\hdelta_c \,f)(r) \vc \delta f(r).
\end{equation}

Let us now consider the situation of Hypothesis~\ref{hypo-weight}, with the supplementary assumption: 

\begin{lemma}
Assume that $D$ is defined on the dense subspace $A \subset H$ by a derivation $\delta$ on $A$, $D\xi = \delta \xi$, for $\xi\in A$, and assume that hypotheses of Lemma~\ref{lem-twistedderivation} are satisfied for the cocycles $p$ and $c(r) = q(r)^*$, where $q$ satisfies also Hypothesis~\ref{hypo-weight}-\ref{hyp-q}. Then
\begin{equation*}
[\hD, \hrho(f)]_\beta = \hrho\,[\hdelta_c(f)] \quad\text{for any $f \in \algA$.}
\end{equation*}
\end{lemma}

\begin{proof}
For any $f \in \algA = C_c(G, A)$ and $\hxi \in Y_\algA \subset \hH$, one has, omitting $\rho$,
\begin{align*}
(\hD \hrho(f) \,\hxi\,)(r)
&= \int_G \dd\mu_G(r')\, U_r^* \delta\left[ f(r')\, U_r \hxi(r'^{-1} r) \right]
\\
&= \int_G \dd\mu_G(r')\, U_r^* \delta\left[ f(r')\, q(r)^* \alpha_r[ \hxi(r'^{-1} r) ]  \right],
\\
(\hrho[\beta(f)] \hD\, \hxi\,)(r)
&= \int_G \dd\mu_G(r')\, U_r^* \, p(r')\, f(r')  \, U_{r'} D U_{r'^{-1} r} \hxi(r'^{-1} r)
\\
&= \int_G \dd\mu_G(r')\, U_r^* \, f(r')\, p(r')\, q(r')^* \alpha_{r'} \circ \delta\left[ q(r'^{-1} r)^* \, \alpha_{r'^{-1} r}[ \hxi(r'^{-1} r)] \right]
\\
&= \int_G \dd\mu_G(r')\, U_r^* \, f(r')\, q(r')^*\, \delta\left[ \alpha_{r'}[q(r'^{-1} r)^*] \, \alpha_{r} [ \hxi(r'^{-1} r)] \right].
\end{align*}
Using $\alpha_{r'}[q(r'^{-1} r)^*] = q(r)^* \, q(r')^{* -1}$ and the simplified notations $a \vc f(r')$, $\xi \vc \alpha_{r} [\hxi(r'^{-1} r)]$, $q^* \vc q(r)^*$ and $q'^* \vc q(r')^*$, the integrand in $([\hD, \hrho(f)]_\beta \,\hxi\,)(r)$ (without the first $U_r^*$) is 
\begin{align*}
\delta(a\, q^* \xi) - a\, q'^*\, \delta[ q'^{* -1} \, q^* \, \xi ] 
= \delta(a \, q'^*) \, q'^{* -1} \, q^* \, \xi
\end{align*}
so that
\begin{equation*}
([\,\hD, \hrho(f)]_\beta\, \hxi\,)(r) = \int_G \dd\mu_G(r')\, U_r^* \, \delta[ f(r')\, q(r')^*] \,q(r')^{* -1} \, U_r \, \hxi(r'^{-1} r)
= (\hrho\,[\hdelta_c(f)] \,\hxi\,)\,(r).
\end{equation*}
\end{proof}

This result gives a natural $\hD$ associated to a derivation. If there are several derivations on $A$, we can use the gamma matrices as in \eqref{eq-DiracG} but in appropriate dimension to define $\DD$.

\section{\texorpdfstring{$C^*$-algebra of a semidirect product of groups}{C*-algebra of a semidirect product of groups}}
\label{sec-C*-algebra-semidirect-product}

In this section, $G$ and $N$ are locally compact second countable Hausdorff groups with an action $\varsigma : G \to \Aut(N)$ such that $(r,n) \mapsto \varsigma_r(n)$ is continuous from $G \times N$ to $N$. \\
Denote by $\mu_G$ and $\mu_N$ the Haar measures on $G$ and $N$ respectively, and by $\Delta_G$ and $\Delta_N$ their modular functions. 

We are interested into applying our construction to the action of $G$ on $C^*_\red(N)$. In that case, the final modular-type twisted spectral triple is a spectral triple on the $C^*$-algebra of the semidirect product $K \vc G \ltimes_\varsigma N$.
\\
We suppose that the original spectral triple is given by $(A_N \vc C_c(N), H \vc L^2(N, \dd\mu_N), D)$ for the left regular representation $\lambda_N$ of $\bA \vc C^*_\red(N)$ on $L^2(N, \dd\mu_N)$, and for an operator $D$ not specified here.

\subsection{\texorpdfstring{The crossed products $G \ltimes_{\alpha, \red} C_\red^*(N)$ and $C^*_\red(G \ltimes_\varsigma N)$}{The crossed product GxC*(N)}}
\label{subsec-crossproductalgebra}

The space $A_N \vc C_c(N)$ of compactly supported functions on $N$ is a $*$-algebra for the product and the involution given for $f_N, g_N \in A_N$ by
\begin{align*}
(f_N \star_N g_N)(n) &\vc \int_N \dd\mu_N(n') \, f_N(n') g_N(n'^{-1} n),
&
f_N^*(n) &\vc \Delta_N(n)^{-1} \overline{f_N(n^{-1})},
\end{align*}
This $*$-algebra generates the $C^*$-algebra $\bA \vc C^*_\red(N)$. The left regular representation  $\rho = \lambda_N$ of $\bA$ on $H \vc L^2(N, \dd\mu_N)$ is
\begin{align}
\label{lambdaN}
[\lambda_N(f_N)\,\xi](n) \vc \int_N \dd\mu_N(n') \, f_N(n')\, \xi(n'^{-1} n),\quad \text{ for any $f_N \in A_N, \xi \in H$.}
\end{align}

According to \cite[Section~3.3]{Will07a}, there exists a continuous homomorphism $\nu : G \to \gR^+$  such that 
\begin{equation}
\label{eq-definition-nu}
\nu(r) \int_N \dd\mu_N(n) \, f_N \big(\varsigma_r(n)\big)  = \int_N \dd\mu_N(n) \, f_N(n),\quad \text{ for any $f_N\in A_N,\, r\in G$.}
\end{equation}

The continuous action of $G$ on $N$ induces a $C^*$-dynamical system $(\bA, G, \alpha)$ where the continuous morphism $\alpha : G \to \Aut(\bA)$ is given by 
\begin{equation}
\label{eq-def-alpha-group}
\alpha_r(f_N)(n) \vc \nu(r)^{-1}\, f_N[\varsigma_{r^{-1}}(n)],\quad \text{for any $f_N \in A_N$}
\end{equation}
(see \cite[Prop~3.11]{Will07a} with $A = \gC$ and $\beta = \Id$). To show that $\alpha_r$ is an automorphism of the $*$-algebra $A_N$ requires the relation
\begin{align}
\label{eq-inv-DeltaN}
\Delta_N[\varsigma_r(n)] = \Delta_N(n),\quad \text{for any $n \in N$},
\end{align}
which we suppose to hold. The relation \eqref{eq-definition-nu} reflects an invariance of the Haar measure of $N$ by the action of $G$ on $A_N$: 
\begin{equation*}
\int_N \dd\mu_N(n)\,\alpha_r(f_N)(n) = \int_N \dd\mu_N(n)\,f_N(n).
\end{equation*}
Here, the strongly continuous unitary representation $U_r$ satisfying \eqref{eq-UaU*} is given by
\begin{equation*}
U_r\,\xi(n)=\nu(r)^{-1/2}\,\xi[\varsigma_{r^{-1}}(n)].
\end{equation*}

Following the construction exposed in Section~\ref{sec-motivation-main-results}, the extended spectral triple is defined on the Hilbert space
\begin{equation*}
\hH \vc L^2(G, \dd\mu_G) \otimes L^2(N, \dd\mu_N) \simeq L^2(G \times N, \dd\mu_G \,\dd\mu_N)
\end{equation*}
and the representation of $C_c(G, \bA)$ on $\hH$ is given by
\begin{equation*}
\big(\hrho(f)\,\hxi\,\big)(r) = \int_{G}\,\dd\mu_G(r') \lambda_N\big( \alpha_{r^{-1}} [f(r')] \big) \,\hxi(r'^{-1} r),
\end{equation*}
for any $f \in C_c(G, \bA)$ and $\hxi \in \hH$,
and
\begin{align}
\big(\hrho(f)\,\hxi\;\big)(r)(n) 
\label{eq-rep-GxN}
&= \nu(r)\, \int_{G \times N} \dd\mu_G(r') \,\dd\mu_N(n')\, f(r')(\varsigma_r(n')) \,\hxi(r'^{-1} r)(n'^{-1} n) 
\\
&= \int_{G \times N} \dd\mu_G(r') \,\dd\mu_N(n')\,  f(r')(n') \,\hxi(r'^{-1} r)\big(\varsigma_{r^{-1}}(n'^{-1}) n\big)
\nonumber
\end{align}
for any $f \in C_c(G, C_c(N))$ and $\hxi \in C_c(G, H)$. By completion, the representation $\hrho$ of $C_c(G, \bA)$ defines the crossed product $C^*$-algebra $\bB=G \ltimes_{\alpha, \red} \bA$.

We equip $A_N \subset \bA$ with the induced topology of $\bA$. It is a dense $*$-subalgebra of $\bA$. The vector space $C_c(G, C_c(N))$ is dense in $C_c(G, \bA)$, but the product $f \star_\alpha g$ defined by \eqref{eq-product-alpha} of any elements $f,g \in C_c(G, C_c(N))$ is not necessary in $C_c(G, C_c(N))$. As an element in $C_c(G, \bA)$, this product is given by the double integral
\begin{align}
(f \star_\alpha g)(r)(n) &= \int_{G \times N}\, \dd\mu_G(r')\,\dd\mu_N(n') \, f(r')(n') \,\alpha_{r'} [g(r'^{-1} r)](n'^{-1} n)
\nonumber
\\
\label{eq-prod-beta}
& = \int_{G \times N} \dd\mu_G(r')\,\dd\mu_N(n') \, f(r')(n') \, g(r'^{-1} r)\big(\varsigma_{r'^{-1}}(n'^{-1} n)\big)\,\nu(r')^{-1},
\end{align}
and the involution of $f \in C_c(G, C_c(N))$, defined in \eqref{eq-involution-alpha}, is
\begin{equation}
\label{eq-involution-beta}
f^*(r)(n) = \nu(r)^{-1}\,\Delta_G(r)^{-1}\,  \Delta_N(n)^{-1} \,\overline{f(r^{-1})\big(\varsigma_{r^{-1}}(n^{-1})\big)}.
\end{equation}

There is a natural isomorphism $C_\red^*(G \ltimes_\varsigma N) \simeq G \ltimes_{\alpha, \red} C_\red^*(N)$ (see \cite[Prop.~3.11]{Will07a}) which permits to consider another description of the Hilbert space on which the modular-type twisted spectral triple of Theorem~\ref{thm-twisted-triple} is defined. Moreover, describing this isomorphism at the level of the spaces $C_c(G \ltimes_\varsigma N)$ and $C_c(G, C_c(N))$ will produce a natural candidate for the algebra $\algA \subset C_c(G, C_c(N))$. This description will also be used in \ref{sec-example-affine-group}.

Let $K \vc G \ltimes_\varsigma N$ be the semidirect product group, with elements $(r,n) \in G \times N$, product $(r,n)(r',n') \vc \big(r r', n \,\varsigma_r(n')\big)$, unit $e_K=(e_G, e_N)$, and inverse $(r,n)^{-1} = \big( r^{-1}, \varsigma_{r^{-1}}(n^{-1}) \big)$. Let $A_K \vc C_c(K)$ and $\bA_K \vc C^*_\red(K)$. 

The Haar measure on $K$ is given by
\begin{equation*}
\dd\mu_K(r,n) = \nu(r)^{-1} \,\dd\mu_G(r) \,\dd\mu_N(n),
\end{equation*}
and its modular function is 
\begin{equation*}
\Delta_K(r,n) = \nu(r)^{-1} \,\Delta_G(r)\, \Delta_N(n).
\end{equation*}
The product of $f_K,g_K \in A_K$ is
\begin{equation}
\label{eq-prod-K}
(f_K \star_K g_K)(r,n) = \int_{G \times N} \dd\mu_G(r') \,\dd\mu_N(n')\,f_K(r',n') \,g_K\big(r'^{-1} r, \varsigma_{r'^{-1}}(n'^{-1} n)\big)\, \nu(r')^{-1},
\end{equation}
while the involution of $f_K \in A_K$ is nothing else but
\begin{equation}
\label{eq-involution-K}
f_K^*(r,n) = \nu(r)\, \Delta_G(r)^{-1}\, \Delta_N(n)^{-1} \,\overline{f_K(r^{-1}, \varsigma_{r^{-1}}(n^{-1}))}.
\end{equation}

The algebra $C_\red^*(K)$ is obtained through the left regular representation $\lambda_K$ of $A_K$ on $L^2(K, \dd\mu_K)$:
\begin{equation*}
(\lambda_K(f_K)\, \xi_K)(r,n) \vc \int_{G \times N} \dd\mu_G(r') \,\dd\mu_N(n')\,f_K(r',n') \,\xi_K \big(r'^{-1} r , \varsigma_{r'^{-1}}(n'^{-1} n)\big) \,\nu(r')^{-1} 
\end{equation*}
for any $f_K \in A_K$ and $\xi_K \in L^2(K, \dd\mu_K)$.
\medskip

The products $\star_K$ and $\star_\alpha$ defined in \eqref{eq-prod-K} and \eqref{eq-prod-beta} respectively are formally the same, while the corresponding involutions, defined in \eqref{eq-involution-K} and  \eqref{eq-involution-beta}, are not the same but nevertheless, there is an isomorphism $\Phi:\,C_\red^*(K) \to G \ltimes_{\alpha, \red} C_\red^*(N)$ which is induced by the map (also denoted by) $\Phi : A_K \to C_c(G, C_c(N))$ defined by
\begin{equation*}
f(r)(n) = \Phi (f_K)(r)(n) \vc \nu(r)^{-1} f_K(r,n),
\end{equation*}
which relates the equations \eqref{eq-prod-K} and \eqref{eq-prod-beta} (resp. \eqref{eq-involution-K} and  \eqref{eq-involution-beta}). Notice that one has to prove that $r \mapsto \Phi (f_K)(r) \in C_c(N) \subset C_\red^*(N)$ is continuous: this can be easily performed using the norm topology on $L^1(N, \dd\mu_N)$.

There is a corresponding unitary operator $W : L^2(K, \dd\mu_K) \to L^2(G \times N, \dd\mu_G\, \dd\mu_N)$ given by
\begin{equation}
\label{eq-defUGbeta}
\hxi(r)(n) = (W\hxi_K)(r)(n) \vc \hxi_K\big(r, \varsigma_r(n)\big), \quad \text{for any $\hxi_K \in C_c(K) \subset L^2(K, \dd\mu_K)$,}
\end{equation}
which intertwines the representation $\lambda_K$ of $C_\red^*(K)$ on $L^2(K, \dd\mu_K)$ and the representation $\hrho \circ \Phi$ of $G \ltimes_{\alpha, \red} C_\red^*(N)$ on $\hH = L^2(G \times N, \dd\mu_G\, \dd\mu_N)$:
\begin{equation*}
\hrho \circ \Phi(f_K) (W\hxi_K) = W\lambda_K(f_K)  \,\hxi_K.
\end{equation*}
The relations between the algebras $C_\red^*(K)$, $G \ltimes_{\alpha, \red} C_\red^*(N)$ and their representations can be summarized in the commutative diagram
\begin{equation}
 \label{eq-diagram}
\raisebox{0.4\totalheight}{\xymatrix{
C_\red^*(K) \ar[r]^-{\Phi} \ar[d]_-{\lambda_K} 
&  G \ltimes_{\alpha, \red} C_\red^*(N) \ar[d]^-{\hrho} \ar@{=}[r]
&  G \ltimes_{\alpha, \red} C_\red^*(N) \ar[d]^-{\hrho_\GNS} 
\\
L^2(K, \dd\mu_K) \ar[r]^-{W}  
& L^2(G \times N, \dd\mu_G\, \dd\mu_N) \ar[r]^-{V}  
& L^2(G \times N, \dd\mu_\GNS).
}}
\end{equation}
where $V$ is given by
\begin{equation*}
(V\\\hxi\,)(r)(n) \vc \nu(r)^{-1} \, \hxi(r)\big(\varsigma_{r^{-1}}(n)\big),
\end{equation*}
and the last column will be described in Section~\ref{sec-plancherel-extension-group}.

Notice that the composite unitary operator  $VW : L^2(K, \dd\mu_K) \to L^2(G \times N, \dd\mu_\GNS)$ is nothing else than
\begin{equation*}
(VW\hxi_K)(r)(n) = \nu(r)^{-1}\, \hxi_K(r,n),
\end{equation*}
which is the same expression as the one defining $\Phi$.

The space
\begin{equation}
\label{group-ongroup}
\algA \vc \Phi(C_c(G\times N)) \subset C_c(G, C_c(N))
\end{equation}
is a dense $*$-subalgebra in $G \ltimes_{\alpha, \red} C_\red^*(N)$, and it is a good candidate to enter into the construction of the spectral triple of Theorem~\ref{thm-twisted-triple}. By construction, $\algA$ is the subspace of functions in $C_c(G, C_c(N))$ which are continuous in both variables. 

Remark that the existence of a finitely summable spectral triple on a group $C^*$-algebra is not guaranteed: Connes proved in \cite{Connes89} that, for an infinite discrete nonamenable group $N$, there exist no finitely summable spectral triples on $\bA=C^*_{red}(N)$. However, there still exist $\theta$-summable spectral triples on $\bA$, that is, there exists $t_0\geq 0$ such that $\Tr\, e^{-tD^2}<\infty$ for all $t>t_0$.

The construction presented here can also be used to transport the extended modular-type twisted spectral triple of Theorem~\ref{thm-twisted-triple} defined on a crossed product $C^*$-algebra, to a spectral triple on a group $C^*$-algebra. This will be done in Section~\ref{sec-extended-affine} in order to compare our extension with a spectral triple already proposed for the affine group.

\subsection{The Plancherel weight and its dual weight}
\label{sec-plancherel-extension-group}

For any locally compact group $N$, the Plancherel weight is defined on positive elements of the von Neumann algebra $C_\red^*(N)''$ generated by the left regular representation $\lambda_N$ of $N$ by 
\begin{equation*}
\varphi_N(f) \vc 
\begin{cases}
\norm*{g}^2_{L^2(N)}  &\text{ when $f^{\onehalf}=\lambda_N(g)$ for some left bounded function  $g\in L^2(N,\dd\mu_N)$},\\
\infty &\text{ otherwise},
\end{cases}
\end{equation*}
and $g\in L^2(N,\dd\mu_N)$ is called left bounded if $f\in A_N\subset  L^2(N,\dd\mu_N) \mapsto g\star_N f \in L^2(N,\dd\mu_N)$ can be extended to a bounded operator on $L^2(N,\dd\mu_N) $. \\
This weight is faithful and semifinite on $C_\red^*(N)''$ \cite[Chap. VII, Theorem 2.5]{Takes2003}.

On the generators $f_N \in A_N$ of $C_\red^*(N)$, it reduces simply to $\varphi_N(f_N) = f_N(e_N)$. This is a densely defined lower-semicontinuous weight on $C_\red^*(N)$ and because
\begin{equation*}
\varphi_N(f_N^* \star_N g_N) = \int_N  \dd\mu_N(n)\,\overline{f_N(n)} \, g_N(n) ,
\end{equation*}
its GNS construction produces the Hilbert space $L^2(N, \dd\mu_N)$ and the associated non-degene\-rate representation is the faithful left regular representation $\lambda_N$ of $C_\red^*(N)$. The Plancherel weight is a KMS-weight for the automorphism
\begin{equation}
\label{sigmaN}
\sigma_N(f_N)(n) \vc \Delta_N(n)^{-1} f_N(n), \quad \text{for any $f_N \in A_N,\,n \in N$.}
\end{equation}
One can also check directly that $\sigma_N(f_N)^* = \sigma_N^{-1}(f_N^*)$.

\begin{proposition}
\label{prop-hyp-satisfied-group-product}
For any $f_N \in A_N$ and $r \in G$, one has
\begin{equation*}
\varphi_N[f_N] = \varphi_N [\nu(r)\, \alpha_r(f_N)].
\end{equation*}
Moreover, $\sigma_N$ and $\alpha_r$ commute for $r \in G$.
\\
In other words, Hypothesis~\ref{hypo-weight} is satisfied for the one-cocycle $q^2(r) \vc \nu(r)\, \bbbone$ where $\bbbone$ is the unit in $M(\bA)$. 
\end{proposition}

\begin{proof}
The first relation follows from \eqref{eq-def-alpha-group}. The assumption \eqref{eq-inv-DeltaN} implies that $\sigma_N$ defined in \eqref{sigmaN} and $\alpha_r$ commute for any $r \in G$. Finally, $q$ is a one-cocycle since $\nu : G \to \gR_+^\times$ is a morphism of groups. 
\end{proof}

The invertible conjugate-linear isometry
\begin{equation}
\label{eq-modularconjugationN}
(J_N\,\xi)(n) \vc \Delta_N(n)^\onehalf \,\xi^*(n) = \Delta_N(n)^{-\onehalf} \,\overline{\xi(n^{-1})}
\end{equation}
defined on $\xi \in A_N \subset L^2(N,\dd\mu_N)$ is the canonical modular conjugation on $L^2(N,\dd\mu_N)$ with $J_N^2 = \bbbone$, and it intertwines the left regular representation $\lambda_N$ and the right regular representation $R_{\varphi_N}$:
$J_N\, \lambda_N(f_N^*) \,J_N = R_{\varphi_N}(f_N)$.

\medskip
Let us now consider the dual weight of $\varphi_N$. The map $v : C_c(G \times N)_\star \to C_c(G \times N)$ defined by $v[x](r) \vc \alpha_r[x(r)]$, that is $v[x](r)(n) =  \nu(r)^{-1}\, x(r)\big(\varsigma_{r^{-1}}(n)\big)$, intertwines  the two products and involutions given in \eqref{eq-prod-conv-takesaki} and (\ref{eq-product-alpha}-\ref{eq-involution-alpha}), and satisfies $\hrho \circ v = \pi_\alpha$, where $\pi_\alpha$ is the representation of $C_c(G \times N)$ (equipped with the product and involution \eqref{eq-prod-conv-takesaki}) on $\hH$ defined as in \eqref{eq-deftpitalpha}. The dual weight $\hvarphi$ associated to $\varphi_N$ is defined by \eqref{eq-dualweightTake} for any $x \in C_c(G \times N)$, so that, with $f = v[x] \in C_c(G \times N)$, it becomes
\begin{equation*}
\hvarphi \,[ \hrho(f)^* \, \hrho(f) ] = \varphi_N\,[(f^* \star_\alpha f)(e_G)].
\end{equation*}
We define $\hvarphi_N$ on $G \ltimes_{\alpha, \red} C_\red^*(N)$ by $\hvarphi_N(f^* \star_\alpha f) \vc \hvarphi \,[ \hrho(f)^* \, \hrho(f) ] = \varphi_N\,[(f^* \star_\alpha f)(e_G)]$. Since
\begin{align*}
(f^* \star_\alpha g)(e_G)= \int_G \dd\mu_G(r)\,\Delta_G(r)^{-1} \, \alpha_r[ f(r^{-1})^*] \star_N \alpha_r[ g(r^{-1})]= \int_G \dd\mu_G(r)\, \alpha_{r^{-1}}[ f(r)^* \star_N g(r)],
\end{align*}
using $\varphi_N [\alpha_{r^{-1}}(f_N)] = \nu(r) \, \varphi_N[f_N]$, we get
\begin{align*}
\hvarphi_N( f^* \star_\alpha g ) 
&= \int_G  \dd\mu_G(r) \,\varphi_N \circ \alpha_{r^{-1}} [ f(r)^* \star_N g(r)]\\
&= \int_{G \times N} \, \dd\mu_G(r)\,  \dd\mu_N(n) \, \nu(r)  \, \overline{f(r)(n)} \, g(r)(n).
\end{align*}
The GNS construction for $\hvarphi_N$ produces the Hilbert space of square integrable functions on $G\times N$ for the measure $(r,n) \mapsto \dd\mu_\GNS(r,n) \vc \nu(r) \, \dd\mu_G(r)\, \dd\mu_N(n)$ and the associated representation is
\begin{equation*}
[\hrho_\GNS(f) \,\hxi \,](r,n) = \int_{G\times N} \dd\mu_G(r')\, \dd\mu_N(n') \, \nu(r')^{-1} \, f(r')(n') \, \hxi(r'^{-1} r)\big( \varsigma_{r'^{-1}}(n'^{-1} n) \big).
\end{equation*}
This representation $\hrho_\GNS$ is unitarily equivalent to the representation $\hrho$ via the unitary equivalence $V :  \hH = L^2(G \times N, \dd\mu_G\, \dd\mu_N) \to \hH_\GNS = L^2(G \times N, \dd\mu_\GNS)$ which is the extension of $v$ to the respective Hilbert spaces. This corresponds to the last column of \eqref{eq-diagram}

\section{Examples}
\label{sec-examples}

\subsection{\texorpdfstring{The affine group $\gR \ltimes \gR$}{The affine group RxR}}
\label{sec-example-affine-group}

We now show that the affine group $K = \gR \ltimes \gR$ is an illustrative example of the above situation and how a spectral triple on $C^*(\gR)$ can be extended to a modular-type twisted spectral triple on $C^*(K)$, which coincides with the one defined in \cite{Mata14a}.

\subsubsection{General considerations}
\label{General considerations}

The group $G\vc\gR$ acts on $N\vc\gR$ by $\varsigma_a(b) \vc e^{-a} b$ for any $a,b \in \gR$. This defines the product 
\begin{equation*}
(a,b) \,(a', b') = (a + a', b + e^{-a} b')
\end{equation*}
on $K\vc\gR \ltimes \gR$. Following \ref{subsec-crossproductalgebra}, a direct computation shows that $\nu(a) = e^{-a}$, so that the modular function of $\gR \ltimes \gR$ is $\Delta_K(a,b) = e^a$.

The original spectral triple we consider is described by the following elements:

-- $A \vc C_c(\gR)$ is a $*$-dense subalgebra (for the convolution product) of $\bA \vc C^*(\gR)$, 

-- $H \vc L^2(\gR, \dd b)$,

-- $\rho\vc \lambda_\gR$ is the left regular representation of $\gR$ which induces the convolution product given in \eqref{lambdaN},

-- $D$ is the unbounded operator
\begin{equation*}
(D \,\xi)(b) \vc b\, \xi(b), \quad \text{for any $\xi \in C_c(\gR) \subset H$.}
\end{equation*}

Identifying $H=L^2(\gR, \dd b)$ with $L^2(\hgR, \dd \hb)$ by Fourier transform, where $\hgR$ is the Pontryagin dual of $\gR$, $D$  becomes the differential operator $-i \partial_{\,\hb}$. 

The spectral triple $(A, H, D)$ is non unital, and $\rho$ is the GNS representation $\pi_\varphi$ of the Plancherel weight $\varphi(f_\gR) \vc f_\gR(0)$ defined for any $f_\gR \in A$.

The action \eqref{eq-def-alpha-group} of $G=\gR$ on $\bA=C^*(\gR)$ is given by
\begin{equation*}
\alpha_a(f_\gR)(b) \vc e^a f_\gR(e^a b),\quad\text{for $a,b \in \gR$ and $f_\gR \in A$},
\end{equation*}
the algebra $\bB=\gR \ltimes_{\alpha, \red} C^*(\gR)$ has explicit product and involution (see \eqref{eq-prod-beta} and \eqref{eq-involution-beta})
\begin{align*}
(f \star g)(a)(b) &= \int \dd a' \dd b' \, e^{a'}\,  f(a')(b') \, g(a-a')(e^{a'}(b-b'))\\
&= \int \dd a' \dd b' \, f(a-a')(b- e^{-(a-a')}b') \, g(a')(b'),\\
f^*(a)(b) &= e^{a} \overline{f(-a)(-e^{a}b)},
\end{align*}
and the representation \eqref{eq-rep-GxN} of $ \bB$ on $\hH \vc L^2(\gR, \dd a) \otimes H \simeq L^2(\gR^2, \dd a \dd b)$ takes the explicit form
\begin{equation*}
\big(\hrho(f)\, \hxi \,\big)(a)(b) = e^{-a} \int_{\gR \times \gR} \dd a'\, \dd b'\, f(a')(e^{-a} b')\, \hxi(a-a')(b-b'),
\end{equation*}
for any $\hxi \in C_c(\gR , C_c(\gR)) \subset \hH$ and $f \in C_c(\gR , C_c(\gR)) \subset \gR \ltimes_{\alpha, \red} C^*(\gR)$.

Proposition~\ref{prop-hyp-satisfied-group-product} gives $q(a) \vc e^{-a/2}$, and from Proposition~\ref{prop-def-unitary}, the unitary operator defined by
\begin{equation*}
(U_a \,\xi)(b) = e^{a/2} \,\xi(e^a b), \quad \text{for $\xi \in C_c(\gR) \subset H$,}
\end{equation*}
satisfies \eqref{eq-UaU*}. A simple computation gives
\begin{equation*}
(U_a D U_a^*\, \xi)(b) =  e^a b \,\xi(b) = e^a (D\, \xi)(b) = \big({z(a)^*}^{-1} D z(a)^{-1} \,\xi \big)(b) \quad \text{with $z(a) \vc e^{-a/2} \,\bbbone$,}
\end{equation*}
so that $\vartheta(a)=e^{-a/2}$ and with
\begin{equation*}
p(a) \vc e^{-a}\,\bbbone,
\end{equation*}
one has $\rho[p(a)] = z(a) z(a)^* = e^{-a}\, \bbbone$.

Since $z(a)$ does not depends on the variable $b$, one has
\begin{equation}
\label{eq-[D,f]-affine}
[D, \rho(f(a)) \,z(a) ] z(a)^{-1} = [D, \rho(f(a))] = \rho[\partial f(a)]
\end{equation}
for any $f \in C_c(\gR , C_c(\gR))$, where $\partial$ is the differential operator $(\partial f_\gR)(b) \vc b \,f_\gR(b)$ on $A$. 

The operators entering the operator $\DD$ of the modular-type twisted spectral triple of Theorem~\ref{thm-twisted-triple} are given by
\begin{equation*}
(\hD \,\hxi\,)(a)(b) = e^{-a} b\; \hxi(a)(b)
\end{equation*}
and, since $(\htheta\, \hxi\,)(a) = e^{-a} \,\hxi(a)$,
\begin{equation*}
(\caT_{\cone,\,\ct} \,\hxi\,)(a)(b) = (\cone + \ct \,e^{-a}) \,\hxi(a)(b).
\end{equation*}
Moreover, the hypothesis of Proposition \ref{prop-compactness} holds true: $[\,\hD, \caT_{\cone,\,\ct}]=0$.\\
Similarly, for $v\in \gC^2$, 
\begin{equation*}
\DD_a\big(\hxi(a)\otimes v\big)(b)=b\,\hxi(a)(b)\otimes \gamma^1 v+(\cone+\ct e^{-a})\,\hxi(a)(b)\otimes \gamma^2 v.
\end{equation*}
We now check Hypothesis \ref{hyp-dirac}-\ref{preservdom}: since 
\begin{align*}
\hY=\{\hxi \in L^2(\gR^2,\dd a\dd b)\,\vert \, \hxi(a)\in \Dom(D), \,[a\in \gR \to \norm{(b\pm i \ct\,e^{-a})\,\hxi(a)}_{L^2(\gR,\dd b)} \in L^2(\gR,\dd a)]\}
\end{align*}
contains $$\caS_c(\gR^2) \vc \{ \text{Schwartz functions in $(a,b)$ with compact support in $a$}\}$$ which is dense in $L^2(\gR^2,\dd a\dd b)$, this proves \eqref{hY}. The existence of $Y_\algA$ will be given in \ref{sec-extended-affine}.

By definition \eqref{beta}, $\beta(f)(a)=p(a)f(a) = e^{-a} f(a)$ for any $C_c(\gR, C^*(\gR))$, but $\beta$ does not extend to an automorphism of $\gR \ltimes_{\alpha, \red} C^*(\gR)$: if $f(a)(b)\vc e^{-\abs{a}/2}\,g(b)$ for a $g\in C_c(\gR)$, then $f\in L^1(\gR,C^*(\gR))$ and $\beta(f)(a,b)=e^{-a+\abs{a}/2}\,g(b)$, but $\beta(f)$ is not in $L^1(\gR,C^*(\gR))$.

\subsubsection{\texorpdfstring{Possible choices for $\algA$}{Possible choices for  A}}
\label{Choice-affine-A}

The Hypothesis~\ref{hyp-dirac}-\ref{hyp-algA} requires the existence of an algebra $\algA$ for the construction of the spectral triple described in Theorem~\ref{thm-twisted-triple}. Here we present three possible algebras, which are related by inclusion, from the smallest one to the largest one, and which are dense $*$-subalgebras of $\bB=\gR \ltimes_{\alpha, \red} C^*(\gR)$.

$\bullet$ Since the present situation enters into the description of Section~\ref{subsec-crossproductalgebra}, we can consider the algebra $\algA_1 \vc \Phi (C_c(\gR\times \gR)) \subset C_c(\gR,C_c(\gR))$ as in \eqref{group-ongroup}: this corresponds to take $A = C_c(\gR)$ for the original spectral triple (as described before), \eqref{eq-hyp-p(r)} is trivially satisfied, and, considering \eqref{eq-[D,f]-affine}, one can take
\begin{equation*}
M_{f,z} \vc \sup_{a \in S_f} \, \int_\gR \dd b\, \abs{b f(a)(b)} < \infty
\end{equation*}
to satisfy \eqref{eq-hyp-comm-bounded} for any $f \in \algA_1$, where $S_f = \Supp(f)$. So we get the conclusion of Theorem \ref{thm-twisted-triple}, postponing the compactness of the resolvent \eqref{modularcompactresol} to \eqref{rescompacte}.

$\bullet$ The original spectral triple $(A, H, D)$ can also be given using a Fréchet $*$-algebra described as follows. For any $b \in \gR$, let $\sigma(b) \vc 1 + \abs{b}$ be a $m$-sub-polynomial scale on $\gR$ (a weight in fact). Following \cite{Schweitzer}, we define $L^\sigma_1(\gR)$ as the space of Borel measurable functions $f_\gR : \gR \to \gC$ such that the semi-norms
\begin{equation*}
\norm{f_\gR}^\sigma_m \vc \int_\gR \dd b\, \sigma(b)^m \abs{f_\gR(b)}
\end{equation*}
are finite for any $m \in \gN$. This is a $m$-convex Fréchet $*$-algebra for the semi-norms $\norm{\,}^\sigma_m$ and the convolution product and involution. Then it is easy to check that $(A = L^\sigma_1(\gR), H, D)$ is a spectral triple, with the same $H$ and $D$ as before.

For any $a \in \gR$, let $w(a)\vc \tfrac{1}{2}(1+e^{\abs{a}})$: this is a $m$-sub-polynomial scale on $G=\gR$. Then one has $\sigma(e^{-a} b) \leq 2\,w(a)\,\sigma(b)$, and this implies that the action $\alpha$ is $m$-$\sigma$-tempered since 
\begin{align*}
\norm{\alpha_a(f)}^\sigma_m = \int_\gR \dd b\, (1+ \abs{b})^m \, \abs{e^af(e^ab)} = \int_\gR \dd b'\, (1+ \abs{e^{-a}b'})^m\, \abs{f(b')} \leq 2^m\,w(a)^m \,\norm{f}^\sigma_m.
\end{align*}
The hypotheses of Lemma~\ref{lem-Afrechet} are satisfied, so consider $\algA_2 \vc C_c(\gR,L^\sigma_1(\gR))$. Then \eqref{eq-hyp-p(r)} holds true since functions in $\algA_2$ are compactly supported along the variable $a$, \eqref{eq-hyp-comm-bounded} is satisfied for any $f \in \algA_2$ with 
\begin{equation*}
M_{f,z} \vc \sup_{a \in S_f} \, \norm{f(a)}^\sigma_1 < \infty
\end{equation*}
(notice that the continuity of $:a \in \gR \mapsto f(a) \in L^\sigma_1(\gR)$ implies the continuity of $a \mapsto \norm{f(a)}^\sigma_1$). Again, we get the conclusion of Theorem~\ref{thm-twisted-triple} for this algebra.

$\bullet$ The construction presented in Remark~\ref{rmk-frechet-w-sigma} leads to a $m$-convex Fréchet $*$-algebra $\algA_3 \vc L^w_1(\gR,L_1^\sigma(\gR))$ \cite[Theorem~3.1.7]{Schweitzer} since $w_- = w$. We denote by 
\begin{equation*}
\norm{f}_{m,n} \vc \int_{\gR^2} \dd a \dd b \, w(a)^m \,\sigma(b)^n \,\abs{f(a)(b)}
\end{equation*}
the family of semi-norms which topologizes $\algA_3$. Note that $\algA_3$ is not included into $C_c(\gR, A)$, so it does not satisfies all the requirements in Hypothesis~\ref{hyp-dirac} and we cannot apply Theorem~\ref{thm-twisted-triple}, but nevertheless we now show that we can proceed directly.\\
From the general theory, one has
\begin{equation*}
\norm{\hrho(f)}_{\caB(\hH)} 
\leq \norm{f}_{L^1(\gR, \dd a)} 
\vc \int_\gR \dd a \, \norm{f(a)}_{C^*(\gR)} 
\leq \int_\gR \dd a \, \norm{f(a)}_{L^1(\gR, \dd b)} 
=\norm{f}_{0,0}.
\end{equation*}
This shows that the inclusion $\algA_3 \to \gR \ltimes_{\alpha, \red} C^*(\gR)$ is continuous.

Following the same line of computations as in the proof of Lemma~\ref{prop-hD}, one gets, for $\hxi$ in the dense subspace $C_c(\gR^2)$ of $\hH$:
\begin{align*}
([\hD, \hrho(f)]_\beta\, \hxi\,)(a)(b)
&= e^{-a} \int_{\gR^2} \dd a' \dd b' \, g(a')( e^{-a} b')  \,  \hxi(a-a')(b-b')
\end{align*}
with $g(a)(b) = b f(a)(b) = (\partial f(a) )(b)$. \\
Proposition~\ref{combounded} gives $[\caT_{\cone,\,\ct},\,\hrho(f)]_\beta=\cone\hrho(f)-\cone\hrho[\beta(f)]$, so that
\begin{align*}
\norm{[\hD, \hrho(f)]_\beta}_{\caB(\hH)} & \leq \norm{g}_{0,0} \leq \norm{f}_{0,1},
\\
\norm{[\caT_{\cone,\,\ct},\,\hrho(f)]_\beta}_{\caB(\hH)} & \leq \abs{\cone} \norm{f}_{0,0} + \abs{\cone} \norm{f}_{1,0}.
\end{align*}
This shows that the twisted commutators $[\DD, \pi(f)]_\beta$ are bounded for $f\in\algA_3$ so $(\algA_3, \caH,\DD)$ is also a modular-type $\beta$-twisted spectral triple. The algebra $\algA_3$, in particular the behavior of its functions at infinity in the $a$ and $b$ directions (governed by $w$ and $\sigma$), is perfectly adapted to the Dirac operator $\DD$ proposed in Theorem~\ref{thm-twisted-triple}.

\subsubsection{The extended triple}
\label{sec-extended-affine}

Since we want to recover the spectral triple proposed in \cite{Mata14a}, we choose to consider the algebra given there, defined by
\begin{equation*}
\algA=\algA_S \vc \caS_c(\gR^2).
\end{equation*}
Remark that $\algA$ is a subalgebra of $\algA_2$ whose elements are smooth functions in both variables $(a,b)$. With slight modifications we could have chosen a larger subalgebra (with minimal constraints on smoothness) in one of the three algebras quoted in previous section.

The second part of Hypothesis \ref{hyp-dirac}-\ref{preservdom} is satisfied with
\begin{align*}
Y_\algA \vc \algA  \subset \hY
\end{align*}
where the inclusion has been proved before and the inclusion of $\hrho(f) \,Y_\algA \subset Y_\algA \,(\subset \hY)$ can be shown using the explicit expression of $\hrho$, the compactness along the variable $a$ and the Schwartz behavior along $b$.

Moreover, $\beta_z(f)(a) \vc e^{-za} f(a)$ defines an automorphism of $\algA_S$ for any $z \in \gC$ (see Remark~\ref{rmk-betaz}).

The Hypothesis~\ref{hypo-weight} is automatically satisfied, as proved in  Proposition~\ref{prop-hyp-satisfied-group-product}, for the one-cocycle
\begin{equation*}
q(a)^2 \vc \nu(a) \,\bbbone = e^{-a} \,\bbbone.
\end{equation*}
Of course $a\in \gR \to \norm{q(a)}=e^{-a/2}$ is continuous, so $q$ is not only in $Z^1(\gR,M(C^*(\gR))^\times)$ but also in $Z^1(\gR,Z(C_c(\gR))^\times)$.

The constraint \eqref{eq-compact} with $c=0$ and $s=1$ is never satisfied:
\begin{align}
\label{div}
\int_G \dd\mu_G(r)\, \Delta_G(r)^{-1} \, \norm{[\bbbone + T_p(r)^2]^{-1}}_{\caB(H)}= \int_\gR \dd a\, [1+(\cone+\ct\,e^{-a})^2]^{-1} =\infty.
\end{align}
However, the constraint \eqref{eq-compact} holds true for $0<c<s$: 
\begin{align}
\int_G \dd\mu_G(r)\, &\Delta_G(r)^{-1} \norm*{\rho(p(r))}_{\caB(H)}^{2c} \, \norm{[\bbbone + T_p(r)^2]^{-1}}_{\caB(H)}^s
\nonumber\\
&=\int_\gR \dd a\,\tfrac{e^{-2ac}}{[1+(\cone+\ct\,e^{-a})^2]^s}  =\int_0^\infty \dd u\,\tfrac{u^{2c-1}}{[1+(\cone+\ct \,u)^2]^s}
\label{div2}
\end{align}
and the last integral converges. Remark that it diverges for $c=0$ and any $s$, or for $s=1$ and any $c\geq 1$. 
Since Proposition \ref{prop-compactness} asks for $c\geq 0$ and $s\geq 1$, we finally get, for any $f\in \algA$, 
\begin{equation}
\label{rescompacte}
\Theta^c\pi(f)(\bbbone+\DD^2)^{-s/2} \text{ is compact if }s>c\geq 1,
\end{equation}
so \eqref{modularcompactresol} is satisfied. \\
{\it Open question}: is the operator $\Theta \pi(f)(\bbbone+\DD^2)^{-1/2}$ compact? The lack of answer explains why we used in \eqref{modularcompactresol} a weak form for the requirement of compactness of the resolvent.

In \cite{Mata14a} the modular spectral triple introduced on the $\kappa$-Minkowski space is based on the same algebra $\algA_S$.  Denote by $\caF$ the Fourier transform on $\gR^2$, given by 
\begin{align*}
\caF(f)(\alpha, \beta) \vc (2\pi)^{-1} \int \dd a \dd b f(a)(b) e^{i a \alpha + i b \beta}.
\end{align*}
In \cite{Mata14a}, the variables $(a,b)$ (resp. $(\alpha, \beta)$) are denoted by $(p_0, p_1)$ (resp. $(x_0, x_1)$), and the Hilbert space of the representation of $\algA_S$ is  $L^2(\gR^2, \dd \alpha \dd \beta) \otimes \gC^2$. The Fourier transform $\caF : \hH = L^2(\gR^2, \dd a\dd b) \to  L^2(\gR^2, \dd \alpha \dd \beta)$ intertwines the representations $\pi$ defined here and the representation defined in \cite{Mata14a}, and it identifies the Dirac operator proposed in \cite{Mata14a} with $\DD$ on $\caH \vc \hH \otimes \gC^2$ for the specific values $\cone = 1$ and $\ct = -1$ (see Proposition~\ref{prop-hJ-operator} for the origin of the relation $\cone = \epsilon\, \ct$ and comment below).

Let $J$ be the reality operator on the original triple $(A, H, D)$, which is given by the modular theory of the GNS representation of the Plancherel weight on $C^*(\gR)$. Here $\sigma$ is trivial ($\gR$ is unimodular), so that $J \xi = \xi^*$ for any $\xi \in C_c(\gR)$, \textsl{i.e.} $(J \xi)(b) = \overline{\xi(-b)}$ (see \eqref{eq-modularconjugationN}). A simple computation gives $JD = - DJ$, so that $\epsilon = -1$ in Proposition~\ref{prop-hJ-operator}. Imposing the existence of the reality operator $\hJ$ of Proposition \ref{prop-hJ-operator} forces $\cone = - \ct$, as noticed before, to recover the Dirac operator in \cite[Theorem~25]{Mata14a}, and this real structure is the one described in \cite[Proposition~34]{Mata14a}. As seen in Section \ref{The dual weight}, this is the natural reality operator defined in this context.

\begin{remark}
\label{comparaison}
This twisted spectral triple is a modular spectral triples as in \cite[Definition 5.1]{Kaa11} except it is non unital, so Hypothesis \ref {hyp-dirac}-\ref{hyp-comm-bounded} is cancelled, and the constraint on the resolvent has been modified in \eqref{modularcompactresol}: the corresponding notations are
\begin{equation*}
\caN \to \algB(\caH),\,\phi \to \Tr(\Theta\,\cdot),\,\mathcal{A}_S \to \algA_S,\,\sigma \to \hsigma,\,\theta \to \beta,\, D\to \DD,
\end{equation*}
where $\hsigma_t(x)\vc \Theta^{-it} x \,\Theta^{it}$. Since $\caN^\sigma$ corresponds to $\{ a\in \algB(\caH)  \, \vert \, [a,\Theta^{it}]=0,\,\forall t\in \gR \}$, the affiliation of $\DD$ to $\caN^\sigma$ means $[\DD, \Theta]=0$ which is satisfied here. Remark that, for any $s>0$, $\pi(a)(\bbbone+ \DD^2)^{-s/2}$ is never in the set $\mathcal{L}^1(\caH)$ of trace-class operators on $\caH$  \cite{Mata14a}. Thus  this twisted affine triple is not finitely summable. The operators $\pi(f)(\bbbone+\DD^2)^{-(1+\epsilon)/2}$ are $\phi$-compact for any $f\in \algA$ and $\epsilon >0$.
 \end{remark}

\begin{theorem}
\label{As0}
For $f, g \in \algA_S$, $\pi(f)\, \Theta\, (1 + \DD^2)^{-s/2}\, \pi(g)\in\mathcal{L}^1(\caH)$ when $s > 2$.
\end{theorem}

The main difficulty to prove this theorem is that the symbol $a(x,\xi)$ of this operator is not classical since some $(\partial^\alpha_x\partial^\beta_\xi a)(x,\xi)$ can increase exponentially on non-zero Lebesgue measure sets. Our strategy is to get traceability of the operator multiplied by the unbounded operator $\Theta^c$ with $c>0$ using a result of Arsu \cite{Arsu08a}, and then to get rid of this modification by letting $c\to 0$ by showing that the trace-class property is preserved via a result of Deift--Simon in \cite{DS} (see Lemma \ref{lem-tracelimit}). A direct approach of this method with $c=0$ fails. \\
So we first begin to show a result similar to the theorem where the replacement of $\Theta$ by $\Theta^{1+c}$ with $c>0$ plays a key role in the proof: actually, when $c=0$ in the following proposition, this is Theorem 30 of \cite{Mata14a} but the proof given there is not correct.

\begin{proposition}
\label{thm-traceclass}
For any $f \in \algA_S$, $\pi(f)\, \Theta^{1+c}\, (1 + \DD^2)^{-s/2} \in \mathcal{L}^1(\caH)$ for $c > 0$ and $s > c+2$.
\end{proposition}

\begin{proof}
Remark that $\pi(f)\, \Theta^{1+c}\, (1 + \DD^2)^{-s/2}\in \mathcal{L}^1(\caH)$  is equivalent to 
\begin{align}
\label{Asc}
A_{s,c}(f) \vc \hrho(f)\, \htheta^{1+c}\, (1 + \hD^2 + \caT_{\cone,\,\ct}^2)^{-s/2}\in \mathcal{L}^1(\hH).
\end{align}
For any $\xi \in \hH$,
\begin{equation}
\label{eq-opAsc}
(A_{s,c}(f) \xi)(a)(b) 
= e^{-a} \int_{\gR^2} \dd a' \dd b' \, f(a-a')(e^{-a}(b-b')) \, g_{s,c}(a',b') \, \xi(a')(b')
\end{equation}
with
\begin{align*}
g_{s,c}(a,b) \vc e^{-(1+c)a} [ 1 + e^{-2a} b^2 + (\cone + \ct e^{-a})^2 ]^{-s/2}.
\end{align*}

Using the Fourier transform $\caF$, for any $\zeta \in L^2(\gR^2, \dd \alpha \dd \beta)$, one has
\begin{equation}
\label{defOp}
[A_{s,c}(f) \,\zeta](\alpha, \beta) = \tfrac{1}{(2\pi)^2} \int_{\gR^4} \dd \alpha' \dd \beta' \dd a' \dd b' \, e^{i(\alpha - \alpha')\,a' + i(\beta - \beta')\,b'}\, a_{s,c}(f)(\alpha, \beta ; a', b') \, \zeta(\alpha', \beta')
\end{equation}
with
\begin{align*}
a_{s,c}(f)(\alpha, \beta; a', b') 
&\vc \int_{\gR^2} \dd a\, \dd b \, e^{i \alpha (a-a') + i \beta (b-b')} e^{-a} f(a-a')(e^{-a} (b-b')) \, g_{s,c}(a', b')
\\
&= \int_{\gR^2} \dd a\, \dd b \, e^{i \alpha a + i e^{a'}\beta b} \, e^{-a} f(a)(e^{-a}b) \, g_{s,c}(a', b')
\\
&=2\pi \, \caF(h)(\alpha, e^{a'}\beta) \, g_{s,c}(a', b'),
\end{align*}
where
\begin{equation*}
h(a) \vc \alpha_{-a}[f(a)], \text{ that is } h(a)(b) = e^{-a} f(a)(e^{-a}b).
\end{equation*}
For $f \in \algA_S$, $h$ is smooth and compactly supported along $a$, and a Schwartz function along $b$.
So $\caF(h)$ is a Schwartz function in both variables $\alpha, \beta$ (and in fact analytic along $\alpha$).

Since $A_{s,c}(f)$ on $L^2(\gR^2, \dd \alpha \dd \beta)$ is a pseudo-differential operator with symbol $a_{s,c}(f)$, we can use the following result of Arsu to prove that $A_{s,c}(f)\in \mathcal{L}^1(\hH)$  since $a_{s,c}(f)$ has sufficiently many derivatives in $L^1(\gR^2\times \gR^2)$. We will use the following norm on the set of functions $a\in\caS'(\gR^2\times \gR^2)$ where $\mathbf{t}=(t_1,t_2)\in \gN^2$:
\begin{equation*}
\abs{a}_{1, \mathbf{t}} \vc \max_{\substack{n_\alpha, n_a  \leq t_1 \\  n_\beta, n_b \leq t_2}} \norm*{\partial_\alpha^{n_\alpha} \partial_\beta^{n_\beta} \partial_a^{n_a} \partial_b^{n_b} a}_{L^1} <\infty\,.
\end{equation*}
Using the orthogonal decomposition $\gR^2 = \gR \times \gR$, the conditions of \cite[Theorem~5.4~(i)]{Arsu08a} (with $\tau = 0$) can be recast as

\begin{theorem}
\label{Arsu}
Let $a\in\caS'(\gR^2\times \gR^2)$ such that $\abs{a}_{1, \mathbf{t}}<\infty$ for $\mathbf{t} = (1,1)$. Then $\Op(a)$ (defined as in \eqref{defOp} by the symbol $a$) has an extension which is a trace-class operator on the Hilbert space $L^2(\gR^2,\dd \alpha\dd\beta)$ and $\norm{\Op(a)}_1\leq C \vert a\vert_{1,\mathbf{t}}$ for some constant $C$.
\end{theorem}
So, we only need to prove that for $a=a_{s,c}(f)$, we have:

-- $a_{s,c}(f) \in \caS'(\gR^2 \times \gR^2)$,

-- $\abs{a_{s,c}(f)}_{1, (1,1)} < \infty$.

Let us check the second point while the first will be shown also in the proof, see ``Case $(0,0,0,0)$'' below. We have to consider derivatives of order at most $1$ in the four directions of the function $(\alpha, \beta; a, b) \mapsto a_{s,c}(f)(\alpha, \beta; a, b)$. This gives 16 cases to consider.

Denote by $(n_\alpha, n_\beta, n_a, n_b) \in \{0, 1\}^4$ the order of derivations along the 4 variables.

\noindent -- Case $(0,0,0,0)$: 
\begin{align*}
\int_{\gR^4} \dd \alpha \dd \beta \dd a \dd b \, \abs*{a_{s,c}(f)(\alpha, \beta; a, b)}
&= 2\pi \, \int_{\gR^4} \dd \alpha \dd \beta \dd a \dd b \, \abs*{\caF(h)(\alpha, e^{a}\beta)} \, g_{s,c}(a, b)
\\
& = 2\pi \, \left[ \int_{\gR^2} \dd \alpha \dd \beta \, \abs*{\caF(h)(\alpha, \beta)}\right] \left[\int_{\gR^2} \dd a \dd b \,  e^{-a} g_{s,c}(a, b) \right],
\end{align*}
where in the second line we use the change of variable $e^{a}\beta \to \beta$. The first integral is finite since $\caF(h)$ is Schwartz in both variables, and the second one is $\int_{\gR^2} \dd a \dd b \, g_{s,c+1}(a, b)$, which is finite for $c+1>0$ and $s > 2+c$ by Lemma~\ref{lem-gsc} (see below).

Since $a_{s,c}(f) \in L^1(\gR^4)$ for $c+1>0$ and $s>2+c$, this shows that $a_{s,c}(f) \in \caS'(\gR^2 \times \gR^2)$. 

\noindent -- Cases $(1, n_\beta, n_a, n_b)$: The  derivative along $\alpha$ concerns only $\caF(h)$, and $\partial_\alpha \caF(h)$ is Schwartz in both variables. So, $(1, n_\beta, n_a, n_b)$ is equivalent to $(0, n_\beta, n_a, n_b)$.

\noindent -- Cases $(n_\alpha, n_\beta, n_a, 1)$: The derivative along $b$ concerns only $g_{s,c}$, and $\abs*{\partial_b g_{s,c}}$ and $\abs*{\partial_a \partial_b g_{s,c}}$ are dominated by $g_{s,c}$ by Lemma~\ref{lem-gsc}. So $(n_\alpha, n_\beta, n_a, 1)$ is equivalent to $(n_\alpha, n_\beta, n_a, 0)$.

\noindent -- Cases $(n_\alpha, n_\beta, 1, n_b)$: The derivative along $a$ produces two terms:

-- $e^a \beta (\partial_\beta \caF(h))(\alpha, e^a\beta) g_{s,c}(a,b)$: the function $(\alpha, \beta) \mapsto \beta (\partial_\beta \caF(h))(\alpha, \beta)$ is Schwartz in both variables, so we reduce to the situation $n_a = 0$.

-- $\caF(h)(\alpha, e^a\beta) (\partial_a g_{s,c})(a,b)$: the function $\partial_a g_{s,c}$ is dominated by $g_{s,c}$.
\\
So $(n_\alpha, n_\beta, 1, n_b)$ is equivalent to $(n_\alpha, n_\beta, 0, n_b)$.

\noindent Combining all these equivalences, the case $(0,0,0,0)$ covers the 8 cases $(n_\alpha, 0, n_a, n_b)$.

\noindent -- Case $(0,1,0,0)$: 
\begin{align*}
\int_{\gR^4} \dd \alpha \dd \beta \dd a \dd b \, \abs*{\partial_\beta a_{s,c}(f)(\alpha, \beta; a, b)}
&= 2\pi \, \int_{\gR^4} \dd \alpha \dd \beta \dd a \dd b \,  e^{a} \abs*{(\partial_\beta \caF(h))(\alpha, e^a\beta)} \, g_{s,c}(a, b)
\\
& = 2\pi \, \left[ \int_{\gR^2} \dd \alpha \dd \beta \, \abs*{\partial_\beta \caF(h)(\alpha, \beta)}\right] \left[\int_{\gR^2} \dd a \dd b \,  g_{s,c}(a, b) \right].
\end{align*}
The first integral is finite since $\partial_\beta \caF(h)$ is Schwartz in both variables, and the second one is finite for $c>0$ and $s > 1+c$ by Lemma~\ref{lem-gsc}.

This case covers the remaining 8 cases $(n_\alpha, 1, n_a, n_b)$.

\medskip
Since $a_{s,c}(f)$ satisfies the hypothesis of Theorem \ref{Arsu}, we end up with the constrains  $c>0$ and $s>2+c$, which are the hypothesis of Proposition~\ref{thm-traceclass}. 
\end{proof}
\begin{lemma}\label{lem-gsc}
The function $g_{s,c}$ has the following properties:

\noindent 1) $g_{s,c} \in L^1(\gR^2, \dd a \dd b)$ if and only if $c>0$ and $s > 1+c$.

\noindent 2) There exists a constant $M_{s,c} > 0$ such that $\abs*{(X\, g_{s,c})(a,b)} \leq M_{s,c}\,g_{s,c}(a,b)$ for $(a,b) \in \gR^2$ when $X=\partial_a,\,\partial_b$ or $\partial_a\partial_b$.
\end{lemma}

\begin{proof}
1) $g_{s,c}$ is positive and $\int_\gR \dd b \, g_{s,c}(a,b)$ is finite if and only if $s > 1$ and in this case,
\begin{equation}\label{eq-intdbgsc}
\int_\gR \dd b \, g_{s,c}(a,b) = \tfrac{\sqrt{\pi}\,\Gamma[(s-1)/2]}{\Gamma(s/2)} \, e^{-c a} \,[ 1 + (\cone + \ct \,e^{-a})^2]^{-(s-1)/2}.
\end{equation}
Then
\begin{align*}
\int_{-\infty}^\infty \dd a \, e^{-c a}\, [ 1 + (\cone + \ct \,e^{-a})^2 ]^{-(s-1)/2}
&= \int_0^\infty \dd u \,  u^{c-1}\,[ 1 + (\cone + \ct\, u)^2]^{-(s-1)/2}
\end{align*}
which is finite if and only if $c>0$ and $s>1+c$. 

2) For $X=\partial_a$, we have
\begin{align*}
(\partial_a g_{s,c}) (a,b) = g_{s,c}(a,b) \left[  -(1+c) + s\, \tfrac{(b^2 + c^2_2) e^{-2a} + \cone \ct e^{-a} }{1 + e^{-2a} b^2 + (\cone + \ct e^{-a})^2} \right]
\end{align*}
and the bracket is bounded on $\gR^2$. Similarly,
\begin{align*}
(\partial_b g_{s,c}) (a,b) = g_{s,c}(a,b) \, \tfrac{ 2  e^{-2a} b }{1 + e^{-2a} b^2 + (\cone + \ct e^{-a})^2}
\end{align*}
and the fraction is bounded on $\gR^2$. Finally,
\begin{multline*}
(\partial_a \partial_b g_{s,c}) (a,b) = (\partial_a g_{s,c})(a,b) \, \tfrac{ 2  e^{-2a} b }{1 + e^{-2a} b^2 + (\cone + \ct e^{-a})^2}
\\
- 4 g_{s,c}(a,b) \left[ \tfrac{ e^{-2a} b}{1 + e^{-2a} b^2 + (\cone + \ct e^{-a})^2} + 
	\tfrac{e^{-2a} b [ (b^2 + c^2_2) e^{-2a} + \cone \ct e^{-a} ]}{[1 + e^{-2a} b^2 + (\cone + \ct e^{-a})^2]^2}
\right].
\end{multline*}
For the first term, we use previous argument and for the second one  the bracket is bounded on $\gR^2$.
\end{proof}

We now show that the family of operators defined in \eqref{Asc} naturally converges in norm when $c\to 0$:

\begin{lemma}
\label{lem-normlimitAsc}
We have $A_{s,0}(f) =\norm{\cdot}$-$\underset{c\downarrow 0}{\lim} \,A_{s,c}(f)$, for any $s > 2$ and $f \in \algA_S$.
\end{lemma}

\begin{proof}
Using \eqref{eq-opAsc} and $g_{s,0}(a,b) = g_{s,c}(a,b) \,e^{c a}$, the kernel of $B_{s,c}(f) \vc A_{s,c}(f) - A_{s,0}(f)$ on $\hH = L^2(\gR^2, \dd a \dd b)$ is
\begin{equation*}
K_{s,c}(f)(a,b; a', b') = e^{-a} f(a-a')(e^{-a}(b-b'))\, g_{s,c}(a', b') (1 - e^{c a'}).
\end{equation*}
To show that $\lim_{c \to 0^+} \norm*{B_{s,c}(f)} =0$, we use the inequality $\norm*{B_{s,c}(f)} \leq \norm*{K_{s,c}(f)}_{L^2}$. One has
 \begin{align*}
\norm*{K_{s,c}(f)}^2_{L^2} 
 &= \int \dd a \dd b \dd a' \dd b' \, e^{-2a}\, \abs*{f(a-a')(e^{-a}(b-b'))}^2 \, g_{s,c}(a', b')^2 \, (1 - e^{c a'})^2
 \\
 &= \left[ \int \dd a \dd b \, e^{-a} \, \abs*{f(a,b)}^2 \right] \, \left[ \int \dd a' \dd b' \, e^{-a'} \, g_{s,c}(a', b')^2 \, (1 - e^{c a'})^2 \right].
\end{align*}
 The first integral is finite and independent of $s$ and $c$. Using \eqref{eq-intdbgsc}, $g_{s,c}^2(a,b) = e^{-(1+c)a} g_{2s, c}(a,b)$ and after the change of variable $u = e^{-a}$, the second integral can be evaluated as
 \begin{equation*}
\int \dd a \dd b \, e^{-a} \, g_{s,c}(a, b)^2 \, (1 - e^{c a})^2 = 
 \tfrac{\sqrt{\pi}\,\Gamma(s-1/2)}{\Gamma(s)}\, \int_0^\infty \dd u \, \tfrac{u (u^c-1)^2}{[1 + (\cone + \ct u)^2]^{s-1/2}} \,.
\end{equation*}
For $u \in [0,1]$, $(u^2 - 1)^2 \leq 1$, so that $\int_0^1 \dd u \, \frac{u (u^c-1)^2}{[1 + (\cone + \ct u)^2]^{s-1/2}} \leq \int_0^1 \dd u \, \frac{u}{[1 + (\cone + \ct u)^2]^{s-1/2}} <\infty$. Let $c_0 >0$ such that $s > 2+ c_0$. Then for any $0 < c < c_0$ and $u \in [1, \infty)$, one has $u^{1+2c} (1-u^{-c})^2 \leq u^{1+ 2 c_0}$, so that $\int_1^\infty \dd u \, \frac{u (u^c-1)^2}{[1 + (\cone + \ct u)^2]^{s-1/2}} \leq \int_1^\infty \dd u \, \frac{u^{1+ 2 c_0}}{[1 + (\cone + \ct u)^2]^{s-1/2}} < \infty$ since $s > 2+ c_0$. When $u \in [0, \infty)$, $\frac{u (u^c-1)^2}{[1 + (\cone + \ct u)^2]^{s-1/2}}$ goes to $0$ as $c \to 0^+$, so, by the dominated convergence theorem, one get $\lim_{c \to 0^+} \norm*{K_{s,c}(f)}^2_{L^2} = 0$.
\end{proof}

Remark that previous lemma implies that the operator $A_{s,0}(f)=\pi(f) \;\Theta (\bbbone+\DD^2)^{-s/2}$ is compact for $s>2$ (and conjectured to be in $\mathcal{L}^1(\caH)$), but the result of \eqref{rescompacte} is stronger.

\begin{proof}[of Theorem \ref{As0}]
In the following, we fix $s>0$, choose $c_0>0$ such that $s > 2 + c_0$, and we restrict $c$ to $c \in (0,c_0)$.

For $f \in \algA_S$ and previous notations, $A_{s,c}(f) \hrho(f^*) = \hrho(f)\, \htheta^{1+c}\, (1 + \hD^2 + \caT_{\cone,\,\ct}^2)^{-s/2}\, \hrho(f^*) $ is a positive operator which is in $\mathcal{L}^1(\hH)$ when $s > 2+c$ since $\hrho(f^*)$ is bounded  and $A_{s,c}(f)$ is also in $\mathcal{L}^1(\hH)$ by Proposition~\ref{thm-traceclass}. Its trace-norm can be evaluated as
\begin{align*}
\norm*{A_{s,c}(f) \hrho(f^*)}_{1} &= \Tr A_{s,c}(f) \hrho(f^*) = \Tr\ \hrho(f^*) A_{s,c}(f) = \Tr A_{s,c}(f^* \star f) 
\end{align*}
since $A_{s,c}(f)$ and $A_{s,c}(f^* \star f)$ are trace-class. Using the kernel of $A_{s,c}(f^* \star f)$ given in \eqref{eq-opAsc}, 
\begin{align*}
\Tr A_{s,c}(f^* \star f) 
&= \int \dd a \dd b\, e^{-a} (f^* \star f)(0)(0) \, g_{s,c}(a,b)
= (f^* \star f)(0)(0)  \int \dd a \dd b\, g_{s,c+1}(a,b)
\end{align*}
which is finite since $s>2+c$ by Lemma~\ref{lem-gsc}, and the integral evaluates to
\begin{equation*}
\int \dd a \dd b\, g_{s,c+1}(a,b) = \tfrac{\sqrt{\pi}\,\Gamma[(s-1)/2]}{\Gamma(s/2)} \int_0^\infty \dd u\, \tfrac{u^c}{\left[ 1 + (\cone + \ct u)^2 \right]^{(s-1)/2}}\,.
\end{equation*}
For $u \in [0,1]$, $\frac{u^c}{\left[ 1 + (\cone + \ct u)^2 \right]^{(s-1)/2}} \leq \frac{1}{\left[ 1 + (\cone + \ct u)^2 \right]^{(s-1)/2}}$ which is integrable on $[0,1]$, and for any $u \in [1, \infty)$, $\frac{u^c}{\left[ 1 + (\cone + \ct u)^2 \right]^{(s-1)/2}} \leq \frac{u^{c_0}}{\left[ 1 + (\cone + \ct u)^2 \right]^{(s-1)/2}}$ which is integrable on $[1, \infty)$. So finally, $\norm*{A_{s,c}(f) \hrho(f^*)}_{1} = \Tr  A_{s,c}(f^* \star f)  \leq M <\infty$ uniformly in $c \in (0, c_0)$.

We deduce from Lemma~\ref{lem-normlimitAsc} that $\norm{\cdot}$-$\lim_{c \to 0^+} A_{s,c}(f) \hrho(f^*) = A_{s,0}(f) \hrho(f^*)$. One concludes by Lemma~\ref{lem-tracelimit} that $A_{s,0}(f) \hrho(f^*)\in\mathcal{L}^1(\hH)$. By polarization, $A_{s,0}(f) \hrho(g)$ is trace-class for any $f, g \in \algA_S$, and so is $\pi(f)\, \Theta\, (1 + \DD^2)^{-s/2}\, \pi(g)$.
\end{proof}

The next result can be found in \cite[Prop. 2]{DS}, but we include a proof for completeness.

\begin{lemma}
\label{lem-tracelimit}
Let $A_n\in \mathcal{L}^1$ be a sequence such that $\norm{\cdot}$-$\underset{n\to \infty}{\lim} A_n = A$ and $\underset{n }{\sup} \,\norm*{A_n}_{1} = M < \infty$. Then $A\in \mathcal{L}^1$.
\end{lemma}

\begin{proof}
Let $B$ be a finite rank operator. Then $AB$ and $(A-A_n)B$ are trace-class, and
\begin{align*}
\abs*{\Tr(A B)} \leq  \norm*{A - A_n} \, \norm*{B}_{1} + \norm*{A_n}_{1} \,  \norm*{B}
\leq \norm*{A - A_n} \, \norm*{B}_{1} + M \norm*{B}.
\end{align*}
Taking $n \to \infty$ on both sides, one gets $\abs*{\Tr(A B)} \leq M \norm*{B}$ for any finite rank operator $B$, with $M$ independent of $B$. The linear form $B \mapsto \Tr(A B)$ is thus continuous, and extends uniquely to a continuous linear form on the Banach space $\caK$ of compact operators. Since the duality $ \mathcal{L}^1 = \caK^*$ is realized by $A \mapsto \Tr(A \cdot)$, one has $A \in \mathcal{L}^1$.\\
 (Remark that the hypothesis $A=\text{strong}$-$\underset{n\to \infty}{\text{limit}} \,A_n$ is sufficient.)
\end{proof}

As claimed after Theorem \ref{As0}, the role of $c$ is important since for instance we cannot use the inequality $\norm{A_{s,c}(f)}_1=\norm{\Op(a_{s,c}(f))}_{1} \leq C \vert a_{s,c}(f)\vert_{1,(1,1)}$ of Theorem \ref{Arsu} to get directly that $A_{s,0}(f)\in\mathcal{L}^1(\hH)$ from previous lemma, because $\sup_{c<1}\vert a_{s,c}(f)\vert_{1,(1,1)}$ is infinite.
\bigskip

Since the spectral dimension of a non-unital modular spectral triple is not yet well settled, we use the following notion of modular dimension in our example
\begin{equation}
\label{def-dimension}
p\vc \inf \{s>0  \ \mid \  \Tr \,\Theta\, \pi(a) (1+\DD^2)^{-s/2}\pi(a^*) <\infty, \,\forall a\in \algA_S\}.
\end{equation}

The modular spectral triple given in \cite{Ioc12} was only $1$-summable, but here the dimension increases:

\begin{theorem}
While the original spectral triple has spectral dimension 1, this modular-type twisted spectral triple given by Theorem \ref{thm-twisted-triple} has modular spectral dimension 2.
\end{theorem}
Since by Theorem \ref{As0}, the dimension is less or equal to 2, this follows from the following:
\begin{lemma}
\label{lem-trace-affine}
For any $f,g\in \algA_S$, $\cone,\,\ct\in\gR$ with $\ct\neq 0$ and $s>2$, we have
\begin{equation*}
\Tr\,\pi(f)\, \Theta\, (\bbbone+\DD^2)^{-s/2} \, \pi(g)
= 2 \hvarphi_\gR(g \star f)\,\big[ \tfrac{\pi}{\abs{\ct} (s-2)}-\tfrac{\cone}{\ct} \tfrac{\sqrt{\pi}\,\Gamma\left((s-1)/2\right)}{\Gamma(s/2)} \,{}_2F_1(\tfrac{1}{2},\tfrac{s-1}{2},\tfrac{3}{2},-\cone^2)\big]
\end{equation*}
where ${}_2F_1$ is some hypergeometric function.
Moreover
\begin{equation*}
\lim_{s\to 2} (s-2) \,\Tr\,\pi(f)\, \Theta\, (\bbbone+\DD^2)^{-s/2} \, \pi(g) = \tfrac{2\pi}{\abs{\ct}}\,\hvarphi_\gR(g \star f).
\end{equation*}
\end{lemma}

\begin{proof}
Thanks to Theorem \ref {As0},
\begin{equation*}
\Tr\,\pi(f)\, \Theta\, (\bbbone+\DD^2)^{-s/2} \, \pi(g) = 2 \Tr \, \hrho(f)\, \htheta\, (1 + \hD^2 + \caT_{\cone,\,\ct}^2)^{-s/2}\, \hrho(g) .
\end{equation*}
The kernel of $A_s(f,g) \vc \hrho(f)\, \htheta\, (1 + \hD^2 + \caT_{\cone,\,\ct}^2)^{-s/2}\, \hrho(g)$ is
\begin{equation*}
K(a,b; a', b') \vc \int \dd a'' \dd b'' \, e^{-a-a''}  f(a - a'')(e^{-a}(b - b'')) \, g_{s,0}(a'', b'') \, g(a'' - a')(e^{-a''}(b'' - b'))
\end{equation*}
so that
\begin{align*}
\Tr A_s(f,g)
&= \int \dd a \dd b \dd a' \dd b' \, e^{-a-a'}  f(a - a')(e^{-a}(b - b')) \, g_{s,0}(a', b') \, g(a' - a)(e^{-a'}(b' - b))
\\
&= (g \star f)(0)(0) \, \int \dd a \dd b\, e^{-a} g_{s,0}(a,b)
\end{align*}
where we have used $(g \star f)(0)(0) = \int \dd a' \dd b' e^{a'} g(a')(b') \, f(-a')(-e^{a'} b')$. The integral is
\begin{align*}
\int \dd a \dd b\, e^{-a} g_{s,0}(a,b)
&= \int \dd a \dd b\, g_{s,1}(a,b)
= \tfrac{\sqrt{\pi}\,\Gamma[(s-1)/2]}{\Gamma(s/2)} \int \dd a \, e^{-a} \left[ 1 + (\cone + \ct e^{-a})^2 \right]^{-(s-1)/2}
\\
&= \tfrac{\sqrt{\pi}\,\Gamma[(s-1)/2]}{\Gamma(s/2)} \, \int_0^\infty \dd u\, [ 1 + (\cone + \ct\, u)^2 ]^{(1-s)/2}
\\
&= \big[ \tfrac{\pi}{\abs{\ct} (s-2)} -\tfrac{\cone}{\ct} \tfrac{\sqrt{\pi}\,\Gamma[(s-1)/2]}{\Gamma(s/2)} \,{}_2F_1(\tfrac{1}{2},\tfrac{s-1}{2},\tfrac{3}{2},-\cone^2)\big].
\end{align*}
Since $\lim_{s\to 2} \,{}_2F_1(\tfrac{1}{2},\tfrac{s-1}{2},\tfrac{3}{2},-\cone^2)=\cone^{-1}\text{arcsinh}(\cone)$, we see that the spectral dimension of the spectral triple is 2 and that $\lim_{s\to 2} (s-2) \,\Tr\,\pi(f)\, \Theta\, (\bbbone+\DD^2)^{-s/2} \, \pi(g) =\tfrac{2\pi}{\abs{\ct}}\,\hvarphi_\gR(g \star f)$, where $\hvarphi_\gR(g \star f) = (g \star f)(0)(0)$.
\end{proof}

In fact, the Dixmier trace of $\Theta\,\pi(f)\,(\bbbone+\DD^2)^{-1}\pi(g)$ appears to be proportional to the dual weight $\hvarphi_\gR(g \star \beta[f])$ exactly as in the unital case, see for instance \cite{CNNR}:

\begin{proposition}
\label{Dixmiertraceable}
We have $\pi(f)\,\Theta(\bbbone+\DD^2)^{-1}\,\pi(g) \in \mathcal{L}^{1,\infty}$ for any $f,g\in \algA_S$. 
\end{proposition}
 
\begin{proof}
We first show that for $f\in \algA_S$, 
\begin{align}
\label{eq-limit}
\lim_{s\downarrow 2} \,(s-2) \,\Tr\, \pi(f)\,\Theta^{s/2}\,(\bbbone+\DD^2)^{-s/2}\pi(f^*)=\tfrac{2\pi}{\abs{\ct}}\,\hvarphi_\gR(f^*\star f).
\end{align}
where the operator is trace-class by Proposition \ref{thm-traceclass}. As in the proof of Lemma \ref{lem-trace-affine} and with $g=f^*\star f$,
\begin{align*}
\Tr\, \pi(f)\,\Theta^{s/2}\,(\bbbone+\DD^2)^{-s/2} \pi(f^*) &=2 \,\hvarphi_\gR(g)\int_{\gR^2} \dd a\,\dd b \,e^{-as/2}\, [1+b^2+(\cone+\ct\,e^{-a})^2]^{-s/2}\\
& = 2 \,\hvarphi_\gR(g)\,\tfrac{\sqrt{\pi}\,\Gamma[(s-1)/2]}{\Gamma(s/2)} \,\int_{-\infty}^\infty \dd a\,e^{-as/2}[1 + (\cone + \ct \,e^{-a})^2]^{(1 - s)/2}
\\
&= 2 \,\hvarphi_\gR(g)\,\tfrac{\sqrt{\pi}\,\Gamma[(s-1)/2]}{\Gamma(s/2)} \, \int_0^\infty \dd u\, u^{s/2-1}\,[ 1 + (\cone + \ct\, u)^2 ]^{(1-s)/2}.
\end{align*}
Since $s-2=\tfrac{2\Gamma(s/2)}{\Gamma(s/2-1)}$ we get 
\begin{multline*}
\lim_{s\downarrow 2} (s-2) \,\Tr\, \pi(f)\,\Theta^{s/2}\,(\bbbone+\DD^2)^{-s/2}\pi(f^*)=\\4 \,\hvarphi_\gR(g)\,\pi \,\lim_{s\to 2} \tfrac{1}{\Gamma(s/2-1)}\,\int_0^\infty \dd u\, u^{s/2-1}[ 1 + (\cone + \ct\, u)^2 ]^{(1-s)/2}.
\end{multline*}
If $I(u,s)\vc u^{s/2-1}\,[ 1 + (\cone + \ct\, u)^2 ]^{(1-s)/2}$, then $\int_0^\infty \dd u\,I(u,s)=\int_0^1 \dd u\,I(u,s)+\int_1^\infty \dd u\,I(u,s)$ where the first integral converges for any $s\geq 2$, so that $\lim_{s\downarrow 2} \tfrac{1}{\Gamma(s/2-1)}\int_0^1 \dd u\,I(u,s)=0$. Moreover 
\begin{align*}
\int_1^\infty \dd u\,I(u,s)=\int_1^\infty \dd u\,[I(u,s)-\tfrac{1}{\vert \ct\vert^{s-1}}u^{-s/2}]+\tfrac{1}{\vert \ct\vert^{s-1}}\int_1^\infty \dd u \,u^{-s/2},
\end{align*}
with $\int_1^\infty \dd u \,u^{-s/2}=\tfrac{2}{s-2}$, thus $\lim_{s\downarrow 2} \tfrac{1}{\Gamma(s/2-1)}\tfrac{1}{\vert \ct\vert^{s-1}}\int_1^\infty \dd u \,u^{-s/2}=\tfrac{1}{\vert \ct\vert}$. \\
For $J(u,s)\vc I(u,s)-\tfrac{1}{\vert \ct\vert^{s-1}}u^{-s/2}$, we get $J(u,2) \underset{u\to \infty}{\simeq} -\tfrac{-\cone}{\vert \ct\vert} u^{-2}$ and \eqref{eq-limit} is proved since
 $\lim_{s\downarrow 2} \tfrac{1}{\Gamma(s/2-1)}\int_1^\infty \dd u \,J(u,s)=0$.
 
We can now follow the arguments of the proof of \cite[Proposition 5.12]{Ioc12}: first we show that $\pi(f)\, \Theta(\bbbone+\DD^2)^{-1} \pi(f^*)\in \mathcal{L}^{1,\infty}(\caH)$ for any $f\in \algA_S$: by \eqref{eq-limit} we get for $G \vc \Theta(\bbbone+\DD^2)^{-1}$
\begin{align*}
\sup_{1\leq s'\leq 2}(s'-1) \Tr \pi(f)\, G^{s'} \pi(f^*)<\infty.
\end{align*}
Note that $G$ is an injective positive operator since $\DD$ and $\Theta$ commutes and $G\leq \ct^{-1}\bbbone$ and we can apply \cite{CGRS} after a renormalisation of $G$. Thus, $\pi(f)$ is in the algebra $B_\zeta(G)$ of \cite[Definition 3.2]{CGRS}, so by \cite[Proposition 3.8]{CGRS}, $\pi(f)\,G\,\pi(f^*) \in \mathcal{L}^{1,\infty}$ and finally by polarization, $\pi(f)\,G\,\pi(g) \in \mathcal{L}^{1,\infty}$ for any $f,g\in \algA_S$.
\end{proof}

We conjecture that $\Theta \,\pi(f)(\bbbone+\DD^2)^{-1} =\pi(\beta[f])\,\Theta (\bbbone+\DD^2)^{-1}\in \mathcal{L}^{1,\infty}$ for any $f\in \algA_S$ and its Dixmier trace is equal to $\tfrac{2\pi}{\abs{\ct}}\,\hvarphi_\gR(f)$, which means that, within the definition \eqref{def-dimension} of the modular dimension, the symmetrization in $a$  is not necessary. 

\begin{remark}
\label{affine bifurcation C-M}
As quoted in Section \ref{bifurcation C-M}, we compute now the spectral dimension of our twisted affine triple represented on $H$ and not on $\hH$. We have
\begin{align*}
[(\rho \rtimes U)(f)(\bbbone+D^2)^{-s/2}\,\xi\,] (b)&= [\int \dd a\,\lambda_\gR(f(a))\,U_a\,(\bbbone+D^2)^{-s/2}\,\xi\,](b)\\
&=\int \dd a\,\dd b'\,f(a)(b')\,[U_a((\bbbone+D^2)^{-s/2} \,\xi\,](b-b')\\
&=\int \dd a\,\dd b'\,f(a)(b')\,e^{a/2}[1+(e^a(b-b'))^2]^{-s/2}) \,\xi(e^a(b-b'))\\
&=\int \dd b''\int \dd a\, e^{-a/2}\,f(a)(b-e^{-a}b'')\,(1+b''^2)^{-s/2}\, \xi(b'')\\
&=\int \dd b'' \,K(b,b'') \,\xi(b'')
\end{align*}
where $K(b,b'')=\int \dd a\,e^{-a/2}\,f(a)(b-e^{-a}b'')\,(1+b'')^{-s/2}$. Thus, computing $\int_\gR \dd b \,K(b,b)$, we get
\begin{align}
\label{dim=1}
\Tr \,(\rho \rtimes U)(f)(\bbbone+D^2)^{-s/2}
=\int_{\gR^2}\,\dd b\,\dd a \,e^{-a/2}\,f(a)(b(1-e^{-a}))\,(1+b^2)^{-s/2}.
\end{align}
Since in \eqref{dim=1}, $f$ has a compact support in both variables $a$ and $b$, the above integral always converges except eventually around $a=0$: 
assume that $a\in \gR \to f(a)\in C_c(\gR)$ is locally constant in $[-\epsilon, \epsilon]$, then 
\begin{align*}
\int_\gR \dd b\, (1+b^2)^{-s/2}\int_{-\epsilon}^\epsilon \dd a \,e^{-a/2}\,\abs{f(a)(b(1-e^{-a}))} 
& = 4\sinh(\tfrac \epsilon 2)\,\abs{f(0)(0)} \int_\gR \dd b\, (1+b^2)^{-s/2}\\
&=4 \sinh(\tfrac \epsilon 2)\,\abs{f(0)(0)} \,\tfrac{\sqrt{\pi}\sqrt{\Gamma ((s-1)/2)}}{\Gamma(s/2)}\text{ for $s>1$}. 
\end{align*}
Thus the trace is finite for $s>1$ and the dimension \eqref{dimension} of the twisted spectral triple remains the same than those of the original triple which is one.
\end{remark}

\subsection{\texorpdfstring{The group $\gZ \ltimes \gR$}{The group ZxR}}
\label{sec-example-discrete-affine-group}

This example is a variation of the previous one, in the sense that the group $G$ which acts on $\gR$ is taken to be $G = \gZ$, and the action is the restriction of the previous action $\varphi$ to the subgroup $\gZ \subset \gR$. Many of the features of the previous example are present in this simpler example, for instance the increase in dimension.

The action of $n \in \gZ$ on $b \in\gR$ is $\varsigma_n(b) \vc e^{-n} b$, the homomorphism of groups defined in \eqref{eq-definition-nu} is $\nu(n) \vc e^{-n}$, it coincides with $q(n)^2$, and the action of $\gZ$ on $C^*(\gR)$ is given on any $f_\gR \in C_c(\gR)$ as $\alpha_n(f_\gR)(b) = e^n f(e^n b)$. The unitary operator \eqref{eq-Ur} is $(U_n\, \xi)(b) = e^{n/2} \xi(e^n b)$ for any $\xi \in C_c(\gR) \subset H = L^2(\gR, \dd b)$. 

Consider, as before, the Dirac operator $(D \,\xi)(b) \vc b\, \xi(b)$ for any $\xi \in C_c(\gR)$. Then $U_n D U_n^* \,\xi = e^n D \xi$, so that $z(n) \vc e^{-n/2}\,\bbbone$, $p(n) = e^{-n} \,\bbbone= q(n)^2$, and $\vartheta(n) = e^{-n/2}$.

As before, Hypothesis~\ref{hypo-weight} and \ref{hyp-dirac} are satisfied. The modular-type twisted spectral triple of Theorem~\ref{thm-twisted-triple} is then given by the following operators:
\begin{equation*}
(\hD \,\hxi\,)(n)(b) = e^{-n} b\; \hxi(n)(b),
\quad\text{and}\quad
(\caT_{\cone,\,\ct} \,\hxi\,)(n)(b) = (\cone + \ct \,e^{-n}) \,\hxi(n)(b).
\end{equation*}

We check now that this twisted spectral triple is also $2$-summable. As in the proof of Lemma~\ref{lem-trace-affine}, using \eqref{eq-tracecomputationagain},
\begin{align*}
\Tr\, \Theta\,\pi(f)\,(1+\DD^2)^{-s/2}\pi(g) &= 2\sum_{n \in \gZ} e^{-n} \Tr_H (g\star \beta(f))(0) [1+ D^2+(\cone+\ct \,e^{-n})^2]^{-s/2} \\
&= 2\, \hvarphi_\gR(g\star \beta(f)) \sum_{n \in \gZ} \int_{\gR} \dd b \,  e^{-n} [1+b^2+(\cone+\ct\,e^{-n})^2 ]^{-s/2}
\\
&= 2 \,\hvarphi_\gR(g\star \beta(f))\,\tfrac{\sqrt{\pi}\,\Gamma((s-1)/2)}{\Gamma(s/2)} \,\sum_{n \in \gZ} \,e^{-n}[1 + (\cone + \ct \,e^{-n})^2]^{(1 - s)/2}.
\end{align*}
Let $a_n(s) \vc e^{-n}[1 + (\cone + \ct \,e^{-n})^2]^{(1 - s)/2}$. For $n \to +\infty$, $a_n(s) \sim e^{-n} (1 + \cone^2)^{(1-s)/2}$, so that $\sum_{n \geq 0} a_n(s) < \infty$ for any $s \in \gR$. For $n \to -\infty$, $a_n(s) \sim e^{\abs{n}} \, \ct^{1-s} \,  e^{\abs{n} (1-s)} = \ct^{1-s} e^{\abs{n}(2-s)}$, so that $\sum_{n \leq 1} a_n(s) < \infty$ iff $s > 2$.

This example shows that the spectral dimension of a original spectral triple can increase under the action of a single automorphism (via $\gZ$), so our extension procedure can be applied to dynamical systems.

\subsection{The conformal group of a Riemannian manifold}
\label{conformal}

Let us consider the situation of Case~\ref{Case-Conformal}: let $(M,g)$ be a $n$-dimensional smooth and complete Riemannian spin manifold where $g$ is a given metric on $M$ and let $[g]$ be the class of metrics conformally equivalent to $g$. We do not assume that $M$ is compact for the moment. \\
Let $\text{Conf}(M,[g])$ be the conformal group of smooth diffeomorphisms of $M$ which preserve the class $[g]$ and the orientation. When $n\geq 3$, $\text{Conf}(M,[g])$ is a Lie group for the compact-open topology which is of dimension less or equal to $(n+1)(n+2)/2$ by theorems of Montgomery and Kobayashi: it is the automorphism group of the conformal structure of $M$ which is a $G$-structure of finite type \cite[I.~Theorem 5.1 and IV.~Theorem 6.1]{Koba}. 

 A subgroup $K$ of $\text{Conf}(M,[g])$ is said to be essential if there exists no $h\in C^\infty(M)$ such that $K\subset \text{Iso}(M,e^{2h}g)$ where $\text{Iso}(M,g)$ is the group of isometries of $M$ for the metric $g$. By a theorem of Ferrand--Obata  \cite{Ferrand}, for $n\geq 2$, if $M$ is not conformally equivalent with $\gS^n$ (with standard metric when $M$ is compact) or $\gR^n$ (with Euclidean metric when $M$ is not compact), then $G$ acts properly on $M$ (implying that $G$ is compact when $M$ is compact) and it is inessential. When $M$ is not compact, the conformal group is inessential if it is equicontinuous. 
 
As original triple, we choose the commutative triple on $M$: $A\vc C_0^\infty(M)$ is the space of smooth functions on $M$ which vanish at infinity as well as all their derivatives, the Hilbert space is $H\vc L^2(M,\slashed{S}_g)$, and $D\vc\Dg$ is the usual Dirac operator associated to the metric $g$ and the given spin structure, which is selfadjoint since $M$ is complete. 

Let $G \vc SCO(M, [g])$ be the subgroup of $\text{Conf}(M,[g])$ which preserves the spin structure, \textsl{i.e.} $G$ lifts to the principal spin bundle. Thus $G$ is the isotropy subgroup of $\text{Conf}(M,[g])$ which fixes the class in $H^1(M, \gZ_2)$ defining the spin structure, so $G$ is a Lie group. 
The group $G$ acts on tensor structures on $M$ by pull-back. In particular, the action $\alpha$ of $\phi \in G$ on $f \in A$ is given by 
\begin{equation*}
\phi^* f \vc f \circ \phi.
\end{equation*}

The Hilbert space $L^2(M,\slashed{S}_g)$ decomposes as $L^2(M,\dd\vol_g) \otimes \slashed{S}_g$, where $\slashed{S}_g$ is the spinor bundle associated to $g$, and $L^2(M,\dd\vol_g)$ is the GNS representation space of the weight $\varphi_g$ defined on the $C^*$-algebra $\bA \vc C_0(M)$ by
\begin{equation*}
\varphi_g(f) \vc \int_M  f \,\dd\vol_g\, \text{ for any $f \in \bA^+$.}
\end{equation*}
The representation of $f \in \bA$ on $L^2(M,\dd\vol_g)$ is $\pi_\varphi(f)\, \xi \vc f\,\xi$, and it extends diagonally to $H$. Notice that $H$ depends on $g$.

By hypothesis, for any $\phi \in G$, there exists a smooth function $h_\phi$ on $M$ such that
\begin{equation}
\label{stretch factor}
\phi^* g = e^{-4 h_\phi} g
\end{equation}
where the $4$ is arbitrary but simplify \eqref{UdU*=}.
This implies that
\begin{equation*}
\phi^*\text{dvol}_g=\text{dvol}_{\phi^*g} = e^{-2 n h_\phi} \,\text{dvol}_g \quad \text{and }\quad e^{-4nh_\phi}=\text{det}^2(\phi')\,\tfrac{(\det\,g)\circ\phi}{\det\,g}\,.
\end{equation*}
From \eqref{stretch factor}, or from $\det\,(\phi_1\circ \phi_2)'=(\det \,\phi_2')\,(\det\,\phi_1')\circ\phi_2$, we deduce that the function $h_\phi$ satisfies the cocycle relation
\begin{equation}
\label{hcocycle}
h_{\phi_1 \circ \phi_2} = h_{\phi_2} + \alpha_{\phi_2}(h_{\phi_1}), \, \,\text{ for any }\phi_1, \phi_2 \in G.
\end{equation}
For $\phi \in G$ and $f \in \bA$, one has also
\begin{equation}
\label{eq-omega-conform-phi*}
\varphi_g(f) 
= \int_M f \, \dd\vol_g
= \int_M \phi^*(f \, \dd\vol_g)
= \int_M (\phi^*f) \, e^{-2 n h_\phi} \, \dd\vol_g.
\end{equation}

Denote by $\tphi$ the (well-defined) lift of $\phi$ to $\slashed{S}_g)$. Then there is a natural action of $\phi$ on $H$ defined by $\xi \mapsto \tphi^{-1} \circ \xi \circ \phi$. 

Some of the relations presented above look similar to relations in the text. In order to recover these relations, one needs to use the opposite group $G^\op$ of $G$.  $G^\op$ has the same space as $G$, and we denote by $r \in G^\op \mapsto \phi_r \in G$ the (smooth) one-to-one correspondence between $G^\op$ and $\phi \in G$ as spaces. Then the product on $G^\op$ is defined such that $\phi_{rr'}=\phi_{r'}\,\phi_r$. 

Let $\alpha$ be the action of $G^\op$ on $\bA$ defined by
\begin{equation*}
\alpha_r(f) \vc \phi_r^* f = f \circ \phi_r\,.
\end{equation*}
Then one can check that $\alpha_{r} \alpha_{r'} = \alpha_{rr'}$.

Since $(G, M)$ is a Lie transformation group, it defines a $C^*$-dynamical system $(\bA, G, \alpha)$ with $\bA = C_0(M)$. Then one has  $M(\bA)\simeq C_b(M)$ and $Z(A)\simeq C_b^\infty(M)$.\\
When $M$ is compact, $\bA=C(M)$, $M(\bA)=\bA$ and $Z(A)=A= C^\infty(M)$.

The dense subspace $$\algA \vc C_c^\infty(G\times M)\text{ of }C_c(G, A)\subset \bB$$ is a $*$-algebra for the product
\begin{equation*}
(\hf \star_\alpha \hg)(r, x) = \int_G \dd\mu_G(r') \, \hf(r', x) \, \hg(r'^{-1} r, \phi_{r'}(x)), \text{ for any $(r,x) \in G \times M$},
\end{equation*}
and the involution $\hf^*(r,x) = \overline{\hf(r^{-1}, \phi_r(x))}$, for  any $\hf, \hg \in \algA$.

The relation \eqref{eq-omega-conform-phi*} can be written as
\begin{equation*}
\varphi_g(\alpha_{r^{-1}}[f]) = \varphi_g( f\, e^{- 2n h_{\phi_r}}) = \varphi_g(q(r)^2 \, f),
\end{equation*}
which is Hypothesis~\ref{hypo-weight}-\ref{hyp-q} with $q(r) \vc e^{- n h_{\phi_r}}$. As a consequence of \eqref{hcocycle}, $q$ is a $\alpha$-one-cocycle in $Z^1(G, C_+^\infty(M))$, where $C_+^\infty(M)$ are strictly positive functions in $C^\infty(M)$ (which is a multiplicative group).

The action $\alpha$ is extended to an action $\alpha^e$ on $H$ as $\alpha^e_r(\xi) = \tphi_r^{-1} \circ \xi \circ \phi_r$ for any $\xi \in H$, with the property that $\alpha^e_r(f \xi) = \alpha_r(f)\, \alpha^e_r(\xi)$. Then the unitary representation of Proposition~\ref{prop-def-unitary} is given by
\begin{equation*}
U_r (\xi) \vc q(r) \,\alpha^e_r(\xi) = e^{-n h_{\phi_r}} \, \tphi_r^{-1} \circ \xi \circ \phi_r\,.
\end{equation*}
This is the unitary operator defined in \cite[eq.~(3.3)]{Moscovici2010}. In fact, since $G$ preserves the spin structure, this unitary links the two Hilbert spaces of spinors associated to two different metrics, and it is shown in \cite{BG1992} (see \cite[Lemma~3.1]{Moscovici2010}) that 
\begin{equation}
\label{UdU*=}
U_r\, \Dg \, U_r^* = {z(r)^*}^{-1}\, \Dg\, z(r)^{-1}, \quad \text{with $z(r) \vc e^{-h_{\phi_r}}\,\bbbone$}.
\end{equation}
This defines the one-cocycle $p \in Z^1(G, C_+^\infty(M))$ as
\begin{equation*}
p(r) \vc e^{-2 h_{\phi_r}}.
\end{equation*}

\begin{lemma}
\label{lem-smoothhphi}
The function $(r,x) \in G^\op \times M \mapsto h_{\phi_r}(x)$ is smooth. 
\end{lemma}

\begin{proof}
In a local chart of $M$ with coordinates $(x^\mu)$, \eqref{stretch factor} gives
\begin{equation*}
 e^{-4 h_{\phi_r}(x)} = g^{\mu\nu}[x] \, \partial_\mu \phi_r^\rho(x)\, \partial_\nu \phi_r^\sigma(x)\, g_{\rho \sigma}[\phi_r(x)].
\end{equation*}
Since $(r,x) \in G^\op \times M \mapsto \phi_r(x) \in M$ is smooth, the function $(r, x) \in G^\op \times M \mapsto h_{\phi_r}(x)$ is smooth. 
\end{proof}
Hypothesis \ref{hyp-dirac}-\ref{preservdom} holds true with the space $Y_\algA \vc C_c^\infty(G\times M,\slashed{S}_g) \subset \hY$ which is dense in $\hH$. Remark that, since $(\,\widehat{\,\Dg} \,\hxi\,)(r)=e^{-h_{\phi_r}}\, \Dg \,e^{-h_{\phi_r}} \hxi(r)$ and $(\htheta\,\hxi\,)(r)=\,e^{-2h_{\phi_r}\circ \phi_{r^{-1}}}\,\hxi(r)$, $\widehat{\,\Dg}$ does not commute with $\htheta$ or $\caT_{\cone,\,\ct}$.

The twist $\beta$ defined by 
\begin{equation*}
\beta(\hf\,)(r,x) \vc e^{-2 h_{\phi_r}(x)}\,\hf(r,x)
\end{equation*}
maps $\algA$ into itself and is an automorphism of $\algA$ as proved in Proposition \ref{prop-beta}.

When $M$ is non compact, the function $e^{-2 h_{\phi_r}}$ need not be bounded on $M$, and so does not belongs to $M(\bA) =  C_b(M)$. When $M$ is compact, this function is always bounded and then $p(r)$ and $q(r)$ belong to $C_+^\infty(M) \subset M(\bA) = \bA = C(M)$.

\begin{corollary}
\label{lem-continuouscocycles}
When $M$ is compact, the maps $p, q, q^{-1}: G^\op \to M(\bA)$ are continuous.
\end{corollary}

\begin{proof}
Applying Lemma~\ref{lem-smoothhphi}, each of functions $(r,x) \mapsto p(r)(x)$, $(r,x) \mapsto q(r)(x)$ and $(r,x) \mapsto q^{-1}(r)(x) = q(r^{-1})(\phi_r(x))$ are smooth on $G^\op \times M$.

Since $M$ is compact, the topology of $M(\bA) = \bA = C(M)$ is the sup-norm topology, so, to show that $p$ is continuous, we have to show that $\norm{p(r') - p(r)}_\infty \to 0$ when $r' \to r$ in $G^\op$. If this is not the case, then there exists $\epsilon > 0$ and sequences $r_n \in G^\op$ and $x_n \in M$ such that $r_n \to r$ and $\abs{p(r_n)(x_n) - p(r)(x_n)} > \epsilon$. But, by the compactness of $M$, we can suppose that $x_n \to x \in M$, and then, by continuity of $(r,x) \mapsto p(r)(x)$, one gets $\lim_{n \to \infty}\, \abs{p(r_n)(x_n) - p(r)(x_n)} = 0$ and a contradiction. \\The same argument applies to $q$ and $q^{-1}$.
\end{proof}

Locally on $M$, one can write $\Dg = - i \gamma^\mu \partial_\mu$. Then, for any $r \in G^\op$ and $\hf \in \algA$, one has
\begin{equation*}
[\  \Dg, \hf(r) \,z(r) ]\, z(r)^{-1}= [\ \Dg, \hf(r) \,e^{-h_{\phi_r}}] \,e^{h_{\phi_r}} = i \gamma^\mu [ \partial_\mu (\hf(r))+ \hf(r)\,\partial_\mu(h_{\phi_r}) ].
\end{equation*}
By Lemma~\ref{lem-smoothhphi}, the map $r\in G^\op \mapsto \hf(r) \,e^{-2h_{\phi_r}} \in A$ belongs to $\algA$, which proves \eqref{eq-hyp-p(r)}, and \eqref{eq-hyp-comm-bounded} is trivially satisfied for $M_{\hf, h} = \sup_{r \in \Supp(\hf)} \sup_{x \in M} \abs*{\partial_\mu (\hf(r))(x) + \hf(r)(x)\, \partial_\mu(h_{\phi_r})(x)}$. Moreover, here $[\Dg, p(r)] = 2i \gamma^\mu (\partial_\mu h_{\phi_r} )e^{-2h_{\phi_r}}$.

Suppose now $M$ is compact and $G$ is the spin structure preserving conformal group of $(M, g)$. Then, $G$ being compact, the norm of $[\Dg, p(r)]$ is uniformly bounded, so that $[\hD, \caT_{\cone,\,\ct}]$ is bounded, and the integral in \eqref{eq-compact} is finite, so \eqref{modularcompactresol} holds true even for $\epsilon=0$. Moreover, Hypotheses~\ref{hyp-dirac} (or  \ref{hypo-weight}) are fulfilled and one can use Theorem \ref{thm-twisted-triple} (or Corollary \ref{cor-twisted-triple-weight}). 

\bigskip

The fact that the group $G$ is essential or not is directly related to the cocycle property \eqref{hcocycle} of $h: \,\phi \in G \mapsto h_\phi \in C^\infty(M)$:

\begin{proposition}
\label{essential}
$G$ is inessential if and only if the class of $q$ (resp. $p$) in $H^1(G, Z(A)^\times)$ is trivial.
\end{proposition}
 
\begin{proof}
The proof is an adaptation of \cite[Proposition 2.2]{Bana} in our context.

If $G$ is inessential, there is a metric $g_0 = e^{\lambda} g \in [g]$ such that $\phi^\ast g_0 = g_0$ for any $\phi \in G$, with $\lambda \in C^\infty(M)$. Then, for any $r \in G^\op$, $\phi_r^\ast g = \phi_r^\ast( e^{-\lambda} g_0) = \alpha_r(e^{-\lambda}) g_0 = \alpha_r(e^{-\lambda}) e^{\lambda} g$, so that $e^{-4 h_{\phi_r}} = e^{\lambda} \alpha_r(e^{-\lambda})$. This shows that $q(r) = e^{n\lambda/4} \alpha_r(e^{-n\lambda/4})$ and $p(r) = e^{\lambda/2} \alpha_r(e^{-\lambda/2})$, which means that they are both coboundaries in $Z^1(G, C_+^\infty(M))$.

From the computation above, it is obvious that if $p$ is a coboundary if and only if the $q$ is a coboundary. So, let assume that $p$ is a coboundary in $Z^1(G, C_+^\infty(M))$: $p(r) = e^{\lambda/2} \alpha_r(e^{-\lambda/2})$ for some $e^{\lambda/2} \in C_+^\infty(M)$. This parametrization gives $e^{-4 h_{\phi_r}} = e^{\lambda} \alpha_r(e^{-\lambda})$ as before. Then, defining $g_0 \vc e^{\lambda} g \in [g]$, the previous computation shows that $\phi^\ast g_0 = g_0$ for any $\phi \in G$, so that $G$ is  inessential.
\end{proof}
Notice that the metric $g_0$ in the previous proof gives rise to the Dirac operator $\bar{D}$ of Proposition~\ref{Dbar}.

\bigskip

We now give an example showing that in {\it Case \ref{Case-Conformal} the spectral dimension can change as in the affine case of Section  \ref{sec-example-affine-group}} for the subgroup of dilations.
\\
Consider the Euclidean space $M = \gR^n$ with $n \geq 3$ acted upon by the conformal group $G=SO(n+1, 1)$ with the simpler case of the action of the dilatation subgroup $(\gR, +)$ of $G$, which acts on $(x^\mu) \in \gR^n$ by $(x^\mu) \mapsto (e^{-a} x^\mu)$ for any $a \in \gR$. 
Choose the original spectral triple $(A, H, D)$ with $A \vc {\caS}(\gR^n)$ the space of Schwartz functions on $\gR^n$, $H\vc L^2(\gR^n,\slashed{S})$ the Hilbert space of $L^2$-spinors and $D \vc \slashed{\partial} = i \gamma^\mu \partial_\mu$. Then, applying previous machinery, one gets $z(a) = e^{-a/2} \,\bbbone$ and $p(a) = e^{-a}\, \bbbone$, so that
\begin{align*}
(\hD\, \hxi\,)(a) = e^{-a} \,\slash\hspace{-0.23cm} \partial \hxi(a),\qquad
(\caT_{\cone,\,\ct}\, \hxi\,)(a) &= (\cone + \ct e^{-a})\, \hxi(a).
\end{align*}
Following similar arguments as in Section \ref{sec-extended-affine}, the dimension of the extended spectral triple can be computed by evaluating the trace of 
\begin{align*}
\int_{-\infty}^\infty \dd a\,e^{-a} \Tr_H X(s),\quad X(s) \vc f(0) [\bbbone+ \slash\hspace{-0.23cm} \partial^2+(\cone+\ct \,e^{-a})^2]^{-s/2}
\end{align*}
as a function of $s$, where $f(0) \in A$. Let $(b^\mu) \in \hgR^n$ be the Fourier transforms of variables $(x^\mu) \in \gR^n$ and denote by $f(a)\,\hat{}$ the Fourier transform of $f(a) \in A$ for $a \in \gR$. 
Then $X(s) $ acts on $\xi\,\hat{}\in H\,\hat{}\, \vc L^2(\hgR^n,\dd^n b)$ as
\begin{align*}
(X(s) \,\xi\,\hat{}\,)(b)&=\int_{\hgR^n} \dd^n b' \,f(0)\,\hat{}\,(b')\,[1+(b-b')^2+ (\cone+\ct\,e^{-a})^2]^{-s/2}\,\xi\,\hat{}\,(b-b')\\
&= \int_{\hgR^n} \dd^n b' \,f(0)\,\hat{}\,(b-b')\,[1+b'^2+ (\cone+\ct\,e^{-a})^2]^{-s/2}\,\xi\,\hat{}\,(b'),
\end{align*}
where $b^2 \vc \sum_{\mu} (b^\mu)^2$. Thus $\Tr_H X(s)=f(0)\,\hat{}\,(0)\,\int_{\hgR^n} \dd^n b \,[1+b^2+ (\cone+\ct\,e^{-a})^2]^{-s/2}$. Assume that $\cone, \ct \in \gR$ and $\ct \neq 0$. For $s > n+1$, using
\begin{equation*}
\int_{\hgR^n} \dd^n b \, [1+b^2+ (\cone+\ct\,e^{-a})^2]^{-s/2} = \tfrac{{\pi}^{n/2} \,\Gamma[(s-n)/2]}{\Gamma(s/2)} \, [1 +  (\cone+\ct\,e^{-a})^2 ]^{(n-s)/2}
\end{equation*}
and by a similar argument as in the proof of Lemma~\ref{lem-trace-affine}, one gets
\begin{align*}
\Tr\, \Theta\,\pi(f)\,(\bbbone+\DD^2)^{-s/2}
&= 2 f(0)\,\hat{}\,(0) \,  \tfrac{{\pi}^{n/2}\,\Gamma[(s-n)/2]}{\Gamma(s/2)} \,\int_{-\infty}^\infty \dd a\,e^{-a} [1 +  (\cone+\ct\,e^{-a})^2 ]^{(n-s)/2}\\
&= 2 f(0)\,\hat{}\,(0) \, \tfrac{{\pi}^{n/2}}{\Gamma(s/2)}  \, 
\big[ 
\tfrac{\sqrt{\pi}\, \Gamma[(s-n+1)/2]}{\vert \ct\vert\, (s-n-1)}
-
\tfrac{\cone\, \Gamma[(s-n)/2]}{\vert \ct \vert}   \,{}_2F_1(\tfrac{1}{2}, \tfrac{s-n}{2},\tfrac{3}{2},-\cone^2)
\big].
\end{align*}
So the extended (modular) spectral triple has dimension $n+1$ and
\begin{equation*}
\lim_{s\to n+1} (s-n-1) \,\Tr\, \Theta\,\pi(f)\,(\bbbone+\DD^2)^{-s/2} = \tfrac{2 \,\pi^{(n+1)/2}}{\vert \ct\vert\, \Gamma[(n+1)/2]} \,f(0)\,\hat{}\,(0).
\end{equation*}

\begin{remark}
\label{pb du noyau 2}
The kernel of $\vartheta$ includes the subgroups of translations and rotations. Only the subgroup of homotheties is sensitive to $\vartheta$.
\end{remark}

The case where $U_r D U_r^* -D$ is bounded (see Remark~\ref{rmk-udu-d-bounded}) has been considered by A.~Connes in \cite[p.~346]{Conn97a}.

\section{Perspectives}

One of questions left open is a better control of this notion of summability for the extended spectral triple. This seems linked to the modular operator and a good notion of non-unital modular spectral triples has to be settled down before.

In the case of the $C^*$-algebra of a semidirect group, it is probably possible to do the same construction with a twisted convolution product instead of the product \eqref{eq-product-alpha} where the twist is given by a 2-cocycle valued in the circle, see \cite[II.10.7.4]{Blac06a}.

We expect that most of our construction could be transferred to (compact) quantum groups. See also \cite{NT} for a different approach.

\section*{Acknowledgments}

We thank Gruia Arsu, Christian Duval, Jens Kaad, Michael Puschnigg, Adam Rennie, Fedor Sukochev, Antony Wassermann and Dmitriy Zanin for fruitful discussions or correspondences. We also thank the referee for his helpful comments.


\bibliography{extension-of-triples-biblio}

\begin{thebibliography}{45}
\providecommand{\natexlab}[1]{#1}
\providecommand{\url}[1]{\texttt{#1}}
\expandafter\ifx\csname urlstyle\endcsname\relax
  \providecommand{\doi}[1]{doi: #1}\else
  \providecommand{\doi}{doi: \begingroup \urlstyle{rm}\Url}\fi

\bibitem[Aliprantis and Border(2006)]{Ali}
C.~Aliprantis and K.~Border.
\newblock \emph{Infinite Dimensional Analysis - A Hitchhiker's Guide}.
\newblock Springer, 3rd edition, 2006.

\bibitem[Arsu(2008)]{Arsu08a}
G.~Arsu.
\newblock On {S}chatten--von {N}eumann class properties of pseudo-differential
  operators. the {C}ordes--{K}ato method.
\newblock \emph{J. Operator Theory}, 59:\penalty0 81--114, 2008.

\bibitem[Banyaga(2000)]{Bana}
A.~Banyaga.
\newblock {On the essential conformal groups and a conformal invariant}.
\newblock \emph{J. geom.}, 68:\penalty0 10--15, 2000.

\bibitem[Bellissard et~al.(2010)Bellissard, Marcolli, and Reihani]{BMR}
J.~Bellissard, M.~Marcolli, and K.~Reihani.
\newblock Dynamical systems on spectral metric spaces.
\newblock arXiv:1008.4617, 2010.

\bibitem[Blackadar(2006)]{Blac06a}
B.~Blackadar.
\newblock \emph{Operator Algebras, Theory of {$C^\ast$}-Algebras and von
  {N}eumann Algebras}, volume 122 of \emph{Encyclopaedia of Mathematical
  Sciences}.
\newblock Springer-Verlag, 2006.

\bibitem[Bourguignon and Goduchon(1992)]{BG1992}
J-P. Bourguignon and P.~Goduchon.
\newblock {Spineurs, op{\'e}rateurs de Dirac et variations}.
\newblock \emph{Commun. Math. Phys.}, 144:\penalty0 581--599, 1992.

\bibitem[Carey et~al.(2009)Carey, Rennie, and Tong]{CRT}
A.~Carey, A.~Rennie, and K.~Tong.
\newblock {Spectral flow invariants and twisted cyclic theory for the Haar
  state of $SU_q(2)$}.
\newblock \emph{J. Geom. Phys.}, 59:\penalty0 1431--1452, 2009.

\bibitem[Carey et~al.(2010)Carey, Phillips, and Rennie]{CPR}
A.~Carey, J.~Phillips, and A.~Rennie.
\newblock {Twisted cyclic theory and an index theory for the gauge invariant
  KMS state on the Cuntz algebra $O_n$}.
\newblock \emph{J. K-theory}, 6:\penalty0 339--380, 2010.

\bibitem[Carey et~al.(2011)Carey, Neshveyev, Nest, and Rennie]{CNNR}
A.~Carey, S.~Neshveyev, R.~Nest, and A.~Rennie.
\newblock {Twisted cyclic theory, equivariant $KK$-theory and KMS states}.
\newblock \emph{J. reine angew. Math.}, 650:\penalty0 161--191, 2011.

\bibitem[Carey et~al.(2012)Carey, Gayral, Rennie, and Sukochev]{CGRS}
A.~Carey, V.~Gayral, A.~Rennie, and F.~Sukochev.
\newblock Integration on locally compact noncommutative spaces.
\newblock \emph{J. Funct. Anal.}, 263:\penalty0 383--414, 2012.

\bibitem[Combes(1968)]{Combes}
F.~Combes.
\newblock {Poids sur une C$^*$-alg{\`e}bre}.
\newblock \emph{J. Math. pures et appl.}, 47:\penalty0 57--100, 1968.

\bibitem[Connes(1989)]{Connes89}
A.~Connes.
\newblock Compact metric spaces, fredholm modules and hyperfiniteness.
\newblock \emph{Ergod. Th. \& Dynam. Sys.}, 9\penalty0 (207--220), 1989.

\bibitem[Connes(1997)]{Conn97a}
A.~Connes.
\newblock Noncommutative differential geometry and the structure of space time.
\newblock In G.~'t~Hooft, A.~Jaffe, G.~Mack, P.~Mitter, and R.~Stora, editors,
  \emph{Quantum Fields and Quantum Space Time}, volume 364 of \emph{NATO ASI
  Series}, pages 45--72. Springer US, 1997.

\bibitem[Connes and Moscovici(2008)]{ConnesMoscovici2006}
A.~Connes and H.~Moscovici.
\newblock Type {III} and spectral triples.
\newblock In \emph{Traces in number theory, geometry and quantum fields},
  Aspects Math., E38, pages 57--71. Friedr. Vieweg, Wiesbaden, 2008.

\bibitem[Deift and Simon(1976)]{DS}
P.~Deift and B.~Simon.
\newblock On the decoupling of finite singularities from the question of
  asymptotic completeness in two body quantum systems.
\newblock \emph{J. Funct. Anal.}, 23:\penalty0 218--238, 1976.

\bibitem[Duflo and Moore(1976)]{DuflMoor76a}
M.~Duflo and C.~Moore.
\newblock On the regular representation of a nonunimodular locally compact
  group.
\newblock \emph{J. Funct. Anal.}, 21\penalty0 (2):\penalty0 209--243, 1976.

\bibitem[Fathizadeh and Gabriel(2015)]{Farzad-Gabriel}
F.~Fathizadeh and 0.~Gabriel.
\newblock {On the Chern-Gauss-Bonnet theorem and conformally twisted spectral
  triples for $C^*$-dynamical systems}.
\newblock arXiv: 1506.07913v1 [math.OA], 2015.

\bibitem[Fathizadeh and Khalkhali(2010)]{FaKha2010}
F.~Fathizadeh and M.~Khalkhali.
\newblock The algebra of formal twisted pseudodifferential symbols and a
  noncommutative residue.
\newblock \emph{Lett. Math. Phys.}, 94\penalty0 (1):\penalty0 41--61, 2010.

\bibitem[Fathizadeh et~al.(2013)Fathizadeh, Khalkhali, Nicola, and
  Rodino]{FaKhaNiRo2013}
F.~Fathizadeh, M.~Khalkhali, F.~Nicola, and L.~Rodino.
\newblock {A two dimensional Adler--Manin trace and bi-singular operators}.
\newblock \emph{Commun. Math. Anal.}, 15:\penalty0 61--78, 2013.

\bibitem[Ferrand(1996)]{Ferrand}
J.~Ferrand.
\newblock {The action of conformal transformations on a Riemanniaan manifold}.
\newblock \emph{Math. Ann.}, 304:\penalty0 277--291, 1996.

\bibitem[Forsyth et~al.(2014)Forsyth, Mesland, and Rennie]{FMR14}
I.~Forsyth, B.~Mesland, and A.~Rennie.
\newblock Dense domains, symmetric operators and spectral triples.
\newblock \emph{New York J. Math.}, 20:\penalty0 1001--1020, 2014.

\bibitem[Gabriel and Grensing(2013)]{GabrielGrensing}
O~Gabriel and M.~Grensing.
\newblock Spectral triples and generalized crossed products.
\newblock arXiv:1310.5993v1 [math.OA], 2013.

\bibitem[Hawkins et~al.(2013)Hawkins, Skalski, White, and Zacharias]{HSWZ}
A.~Hawkins, A.~Skalski, S.~White, and J.~Zacharias.
\newblock On spectral triples on crossed products arising from equicontinuous
  action.
\newblock \emph{Math. Scand.}, 113:\penalty0 262--291, 2013.

\bibitem[Iochum et~al.(2012)Iochum, Masson, and Sitarz]{Ioc12}
B.~Iochum, T.~Masson, and A.~Sitarz.
\newblock {$\kappa$-deformation, affine group and spectral triples.}
\newblock {W. Pusz and P. So{\l}tan (eds.), Operator algebras and quantum
  groups. Banach Center Publications 98, 261--291}, 2012.

\bibitem[Kaad(2011)]{Kaa11}
J.~Kaad.
\newblock On modular semifinite index theory.
\newblock arXiv:1111.6546, 2011.

\bibitem[Kaad and Senior(2012)]{Ks}
J.~Kaad and R.~Senior.
\newblock {A twisted triple for quantum $SU(2)$}.
\newblock \emph{J. Geom. Phys.}, 62:\penalty0 731--739, 2012.

\bibitem[Kato(1995)]{Kato}
T.~Kato.
\newblock \emph{{Perturbation for Linear Operators}}.
\newblock Springer-Verlag, 1995.

\bibitem[Kobayashi(1995)]{Koba}
S.~Kobayashi.
\newblock \emph{{Transformation Groups in Differential Geometry}}.
\newblock Springer-Verlag, 1995.

\bibitem[Kustermans and Vaes(1999)]{KV1999}
J.~Kustermans and S.~Vaes.
\newblock {Weight theory for $C^*$-algebraic quantum groups}.
\newblock \\arXiv:math/9901063, 1999.

\bibitem[Lang(1993)]{Lang93a}
S.~Lang.
\newblock \emph{Real and Functional Analysis}.
\newblock Springer-Verlag, 3rd edition, 1993.

\bibitem[Matassa(2013)]{Mata13}
M.~Matassa.
\newblock {Non-commutative integration, zeta functions and the Haar state for
  $SU_q(2)$}.
\newblock arXiv:1310.7477 [math-ph], 2013.

\bibitem[Matassa(2014{\natexlab{a}})]{Mata14a}
M.~Matassa.
\newblock A modular spectral triple for $\kappa$-{M}inkowski space.
\newblock \emph{J. Geom. Phys.}, 76:\penalty0 136--157, 2014{\natexlab{a}}.

\bibitem[Matassa(2014{\natexlab{b}})]{Mata14b}
M.~Matassa.
\newblock Quantum dimension and quantum projective spaces.
\newblock \emph{SIGMA}, 10:\penalty0 097, 13 pages, 2014{\natexlab{b}}.

\bibitem[Moscovici(2010)]{Moscovici2010}
H.~Moscovici.
\newblock {Local index formula and twisted spectral triples.}
\newblock \emph{Quanta of maths}, Clay Math. Proc., 11:\penalty0 465--500,
  2010.

\bibitem[Neshveyev and Tuset(2014)]{NT}
S.~Neshveyev and L.~Tuset.
\newblock {Deformation of $C^*$-algebras by cocycles on locally compact quantum
  groups}.
\newblock \emph{Advances in Math.}, 254:\penalty0 454--496, 2014.

\bibitem[Paterson(2014)]{Pat}
A.~Paterson.
\newblock Contractive spectral triples for crossed products.
\newblock \emph{Math. Scand.}, 2014.

\bibitem[Pedersen(1979)]{pedersen1979c}
G.~K. Pedersen.
\newblock \emph{{$C^*$-algebras and their automorphism groups}}.
\newblock Academic press London, 1979.

\bibitem[Quaegebeur and Verding(1999)]{QV1999}
J.~Quaegebeur and J.~Verding.
\newblock {A construction for weights on $C^*$-algebras. Dual weights for
  $C^*$-crossed products}.
\newblock \emph{Int. J. Math.}, 10:\penalty0 129--157, 1999.

\bibitem[Rennie and Senior(2014)]{RS14}
A.~Rennie and R.~Senior.
\newblock {The resolvent cocycle in twisted cyclic cohomology and a local index
  formula for the Podle\'s sphere}.
\newblock \emph{J. Noncommut. Geom.}, 8:\penalty0 1--43, 2014.

\bibitem[Rennie et~al.(2015)Rennie, Robertson, and Sims]{RennieRobertsonSims}
A.~Rennie, D.~Robertson, and A.~Sims.
\newblock The extension class and {KMS} states for {Cuntz--Pimsner} algebras of
  some bi-hilbertian bimodules.
\newblock arXiv:1501.05363, 2015.

\bibitem[Rudin(1991)]{Rudin}
W.~Rudin.
\newblock \emph{Functional Analysis}.
\newblock McGraw-Hill, Inc., 2nd edition, 1991.

\bibitem[Schweitzer(1993{\natexlab{a}})]{Schweitzer}
L.~Schweitzer.
\newblock {Dense $m$-convex Fr\'echet subalgebras of operator algebra crossed
  products by Lie groups}.
\newblock \emph{Int. J. Math.}, 04:\penalty0 601--673, 1993{\natexlab{a}}.

\bibitem[Schweitzer(1993{\natexlab{b}})]{Schweitzer1}
L.~Schweitzer.
\newblock Summary of spectral invariance results.
\newblock \emph{Math. Rep. Acad. Sci. Canada}, 15\penalty0 (1):\penalty0
  13--18, 1993{\natexlab{b}}.

\bibitem[Takesaki(2003)]{Takes2003}
M.~Takesaki.
\newblock \emph{Theory of operator algebras II}.
\newblock Springer Verlag, 2003.

\bibitem[Williams(2007)]{Will07a}
D.~Williams.
\newblock \emph{Crossed Products of $C^\ast$-Algebras}, volume 134 of
  \emph{Math. Surveys and Monographs}.
\newblock AMS, 2007.

\end{thebibliography}

\end{document}